\documentclass[10pt,reqno]{amsart}
\usepackage[utf8]{inputenc}
\usepackage[margin=1in]{geometry}
\usepackage[hidelinks]{hyperref}
\usepackage{scalerel,stmaryrd} 
\numberwithin{equation}{section}
\usepackage[normalem]{ulem}
\usepackage{aligned-overset}
\usepackage{amsfonts,amssymb,amsmath,amsthm,tikz,comment,mathtools,setspace,stmaryrd,enumitem,bbm}

\global\long\def\Rayleigh{\operatorname{Rayleigh}}%

\newcommand{\N}{\mathbb{N}}
\newcommand{\R}{\mathbb{R}}
\newcommand{\Q}{\mathbb{Q}}
\newcommand{\Z}{\mathbb{Z}}
\newcommand{\F}{\mathcal{F}}

\newcommand{\C}{\mathcal{C}}
\newcommand{\D}{\mathcal{D}}

\newcommand{\Nor}{\mathcal N}

\newcommand{\Ll}{\mathcal L}
\newcommand{\h}{\mathfrak h}
\newcommand{\wt}{\widetilde}

\newcommand{\deq}{\overset{\mathrm{d}}{=}}

\newcommand{\ve}{\varepsilon}
\newcommand{\eps}{\varepsilon}

\newcommand{\f}{\frac}

\newcommand{\mbf}{\mathbf}
\newcommand{\argmin}{\operatorname{argmin}}

\newcommand{\wh}{\widehat}
\newcommand{\Pp}{\mathbb P}

\newcommand{\Rup}{\R_{\uparrow}^4}
\newcommand{\Ff}{\mathcal F}
\newcommand{\W}{W}

\newcommand{\CFP}{\C_{\mathrm{FP}}}

\global\long\def\Law{\operatorname{Law}}
\global\long\def\dif{\mathrm{d}}
\global\long\def\Uniform{\operatorname{Uniform}}%
\global\long\def\st{\mathrel{:}}

\newcommand{\ind}{\mathbf 1}

\DeclareMathOperator*{\argmax}{arg\,max}

\global\long\def\BM{\operatorname{BM}}%

\global\long\def\BES{\operatorname{BES}}%

\global\long\def\Arcsine{\operatorname{Arcsine}}%

\global\long\def\Bm{\operatorname{Bm}}%

\newcommand{\be}{\begin{equation}}
\newcommand{\ee}{\end{equation}}

\newcommand{\e}{e}

\newcommand{\sig}{{\scaleobj{0.8}{\boxempty}}} 
\newcommand{\sigg}{{\scaleobj{0.9}{\boxempty}}}

\newtheorem{theorem}{Theorem}[section]
\newtheorem{proposition}[theorem]{Proposition}
\newtheorem{corollary}[theorem]{Corollary}
\newtheorem{lemma}[theorem]{Lemma}

\theoremstyle{definition}
\newtheorem{definition}[theorem]{Definition}

\theoremstyle{remark}

\RequirePackage{color}
\definecolor{darkgreen}{rgb}{0.0,0.5,0.0}
\definecolor{indigo}{rgb}{0.3,0,0.5}

\begin{document}
\title{Invariant measures and shocks in the KPZ fixed point}
\author{Alexander Dunlap}
\address{Department of Mathematics, Duke University, Durham, NC 27708, USA.}
\email{alexander.dunlap@duke.edu}
\urladdr{https://sites.math.duke.edu/~ajd91/}

\author{Evan Sorensen}
\address{Department of Mathematics,  Columbia University, New York, NY 10027, USA}
\email{evan.sorensen@columbia.edu}
\urladdr{https://sites.google.com/view/evan-sorensen}

\begin{abstract}
    We construct a family of invariant measures from the perspective of a shock in the KPZ fixed point. These measures are parameterized by a positive number $\theta > 0$, and are supported on functions $f$ satisfying $\lim_{|x| \to \infty} \frac{f(x)}{|x|} = 2\theta$. Each can be described as the sum of a Brownian motion and an independent Bessel-$3$ process with drift. We show that these measures appear as the $L \to \infty$ limit of the (conjectural) stationary measures for the conjectural open KPZ fixed point on $[0,L]$, after recentering by an appropriately defined shock location. Furthermore, we show that, with respect to the standard, deterministic recentering at $x = 0$, all extremal invariant measures for the KPZ fixed point are Brownian motions with drift. To do this, we first show that any extremal invariant measures must be supported on functions having fixed asymptotic slopes at $\pm \infty$. Using a one-force-one-solution principle from the work of Busani, Sepp\"al\"ainen, and  the second author, this rules out all other invariant measures except those having left slope $-2\theta$ and right slope $+2\theta$ for some $\theta > 0$. To handle this case, we derive the limiting fluctuations of the shock for a special choice of initial condition. Additionally, we derive the limiting fluctuations of the shock for the case of the invariant measure from the perspective of a shock, and for the case of initial data $f(x) = 2\theta|x|$.
\end{abstract}
\maketitle
\tableofcontents

\section{Introduction}
The object of study in the present paper is the long-time behavior of the KPZ fixed point. We give a detailed study of the invariant measures and shock fluctuations of the KPZ fixed point via the variational description from the directed landscape. Our main results can be placed into the following three main categories:
\begin{enumerate}
    \item Invariant measures from the perspective of a shock (Theorems~\ref{thm:KPZFP_invmeas_shock} and~\ref{thm:open_converge}).
    \item Classification of invariant measures for the (recentered) KPZ fixed point (Theorem~\ref{thm:KPZFP_characterize})
    \item Limiting fluctuations of shocks/competition interfaces for the KPZ fixed point (Theorem~\ref{thm:shock_fluctuations}). 
\end{enumerate}
Before stating our main results, let us first review some background and history on the KPZ fixed point and directed landscape.

\subsection{Background on the KPZ fixed point and directed landscape}

The KPZ universality class is a large class of stochastic growth models, including random matrices, particle systems, vertex models, last-passage percolation, directed polymers in random environments, stochastic PDEs, among others. Each of these models exhibits a notion of a growing interface, which is governed by universal growth rates and limiting statistics, depending only on the initial data, but not on the precise dynamics. The first major breakthrough in this area came in the work of Baik, Deift, and Johansson \cite{baik-deif-joha-99} and, shortly after, the work of \cite{Johansson-2000}, which identified the scaling exponents and limiting Tracy-Widom fluctuations for Poisson and geometric last-passage percolation. We call these models exactly solvable models because they exhibit combinatorial and representational-theoretic structures that allow for asymptotic analysis, and other models have been seen numerically to follow the same growth rates and limiting statistics. This universal phenomenon has also been discovered experimentally in several physical contexts, including the growth of liquid crystals \cite{Takeuchi-10,Takeuchi2011GrowingIU,Takeuchi-Sano-12}, colonies of living cells \cite{Mitsugu-98,Wakita-Itoh-Matsuyama-97,Huergo-10}, 
the slow burning of paper \cite{Miettinen-05,Myllys2001KineticRI,Maunuksela-97}, geological processes \cite{Halpin-Zhang-95,Yunker-13}, the interaction of molecules on smooth surfaces \cite{WANG-03,WANG-05}, and quantum entanglement \cite{Nahum-17}.

Going further, the work of Prahofer and Spohn \cite{Prahofer-Spohn-02} introduced the Airy process, which captures the spatial fluctuations of the limiting growth processes in the KPZ universality class. This led to a great deal of further work around the Airy process and related Airy line ensemble (see, for example, \cite{CorwinHammond,Dauvergne-Sarkar-Virag-2020,Dauvergne-Zhang-2021,Dauvergne-Virag-21,Aggarwal-Huang-2023}). This sparked a conjecture \cite{Corwin-Quastel-Remenik-15} that there should be a limiting object, which they called the space-time Airy sheet, that describes both the space and time fluctuations of models in the KPZ universality class. The corresponding height functions would then be described as a variational problem across the space-time Airy sheet, just as one solves an inviscid Hamilton-Jacobi equation. 

On this level, the first major breakthrough came in the construction of the KPZ fixed point, first introduced in \cite{KPZfixed} as the scaling limit of the height functions of totally asymmetric simple exclusion process. Afterwards came the construction of the conjectured space-time Airy sheet \cite{Directed_Landscape} as the scaling limit of Brownian last-passage percolation. The authors of \cite{Directed_Landscape} renamed this object the directed landscape. Then, in \cite{reflected_KPZfixed}, it was shown that a scaled system of reflected Brownian motions started from a large class of initial data converges to the KPZ fixed point with the appropriate limiting initial data. As these reflected Brownian motions can be written as a variational problem in an environment for Brownian last-passage percolation, this description was taken to the limit to show that the KPZ fixed point can be described as a variational problem in the directed landscape. Since then, the KPZ fixed point and directed landscape have been shown to be the scaling limits of several exactly solvable last-passage percolation models \cite{Dauvergne-Virag-21}, particle systems and vertex models \cite{Aggarwal-Corwin-Hegde-2024a,Aggarwal-Corwin-Hegde-24b}, the KPZ equation \cite{Wu-23}, Brownian web distance (under a shear scaling) \cite{Veto-Virag-2023}, and directed polymers \cite{Zhang-2025}. More recently, it was shown in \cite{2Dauvergne-Zhang-2024} that the directed landscape gives the unique coupling of solutions to the KPZ fixed point under natural assumptions for models in the KPZ class. This allowed them to upgrade known convergences of some models to the KPZ fixed point to convergence to the directed landscape.

\subsection{Preliminaries and notation}

Let $\BM^{d}$ denote the law of a two-sided $d$-dimensional Brownian
motion such that if $B\sim\BM^{d}$ then $B(0)=0$. We will abbreviate
$\BM=\BM^{1}$. If $B\sim\BM^{d}$, then we define 
\[
\BM^{d}(\theta,\sigma)=\Law(x\mapsto \sigma B(x)+\theta x\mathrm{e}_{1}),
\]
where $\mathrm{e}_{1}=(1,0,\ldots,0)$ is the first unit vector, and
let 
\[
\BES^{d}(\theta)=\Law(x\mapsto|B^d(x)+\theta x\mathrm{e}_{1}|).
\]
For $\sigma = 1$, we use the shorthand notation $\BM^d(\theta) := \BM^d(\theta,1)$.
If $Z\sim\mathrm{BES}^{3}(\theta)$, then $(Z(x))_{x\ge0}$ has the
distribution of a Bessel-$3$ process with drift $\theta$ as considered
in \cite{Rogers-Pitman-81}. In particular, $(Z(x))_{x\ge0}$ is a
Markov process with generator
\[
\frac{1}{2}\frac{\dif^{2}}{\dif x^{2}}+\theta\coth(\theta x)\frac{\dif}{\dif x}
\]
(see \cite[(13)]{Rogers-Pitman-81}) or, equivalently, satisfies the
SDE
\[
\dif Z(x)=\theta\coth(\theta Z(x))\dif x+\dif W(x),
\]
where $W$ is a standard Brownian motion.

We define the KPZ fixed point at time $t > 0$ and spatial location $x$, started from initial condition $f$, as 
\be \label{eq:FKPFP_def}
h(x,t\mid f) := \sup_{y \in \R}[f(y) + \Ll(y,0;x,t)],
\ee
where $\Ll$ is the directed landscape (see Section~\ref{sec:DLKPZ_intro} below). We recall the definition of the competition interface from an initial condition $f$ in \cite[Section 4]{Rahman-Virag-21}. For $p \in \R$ and $(x,t) \in \R \times (0,\infty)$, define the function
\be \label{eq:dp_def}
d_p(x,t; f) = \sup_{y \ge p}[f(y) + \Ll(y,0;x,t)] - \sup_{y \le p}[f(y) + \Ll(y,0;x,t)].
\ee
By (\cite[Proposition~4.1]{Rahman-Virag-21}), the function $x \mapsto d_p(x,t;f)$ is non-decreasing for each $p \in \R$ and $t > 0$. We then define 
\be \label{eq:Ippm}
\begin{aligned}
I_p^-(t,f) &:= \inf\{x \in \R: d_p(x,t;f) \ge 0\} \\
I_p^+(t,f) &:= \sup\{x \in \R: d_p(x,t;f) \le 0\}.
\end{aligned}
\ee

\subsection{Invariant measure from the perspective of a shock}
We now present our first main theorem, which gives a family of invariant measures from the perspective of a shock. Later, in Theorem~\ref{thm:shock_fluctuations}\ref{itm:stat_shock}, we describe the limiting fluctuations of this shock.
\begin{theorem} \label{thm:KPZFP_invmeas_shock}
    Let $B\sim \BM$ and $Z \sim \BES^3(2\theta)$ be independent and mutually independent of the directed landscape $\Ll$. Define the random function $f(x) = B(x) + Z(x)$, 
and let
    \[
    b_t^- = I_0^-(t,f),\quad\text{and}\quad b_t^+ = I_0^+(t,f).
    \]
    Then, for each fixed $t > 0$, $b_t^- = b_t^+$ with probability $1$, and
    \[
    \Law\Bigl(h(b_t^+ +\cdot,t\mid f) - h(b_t^+,t\mid f)\Bigr)  = \Law(f). 
    \]
\end{theorem}
Theorem~\ref{thm:KPZFP_invmeas_shock} is in fact a special case of a stronger invariance statement, where we build the function $f$ as the maximum of two functions $f_-,f_+$, where the tuple $(f_-,f_+)$ is invariant under recentering by $b_t^{\pm}$. See Theorem~\ref{thm:stat_from_shock_general} (Theorem~\ref{thm:KPZFP_invmeas_shock} is proved immediately afterwards). The proof of invariance comes by starting from a tuple of functions $(f_-,f_+)$ which is not invariant from the perspective of a shock, but is jointly invariant for the KPZ fixed point, in the sense that their joint distribution is preserved (after a global height shift). After a time-averaging, the coupled height functions are shown to converge to a limiting object in Lemma~\ref{lem:convergetonuhat}, whose maximum is distributed as the measure in Theorem~\ref{thm:KPZFP_invmeas_shock}. This strategy of starting from a jointly invariant initial condition is similar to the authors' previous joint work \cite{Dunlap-Sorensen-24} for the KPZ equation with space-time white noise and to the previous work of the first author and Ryzhik \cite{Dunlap-Ryzhik-2020} for the Burgers equation with smooth forcing. However, in each of those previous cases, the limiting object is absolutely continuous with respect to the initial conditions, and the proof that the time-average solution converges to a limit comes by soft ergodicity properties. By contrast, in the present work, the limiting object is singular with respect to the law of the initial condition, and the proof of convergence comes by an explicit comparison using properties of the Bessel process. See the discussion above Lemma~\ref{lem:convergetonuhat}.

Invariant measures from the perspective of a shock have also been studied in the context of the totally asymmetric simple exclusion process (TASEP). In this context, one considers the system with a second-class particle and looks at the configuration of particles centered at the random location of the second-class particle. Such invariant measures are constructed in \cite{Ferrari-Fontes-Kohayakawa-1994}. To do this, they sample a sequence of independent finite configurations of $0$s and $1$s according to a measure given in \cite[Equation (2.7)]{Ferrari-Fontes-Kohayakawa-1994}. They place second-class particles in between these configurations. Then, they select one second-class particle, set all second-class particles to the right to be first-class particles, and all second-class particles to the left to be holes. According to \cite[Theorem~1]{Ferrari-Fontes-Kohayakawa-1994}, this measure is invariant from the perspective of the single remaining second-class particle. It was later shown in \cite[Proposition~2.2]{Martin-Sly-Zhang-2025} that this process started from this initial condition and recentered by the location of the second-class particle is ergodic in time. As the KPZ fixed point is the scaling limit of TASEP \cite{KPZfixed}, one expects to be able to scale this measure and obtain the invariant measures found in Theorem~\ref{thm:KPZFP_invmeas_shock}.

Theorem~\ref{thm:KPZFP_invmeas_shock} is new, but the presence of the Brownian motion and Bessel process has been hinted in previous works. In \cite[Theorem~1.1]{Dauvergne-Sarkar-Virag-2020}, the authors considered an approximate point along the almost sure geodesic between the points $(0,0)$ and $(0,1)$. Letting $\ve \searrow 0$, they choose a time $t_\ve$ such that $t_\ve,1-t_\ve \gg \ve^3$, and consider the maximizer $X_\ve$ of the function $x \mapsto \Ll(0,0;x,t_\ve) + \Ll(x,t_\ve + \ve^3;0,1)$, which, if the $\ve^3$ term were removed, would be exactly the location of the geodesic from $(0,0)$ to $(0,1)$ at time $t_\eps$. 
In \cite[Theorem 1.1]{Dauvergne-Sarkar-Virag-2020}, the authors showed that the tuple
\[
z \mapsto \Bigl(\f{\Ll(0,0;X_\ve + \ve^2z,t_\ve)-\Ll(0,0;X_\ve,t_\ve)}{\ve}, \f{\Ll(X_\ve + \ve^2 z, t_\ve + \ve^3;0,1) - \Ll(X_\ve,t_\ve + \ve^3;0,1)}{\ve}\Bigr),
\]
as a function in $\C(\R,\R^2)$, converges in distribution, as $\ve \searrow 0$, to $(B-Z,-B - Z)$, where $B\sim \BM$ and $Z \sim \BES^3(0)$ are independent. In \cite{Dauvergne-2021}, Dauvergne considered the KPZ fixed point started from an initial condition $f$ that is equal to $-\infty$ everywhere, except for at two points $p_1 < p_2$. Then, for each $t > 0$, the location of the competition interface is a cutoff point $A$ between points $x$ such that $z \mapsto f(z) +\Ll(z,0;x,t)$ is maximized at $z = p_1$ to the left of $A$ and maximized at $z = p_2$ to the right of $A$. Then, \cite[Theorem 1.7]{Dauvergne-2021} states that the recentered process $x \mapsto h(A + x,t \mid f) - h(A,t)$ is locally absolutely continuous with respect to $B + Z$, where $B \sim \BM$, and $Z \sim \BES^3(0)$. We expect that, when started from a large class of initial conditions, the KPZ fixed point recentered at its competition interface location converges to the invariant measures in Theorem \ref{thm:KPZFP_invmeas_shock}, where $\theta$ is determined by the asymptotics of the initial condition.  Note that the initial condition in Theorem \ref{thm:KPZFP_invmeas_shock} satisfies
\[
\lim_{|x| \to \infty}\f{f(x)}{|x|} = 2\theta.
\]
If instead, one assumes that
\[
\lim_{x \to \pm \infty}\f{f(x)}{|x|} = 2\theta_\pm,
\]
with $\theta_+ > \theta_-$, then our result can be used to construct invariant measures from the perspective of the shock satisfying these conditions by using the shear invariance/skew stationarity of the directed landscape (from \cite{Directed_Landscape}, recorded here as Lemma \ref{lm:landscape_symm}\ref{itm:skew_stat}). This is analogous to the construction described in \cite[Remark~1.5]{Dunlap-Sorensen-24}.

\subsection{Convergence of invariant measures from the perspective of a shock for the open KPZ fixed point }
The works \cite{Corwin-Knizel-24,Barraquand-Le_Doussal-2022,Bryc-Wang-Wesolowski-2023} gave an explicit predicted formula for the stationary measure of the open KPZ equation with given boundary conditions $u,v$, which correspond, respectively, to the slopes of the solution at $x = 0$, and the negative of the slope at $x = L$. These formulas have now been proved for the full range of boundary conditions \cite{Barraquand-Corwin-Yang-2024,Dunlap-Gu-Rosati-25}. The authors of \cite{Barraquand-Le_Doussal-2022} also noted that this measure converges under diffusive scaling to a limiting measure, which should be the stationary measure for the conjectural open KPZ fixed point with corresponding boundary conditions (although the KPZ fixed point in the open setting has not yet actually been constructed). The limiting measure has also been realized as the limit of the stationary measure for open ASEP with appropriate boundary conditions \cite{Bryc-Wang-Wesolowski-2023,Bryc-Zatitskii-2024,Wang-Yang-2025}. In the following, we define a notion of a shock for this measure. In the case $u = v < 0$, the location of the shock stays bounded away from the boundaries, and the recentered function converges in distribution to the invariant measure from the perspective of a shock given in Theorem~\ref{thm:KPZFP_invmeas_shock}.

In \cite[Section 1.5]{Barraquand-Corwin-Yang-2024}, this measure is described as the marginal law of $W_1$, where the tuple $(W_1,W_2)$ is distributed according the measure $\mu^{u,v;L}$ on $\C[0,L]$, defined by the Radon-Nikodym derivative
\[
\f{d\mu^{u,v;L}}{d(\BM(-v) \otimes \BM(v))}(W_1,W_2) = \f{1}{C_{u,v,L}} \exp\Bigl\{(u+v)\min_{x \in [0,L]}[W_1(x) - W_2(x)]  \Bigr\},
\]
where $C_{u,v,L}$ is a normalizing constant. Here, $(W_1,W_2) \sim \BM(-v) \otimes \BM(v)$ means that $W_1,W_2$ are independent, with $W_1 \sim \BM(-v)$ and $W_2 \sim \BM(v)$. Consider the change of coordinates 
\[
B = \f{W_1 + W_2}{\sqrt 2},\quad\text{and}\quad Z = \f{W_1 - W_2}{\sqrt 2}.
\]
Then, we can alternatively represent the law of $W_1$ as the law of $\f{B + Z}{\sqrt 2}$ on $[0,L]$, where $B$ and $Z$ are independent, $B \sim \BM$, and $Z \sim \nu^{u,v;L}$, where 
\[
\f{d\nu^{u,v;L}}{d \BM(-2^{1/2} v)}(Z) = \f{1}{C_{u,v,L}} \exp\Bigl\{2^{1/2}(u+v) \min_{x \in [0,L]} Z(x)\Bigr\}. 
\]
for a different normalizing constant $C_{u,v,L}$.
Using Girsanov's theorem, we may rewrite this measure in terms of driftless Brownian motion as follows (changing the constant $C_{u,v,L}$ from line to line): 
\[
\f{d\nu^{u,v;L}}{d \BM}(Z) = \f{1}{C_{u,v,L}} \exp\Bigl\{2^{1/2} (u+v) \min_{x \in [0,L]} Z(x) - 2^{1/2} v Z(L)\bigr\}
\]
When $u = v$, we simply write $\nu^{u;L}$ so that
\[
\frac{\dif\nu^{u;L}}{\dif\BM}(Z)=\frac{1}{C_{u,L}}\exp\left\{ 2^{3/2}u\min_{x\in[0,L]}Z(x)-2^{1/2}uZ(L)\right\}.
\]
We have the following theorem, proved in Section~\ref{sec:open_to_full_conv},  which shows that, for $u = v$, the open KPZ fixed point stationary measure, after recentering by an appropriately defined shock, converges to the measure given in Theorem~\ref{thm:KPZFP_invmeas_shock}.

\begin{theorem} \label{thm:open_converge}
Given $\theta > 0$, let $B \sim \BM$ and $Z \sim \nu^{-\theta;L}$ be independent.
If we set $
A\coloneqq\argmin_{x\in[0,L]}Z(x)$, then 
we have
\begin{equation}
\lim_{L\to\infty}\Law(A/L)=\Uniform([0,1]).\label{eq:AAsymptoticallyuniform}
\end{equation}
Moreover, if we set
\be \label{eq:fA}
f_A(x) := \f{B(A + 2x) +Z(A + 2x) - B(A) - Z(A)}{\sqrt 2}, \quad\text{for }x \in \Bigl[-\f{A}{2},L - \f{A}{2}\Bigr],
\ee
and extend $f_A$ to a continuous function $\R \to \R$ by setting $f_A$ to be constant on $(-\infty,-\f{A}{2}]$ and on $[L - \f{A}{2},\infty)$, 
  then 
$f_A$ converges in distribution as $L\to\infty$, with respect to the topology of uniform convergence on compact sets, to $\wt B + \wt Z$, where $\wt B \sim \BM$ and $\wt Z \sim \BES^3(2\theta)$ are independent. 
\end{theorem}
Theorem~\ref{thm:open_converge} is proved in Section~\ref{sec:open_to_full_conv}.
   We interpret the random variable $A$ in Theorem~\ref{thm:open_converge} as the location of a shock, which we note is not measurable with respect to the initial condition. Instead, the shock is measurable with respect to $(B,Z)$, and then the invariant initial condition is constructed as $B + Z$. Once the shock location is identified, the key to the proof follows by the arcsine law of the maximizer of Brownian motion and Denisov's path decomposition (recorded here as Proposition \ref{prop:BM_min}).  The latter states that, given the maximizer of Brownian motion on a compact interval, the Brownian motion to the left and right of the maximizer are independent Brownian meanders.  While these follow by classical facts about Brownian motion, the key innovation in the proof of Theorem \ref{thm:open_converge} is to identify the appropriate notion of a shock, noting that one cannot do this directly from the initial condition itself. 
   
   The convergence of the rescaled shock to the uniform distribution can be compared with \cite[Theorem~1.6]{Wang-Wesolowski-Yang-2024}, where it was shown that rescaled and recentered height functions for stationary open ASEP converge, as the number of sites converges to $\infty$, to a continuous function with two linear segments, whose transition point is given by a uniform random variable. The notion of a shock in that setting, however, was not directly defined.

\subsection{Characterization of the invariant measures for the KPZ fixed point}
Our next main theorem gives a characterization of all invariant measures for the KPZ fixed point. It has been known since \cite{KPZfixed} that for each $\theta \in \R$, $\BM(2\theta,\sqrt 2)$ is invariant for the recentered KPZ fixed point. ``Recentering'' here means 
recentering by the value at $x = 0$. That is, we say that $\mu$ is an invariant measure for the recentered KPZ fixed point evolution if, when $f \sim \mu$, for all $t > 0$, we have 
\[
\Law\bigl(h(\cdot,t \mid f)- h(0,t \mid f)\bigr) = \mu.
\]
Without the recentering, there are no invariant measures because the KPZ fixed point grows in time. In \cite{Pimentel-21a,Pimentel-21b}, it was shown that $\BM(0,\sqrt 2)$ is the unique invariant measure under a certain scaling condition. This was generalized in \cite{Busa-Sepp-Sore-22a}, where it was shown that, for each $\theta \in \R$, $\BM(2\theta,\sqrt 2)$ is the unique invariant measure supported on functions $f$ satisfying the following condition:
\begin{equation} \label{eqn:drift_assumptions}
    \begin{aligned}
    &\text{if } \theta = 0, \quad &\liminf_{x \to -\infty} \f{f(x)}{x} \in [0,+\infty]  \qquad &\text{and}\quad &\limsup_{x \to +\infty} \f{f(x)}{x} \in [-\infty,0], \\
    &\text{if } \theta > 0,\quad  &\liminf_{x \to -\infty} \f{f(x)}{x} \in (-2\theta,+\infty]\qquad&\text{and}\quad &\lim_{x \to +\infty} \f{f(x)}{x} = 2\theta, \\
    &\text{and if } \theta < 0,\quad &\lim_{x \to -\infty} \f{f(x)}{x} = 2 \theta\qquad&\text{and}\quad &\limsup_{x \to +\infty} \f{f(x)}{x} \in [-\infty, -2\theta).
    \end{aligned}
\end{equation}
These conditions form natural (although not complete) basins of attraction for solutions that are invariant under spatial and temporal translations. Similar results for another zero-temperature model in the KPZ universality class, this one without any integrability structure, were shown in \cite{Bakhtin-Cator-Konstantin-2014}; see also \cite{EKMS-2000,Bakhtin-2013,Bakhtin-16,Bakhtin-Li-19} for similar results in this direction from a qualitative perspective. The work \cite{Janj-Rass-Sepp-22} gives similar basins of attraction in the KPZ equation. 
In the following theorem, we characterize all extremal invariant measures for the KPZ fixed point, without any assumptions on the asymptotic slope. In the present paper, we work on the state space
\be \label{eq:CFP}
\begin{aligned}
\mathcal{C}_{\mathrm{FP}}\coloneqq\Bigl\{ f:\C(\R):  \sup_{x \in \R}[f(x) - ax^2] < \infty \text{ for all }a > 0\Bigr\},
\end{aligned}
\ee
and we describe its topology (and therefore also its Borel $\sigma$-algebra) in Section \ref{sec:state_space}.
\begin{theorem} \label{thm:KPZFP_characterize}
    A measure $\mu$ on $\mathcal{C}_{\mathrm{FP}}$ (defined in \eqref{eq:CFP} below) is an extremal (i.e.,\ time-ergodic) invariant measure for the recentered KPZ fixed point if and only if $\mu = \BM(2\theta,\sqrt 2)$ for some $\theta \in \R$.
\end{theorem}
The proof of Theorem \ref{thm:KPZFP_characterize} has three main ingredients. The first is to show that all extremal invariant measures are supported on functions having fixed asymptotic slopes at $-\infty$ and $+\infty$. This is done in Proposition \ref{prop:inv_meas_have_slope}. 
Given an invariant measure, we construct an eternal (i.e., bi-infinite in time) solution to the KPZ fixed point by taking subsequential limits as the initial time goes to $-\infty$. Given an eternal solution, it was recently shown in \cite{Bhattacharjee-Busani-Sorensen-25} that one can construct a family of semi-infinite geodesics from it. The key new innovation comes in Lemma~\ref{lem:only_3_slopes}, where we use stationarity to show that there can be no more than $3$ distinct directions of semi-infinite geodesics starting from a single time horizon. By monotonicity of these geodesic directions, when we look far to the right, all geodesics built from the eternal solution will have the same direction. Then, in Proposition \ref{prop:inv_meas_have_slope}, we use coalescence properties from \cite{Busa-Sepp-Sore-22a,Busani-2023} to show that that the extreme directions of the bi-infinite geodesics correspond to the asymptotic slopes of the solution.  An analogous result that extremal invariant measures are supported on functions with slopes was proved in \cite{Janj-Rass-Sepp-22} for the KPZ equation, which also uses eternal solutions, but the techniques in zero-temperature are quite different. In particular, the proof in \cite{Janj-Rass-Sepp-22} uses a martingale argument in the quenched environment. In the zero-temperature setting, there is no additional layer of randomness, and there is no analogous martingale one can construct. 

The next ingredient is the uniqueness under the slope conditions \eqref{eqn:drift_assumptions}, which was shown previously in \cite{Busa-Sepp-Sore-22a} (see Propositions~\ref{prop:invariance_of_SH} and~\ref{prop:uniform_upbd} below). Given this input, to complete the classification of extremal invariant measures, it remains to rule out the case of an invariant measure supported on functions $f$ satisfying $\lim_{|x| \to \infty} \f{f(x)}{|x|} = 2\theta$ for some $\theta > 0$. We call such functions \textit{$V$-shaped functions}. In Section~\ref{sec:no_V-shape}, we rule out the existence of invariant measures supported on $V$-functions in several steps. On a high level, the idea is to show that, when started with a $V$-shaped initial condition, there is a random shock that fluctuates on order greater than $O(1)$ as $t \to \infty$. The study of shocks for three choices of initial conditions is discussed in the following section. 

It is a natural problem to classify all extremal (time-ergodic) stationary measures for a Markov process. In the context of interacting particle systems, Spitzer \cite{Spitzer-1970} initiated this work for the simple exclusion process on $\Z$. The works of Spitzer
and Liggett proved that i.i.d.~Bernoulli measures are the
only extremal stationary measures 
in the case when the transition rates are symmetric in space \cite{Liggett-1973,Liggett1974a,Spitzer1974},
and in the case when the Markov chain is positive recurrent and reversible
\cite{Liggett1974b}. In a more complicated setting that does not enjoy these symmetries, Liggett \cite{Liggett1976} showed that in the context of the asymmetric simple exclusion process (ASEP), the extremal stationary measures consist of the family of i.i.d.\ Bernoulli measures and a family of blocking measures. This is a general classification that also works in the case of TASEP, noting that the blocking measures depend on the asymmetry parameter. ASEP scales to the KPZ equation \cite{Bertini-Giacomin-1997} as the asymmetry parameter tends to $1$ at the appropriate rate compared to the scaling of space and time, and also scales to the KPZ fixed point as the asymmetry parameter is fixed and space and time are scaled \cite{Aggarwal-Corwin-Hegde-2024a,Aggarwal-Corwin-Hegde-24b}. The set of blocking measures does not survive these scalings, so it is natural to expect that Brownian motions with drift are the only extremal stationary measures in both of these settings. However, these scaling limits can be used to show that Brownian motions with drift are invariant, but cannot tell us that these are the \textit{only }extremal stationary measures. The methods used for the classification in particle systems also rely heavily on the discrete nature of the system, so they cannot be transferred to the continuum setting.  In our previous work \cite{Dunlap-Sorensen-24}, we completed this classification for the KPZ equation. This built on the previous aforementioned work \cite{Janj-Rass-Sepp-22}, which showed that all extremal stationary measures are either Brownian motions with drift or are supported on $V$-shaped initial conditions. In the present work, we do not have the same starting point, so, as mentioned above, we need to prove that the extremal stationary measures are supported on functions having fixed asymptotic slopes, and the methods for doing so in this zero-temperature continuum setting are different.

\subsection{Fluctuations of the shocks}
Our last main theorem concerns fluctuations of the shock location for three choices of initial data. We define shocks in terms of the location of competition interfaces. An equivalent definition of shocks under a technical assumption (satisfied by each of the cases below) is described in Proposition~\ref{prop:interfaces_are_shocks}.

\begin{theorem} \label{thm:shock_fluctuations}
We have the following distributional convergences as $t \to \infty$.
\begin{enumerate}[label=\rm(\roman{*}), ref=\rm(\roman{*})]  \itemsep=3pt
\item \rm{(}Invariant measure from the shock\rm{)}\label{itm:stat_shock} If  $f$ is distributed as in Theorem~\ref{thm:KPZFP_invmeas_shock}, and we define 
\be
    b_t^- = I_0^-(t,f),\quad\text{and}\quad b_t^+ = I_0^+(t,f),\label{eq:btsimple}
\ee then 
\[
\f{b_t^+}{\sqrt t} \Longrightarrow \Nor\bigl(0,(4\theta)^{-1}\bigr). 
\]
(Recall that $b_t^- = b_t^+$ with probability $1$ by Theorem~\ref{thm:KPZFP_invmeas_shock}.)
\item \rm{(}Deterministic initial data\rm{)} \label{itm:flat_shock} If $f(x) = 2\theta |x|$ and $b_t^+$ and $b_t^-$ are defined as in \eqref{eq:btsimple}, then
\[
\f{b_t^+ - b_t^-}{t^{1/3}} \Longrightarrow 0,\quad\text{and}\quad\f{b_t^+}{t^{1/3}} \Longrightarrow \f{X_1 - X_2}{2^{8/3}\theta},
\]
where $X_1$ and $X_2$ are independent Tracy-Widom GOE random variables. 
\item \rm{(}Jointly invariant initial data\rm{)} \label{itm:joint_stat_shock} If $f(x)  = f_-(x) \vee f_+(x)$, where $(f_-,f_+) \sim \nu_\theta$ (see Definition~\ref{nutheta_def}), and we define
\[
b_0^- = \inf\{x \in \R: f_-(x) = f_+(x)\},\quad b_0^+ = \sup\{x \in \R: f_-(x) = f_+(x)\},
\]
and for $t > 0$, define
\[
b_t^{\pm} = I_{b_0^{\pm}}^{\pm}(t,f),
\]
then we obtain the same convergence as in Item~\ref{itm:stat_shock} and also have 
\[
\f{b_t^+ - b_t^-}{\sqrt t} \Longrightarrow 0.
\]
\end{enumerate}
\end{theorem}
Theorem~\ref{thm:shock_fluctuations} is proved in Section~\ref{sec:shocks}. An analogous theorem for the KPZ equation was proved in our previous work \cite[Theorem~1.9]{Dunlap-Sorensen-24}. Similarly as in that work, the key is to represent $f$ as a function of $f_-$ and $f_+$ for appropriate choices of initial conditions (in this case, the maximum, which differs from the positive temperature case). We then prove fluctuation results for the difference $h_+(0,t) - h_-(0,t)$, where $h_{\pm}(x,t) = h(x,t \mid f_{\pm})$; see Propositions~\ref{prop:hpm_dif},~\ref{prop:tilt_h=-_dif}, and~\ref{prop:flat_h_dif}. The shocks can then also be defined as the leftmost and rightmost locations $x$ where $h_-(x,t) = h_+(x,t)$; see Proposition~\ref{prop:interfaces_are_shocks}.  The solution $x \mapsto h_\pm(x,t)$ has slope $\pm 2\theta$, so we make the approximation
\[
h_+(x,t) -h_+(t,0 )  - (h_-(x,t) - h_-(0,t)) \approx 4\theta x.
\]
Then, we have 
\[
b_t^- \approx b_t^+ \approx \f{h_-(0,t) - h_+(0,t)}{4\theta}.
\]
Lemma~\ref{lem:h_to_b} makes this precise, allowing us to translate our results on the fluctuations of $h_+(0,t) - h_-(0,t)$ to the fluctuations of the shock. For Item~\ref{itm:joint_stat_shock}, this comes via an explicit calculation,  the fact that $\nu_\theta$ is a jointly invariant measure for the KPZ fixed point \cite{Busa-Sepp-Sore-22a}, and rescaling properties of the directed landscape from \cite{Directed_Landscape}. This is quite similar to the analogous result in \cite{Dunlap-Sorensen-24}. The analogue to Item~\ref{itm:stat_shock} in the context of the KPZ equation came because the two relevant measures were mutually absolutely continuous, so the computation boiled down to controlling the Radon-Nikodym derivative. In the current setting of the KPZ fixed point, this absolute continuity does not hold, so we instead make explicit estimate using properties of the Bessel process. To get Item~\ref{itm:flat_shock}, the Tracy--Widom GOE distribution appears after shear invariance since the KPZ fixed point with flat initial data has this one-point distribution. However, to get the independence of the two random variables, we employ a decoupling technique for the directed landscape via an approximation from Poisson last-passage percolation, using the convergence proved in \cite{Dauvergne-Virag-21}. This bears similarity to the approach taken in the proof of \cite[Proposition~2.6]{Dauvergne-2024}.

Results of this flavor have been shown previously for particle systems. Ferrari and Fontes \cite{Ferrari-Fontes-1994,Ferrari-Fontes-1994b} showed that, when started from a configuration in ASEP with a single second-class particle, and independent i.i.d.\ Bernoulli measures of different rates to the left and right of the second-class particle, the temporal process of second-class particle converges under diffusive scaling to a Brownian motion. Ferrari, Ghosal, and Nejjar \cite{Ferrari-Ghosal-Nejjar-2019} later showed that, in TASEP, when started from a deterministic initial condition with equally-spaced particles to the left and right of the second-class particle, the particle fluctuates on the order $t^{1/3}$ and similarly converges to the difference of two independent Tracy-Widom GOE random variables.

\subsection{Organization of the paper}
Section~\ref{sec:prelim} gives some preliminaries on the functions spaces for the KPZ fixed point, the directed landscape, and the stationary horizon. The following three sections contain the proofs of the main results. Theorems~\ref{thm:KPZFP_invmeas_shock} and~\ref{thm:open_converge} are proved in Section~\ref{sec:inv_meas_shock_proofs}, Theorem~\ref{thm:KPZFP_characterize} is proved in Section~\ref{sec:characterize}, and Theorem~\ref{thm:shock_fluctuations} is proved in Section~\ref{sec:shocks}. Appendix~\ref{sec:DLKPZFP} records some inputs on the directed landscape and KPZ fixed point, Appendix~\ref{sec:SIG} records inputs regarding infinite geodesics in the directed landscape, Appendix~\ref{appx:state_space} proves some auxiliary lemmas regarding our choice of state space for the KPZ fixed point, and Appendix~\ref{sec:Bessel} records some technical facts about the Bessel process and Brownian meander.

\subsection{Acknowledgments and funding}
We wish to thank Duncan Dauvergne and Dominik Schmid for pointers to connections to this work in related literature. We also thank the anonymous referee for the incredibly detailed comments which greatly improved this paper.  A.D. was partially supported by the National Science Foundation under grant no.\ DMS-2346915. E.S. was partially supported by the Fernholz Foundation and by Ivan Corwin's Simons Investigator Grant \#929852.

\section{Preliminaries} \label{sec:prelim}

\subsection{Notation and conventions}
We use $\C(A,B)$ to denote the space of continuous functions $A \to B$, where $A$ and $B$ are topological spaces. We use $\C(A)$ to denote $\C(A,\R)$ and $\C_b(A)$ to denote the set of bounded continuous functions $A \to \R$. We use $\deq$ to denote equality in distribution. We use $a \vee b$ to denote $\max(a,b)$ and $a \wedge b$ to denote $\min(a,b)$.

\subsection{State space and topology} \label{sec:state_space}
The state space for solutions to the KPZ fixed point in the present paper is the space $\CFP$ defined in \eqref{eq:CFP}. This space differs from state space used in the paper \cite{KPZfixed} in two ways. First, it is assumed in \cite{KPZfixed} that the functions in the state space are bounded  from above by the absolute value of an affine function. This is a strictly weaker assumption than that used here.  Proposition~\ref{prop:h_pres_CFP} in Appendix~\ref{appx:state_space} proves that the space $\CFP$ is preserved by the KPZ fixed point dynamics. Furthermore, the growth bounds on the directed landscape (recorded here as Lemma~\ref{lem:Landscape_global_bound}) imply that if $\sup_{x \in \R}[f(x) - ax^2] = \infty$ for some $a > 0$, then $h(\cdot,t \mid f)$ explodes for some finite  $t > 0$. Ultimately, we show in Theorem~\ref{thm:KPZFP_characterize} that Brownian motions with drift are the only invariant measures, and these live in this smaller space of functions bounded by an affine function. However, using this larger space allows us to rule out other possibilities that could potentially live in this larger space of non-explosive initial conditions. Second, the state space in \cite{KPZfixed} consists of upper semi-continuous functions $f:\R \to \R \cup \{-\infty\}$ such that $f(x) > -\infty$ for some $x \in \R$. However, started from such initial data, the solution becomes continuous immediately. This was shown for functions bounded from above by the absolute value of an affine function in \cite[Theorem~4.13]{KPZfixed},  but can be shown to hold more generally for functions satisfying the condition $\sup_{x \in \R}[f(x) - ax^2]< \infty$ for all $a > 0$ (for example, by using the bounds on the growth of the directed landscape from \cite[Corollary 10.7]{Directed_Landscape}, recorded here as Lemma~\ref{lem:Landscape_global_bound}, and adapting the proof of Proposition~\ref{prop:h_pres_CFP} in the present paper). Hence, any invariant measure must be supported on the space of continuous functions, and we do not lose generality by focusing on continuous functions.  We prove several properties of the space $\CFP$ in Appendix~\ref{appx:state_space}. 
 
We define a topology on $\mathcal C_{\mathrm{FP}}$ as follows.
We say that $f_n \to f$ in $\mathcal{C}_{\mathrm{FP}}$ if the following two conditions hold:
\begin{enumerate} 
    \item $f_n$ converges to $f$, uniformly on compact sets.
    \item For each $a > 0$, 
\[
\lim_{n \to \infty}\sup_{x \in \R}[f_n(x)  - ax^2] = \sup_{x \in \R}[f(x)  - ax^2].
\]
\end{enumerate}
 
We show in Lemma~\ref{lem:Polish} that this is a Polish topology. Since a continuous function $\R \to \R$ is determined by its values on the rationals, one can readily see that the Borel $\sigma$-algebra induced by this metric is the $\sigma$-algebra generated by the coordinate maps $x \mapsto f(x)$.

We now define several important subspaces of $\CFP$. First, set 
\[
\mathcal{C}_{\mathrm{FP};0}=\left\{ f\in\mathcal{C}_{\mathrm{FP}}\st f(0)=0\right\} .
\]
By the choice of recentering, any invariant measure on $\CFP$ for the recentered KPZ fixed point is necessarily supported on the set $\mathcal C_{\mathrm{FP};0}$.
Furthermore, for $\theta > 0$, define the space of $V$-shaped initial conditions
\be \label{V_space}
\mathcal V(\theta) = \Bigl\{f \in \mathcal C_{\mathrm{FP}}: \lim_{|x| \to \infty} \f{f(x)}{|x|} = 2\theta\Bigr\}.
\ee
Additionally, set 
\be \label{Y_space}
\mathcal Y(\theta) = \Bigl\{(f_-,f_+) \in \mathcal C_{\mathrm{FP}}^2: \lim_{|x| \to \infty} \f{f_{\pm}(x)}{x} = \pm 2\theta\Bigr\},
\ee
and 
\be \label{X_space}
\mathcal X(\theta) = \{(f_-,f_+) \in \mathcal Y(\theta): f_+ - f_- \text{ is nondecreasing}\}.
\ee
Lastly, we set 
\be \label{zero_spaces}
\mathcal V_0(\theta) = \mathcal V(\theta) \cap\mathcal C_{\mathrm{FP};0},\quad  \mathcal Y_0(\theta) = \mathcal Y(\theta) \cap\mathcal C_{\mathrm{FP};0}^2,\quad\text{and}\quad \mathcal X_0(\theta) = \mathcal X(\theta) \cap\mathcal C_{\mathrm{FP};0}^2.
\ee
The sets $\mathcal{C}_{\mathrm{FP};0}$, $\mathcal V(\theta)$, and $\mathcal V_0(\theta)$ are each Borel-measurable subsets of the space $\mathcal{C}_{\mathrm{FP}}$, while $\mathcal Y(\theta)$, $\mathcal Y_0(\theta)$, $\mathcal X(\theta)$, and $\mathcal X_0(\theta)$ areeach Borel-measurable subsets of the space $\mathcal C_{\mathrm{FP}}^2$. We equip these subspaces with their respective subspace topologies. 

Define the map $V:\mathcal Y(\theta) \to \mathcal V(\theta)$ by
\be \label{eq:V_map}
V[f_-,f_+](x) = f_-(x) \vee f_+(x).
\ee

\subsection{Directed landscape and KPZ fixed point} \label{sec:DLKPZ_intro} We now recall the directed landscape and the KPZ fixed point and briefly review their relevant properties. 
\subsubsection{Directed landscape}
Let $\Rup = \{(y,s;x,t) \in \R^4: s < t\}$. First introduced in \cite{Directed_Landscape}, the directed landscape is a random continuous function $\Ll:\Rup \to \R$. It satisfies the following \textit{metric composition property}: for $s < r < t$ and $y,x \in \R$,
\be \label{L_comp}
\Ll(y,s;x,t) = \sup_{z \in \R}\{\Ll(y,s;z,r) + \Ll(z,r;x,t)\}.
\ee
This implies the \textit{reverse triangle inequality}:
for $s < r < t$ and $x,y,z \in \R$,
\be \label{triangle}
\Ll(y,s;x,t) \ge \Ll(y,s;z,r) + \Ll(z,r;x,t).
\ee
Further properties of the directed landscape are discussed in Appendix~\ref{sec:DLKPZFP}. One can define a notion of geodesics for the directed landscape as follows.  We consider a continuous function $\gamma:[s,t] \to \R$ as a directed path in the plane. For such a path, define its $\Ll$-length as
\[
\Ll(\gamma) = \adjustlimits\inf_{k \in \N} \inf_{s = t_0 < t_1 < \cdots < t_k = t} \sum_{i = 1}^k \Ll\bigl(\gamma(t_{i - 1}),t_{i - 1};\gamma(t_i),t_i\bigr).
\]
Note that the $\Ll$-length of a path $\gamma$ can be $-\infty$. A path $\gamma$ is called a \textit{geodesic} from $(y,s)$ to $(x,t)$ if its $\Ll$ length is maximal among all paths $\gamma:[s,t] \to \R$ with $\gamma(s) = y$ and $\gamma(t) = x$. Equivalently, 
\[
\Ll(\gamma) = \sum_{i = 1}^k \Ll\bigl(\gamma(t_{i - 1}),t_{i - 1};\gamma(t_i),t_i\bigr)
\]
for all $k \in \N$ and partitions $s = t_0 < t_1 < \cdots < t_k = t$. In particular, if $s < r <t$, a point $z \in \R$ satisfies $\gamma(r) = z$ for some geodesic from $(y,s)$ to $(x,t)$ if and only if equality holds in \eqref{triangle}. The fact that geodesics exist is stated in the following lemma.
\begin{lemma} \cite[Theorems~12.1 and~13.1]{Directed_Landscape} \label{lem:geodesics}
With probability $1$, simultaneously for all $(y,s;x,t) \in \Rup$, there exist leftmost and rightmost geodesics between $(y,s)$ and $(x,t)$. For fixed $(y,s)$ and $(x,t)$, there exists a unique geodesic with probability $1$.
\end{lemma}

 \begin{definition} \label{def:sigma_alg}
     For $r\in\mathbb{R}$,
we let $\mathcal{F}_{\le r}$ and $\mathcal{F}_{\ge r}$ be the independent $\sigma$-algebras (see \cite[Definition~10.1]{Directed_Landscape})
generated by $\{\mathcal{L}(y,s;x,t)\st s<t\le r\}$ and $\{\mathcal{L}(y,s;x,t)\st r\le s<t\}$,
respectively.
 \end{definition}

\subsubsection{KPZ fixed point}
Generalizing the definition \eqref{eq:FKPFP_def} to arbitrary starting times, the KPZ fixed point evolution with initial data $f$ at time $s$
is given by
\[
h(x,t \mid s,f)=\max_{y\in\mathbb{R}}\left[f(y)+\mathcal{L}(y,s;x,t)\right].
\]
If $\underline{f}=(f_{1},\ldots,f_{N})\in\mathcal{C}_{\mathrm{FP}}^{N}$,
we define the coupled evolution via $\Ll$ as
\[
\underline{h}(x,t \mid s,\underline{f})=\left(h(x,t \mid s,f_{i})\right)_{i=1}^{N}.
\]
Given $F\in\mathcal{C}(\mathcal{C}_{\mathrm{FP}}^{N})$ and $\underline{f}\in\mathcal{C}_{\mathrm{FP}}^{N}$,
we define the semigroup corresponding to the KPZ fixed point evolution as
\[
P_{t}F(\underline{f})\coloneqq\mathbb{E}\left[F(\underline{h}(x,t \mid s,\underline{f}))\right].
\]
 The
adjoint is defined as follows: for a measure $\nu$ on $\mathcal{C}_{\mathrm{FP}}^{N}$,
\[
P_{t}^{*}\nu\coloneqq\Law\left(\underline{h}(x,t \mid s,\underline{f})\right),\qquad\text{where }\underline{f}\sim\nu\text{ (independent of }\mathcal{L}).
\]
We now define a few projection maps. First we define
\begin{equation}
\pi_{x_{0}}[\underline{f}](x)=\underline{f}(x_{0}+x)-\underline{f}(x_{0}).\label{eq:pix0}
\end{equation}
For $\underline{f}=(f_{-},f_{+}) \in \mathcal X(\theta)$, we define
\begin{equation}
\mathfrak{b}_{-}[\underline{f}]\coloneqq\min\{x\in\mathbb{R}\st f_{-}(x)=f_{+}(x)\}\qquad\text{and}\qquad\mathfrak{b}_{+}[\underline{f}]\coloneqq\max\{x\in\mathbb{R}\st f_{-}(x)=f_{+}(x)\}.\label{eq:bdef}
\end{equation}
Then, we define
\begin{equation}
\pi_{\mathrm{Sh};\pm}[\underline{f}](x)\coloneqq\pi_{\mathfrak{b}_{\pm}[\underline{f}]}[\underline{f}](x)\overset{\eqref{eq:pix0}}{=}\underline{f}(\mathfrak{b}_{\pm}[\underline{f}]+x)-\underline{f}(\mathfrak{b}_{\pm}[\underline{f}]).\label{eq:piShdef}
\end{equation}
From the definition of the space $\mathcal X(\theta)$, it is immediate that the quantities $\mathfrak{b}_{-}[\underline{f}]$ and $\mathfrak{b}_{+}[\underline{f}]$ are well-defined and finite. 

In the following result, we connect these projection mappings to the definitions of competition interfaces. Recall the definitions of the left and right competition interfaces in \eqref{eq:Ippm}.  
\begin{proposition} \label{prop:interfaces_are_shocks}
Let $\theta > 0$, and $\underline f = (f_-,f_+) \in \mathcal X(\theta)$. Set $f_{\mathsf V}(x) = V[f_-,f_+](x) = f_-(x) \vee f_+(x)$. Then, for all $t > 0$, 
\[
\mathfrak b_\pm[h(\cdot,t \mid \underline f)] = I^\pm_{\mathfrak b_\pm[\underline f]}(t,f_{\mathsf V}).
\]
\end{proposition}
\begin{proof}
  We just prove the ``$-$'' statement, since the ``$+$" second follows by a symmetric proof. For shorthand notation, let $h_{\pm} = h(\cdot,t \mid f_{\pm})$ and $b_t^{\pm} = \mathfrak b_\pm[h(\cdot,t \mid \underline f)]$ for $t > 0$. By Lemmas~\ref{lem:slope_conserved} and~\ref{KPZFP_attr}, $(h_-,h_+)(\cdot,t) \in \mathcal X(\theta)$ for all $t > 0$. Then,
\be \label{eq:ht_mont}
h_{\mathsf V}(x,t) = \begin{cases} h_-(x,t) &x \le b_t^+ \\
h_+(x,t) &x \ge b_t^-.
\end{cases}
\ee
Since $b_t^- \le b_t^+$, we note that, for $x \in [b_t^-,b_t^+]$, we have $h_{\mathsf V}(x,t) = h_-(x,t) = h_+(x,t)$. We also note that \eqref{eq:ht_mont} is valid for $t = 0$ when we define $h_{\pm}(x,0) = f_{\pm }(x)$, $h_{\mathsf V}(x,0) = f_{\mathsf V}(x)$, and $b_0^\pm = \mathfrak b_{\pm}[h(\cdot,0 \mid \underline f)] = \mathfrak b_{\pm}[\underline f]$. For concreteness, we shall write this as 
\be \label{eq:f_mont}
f_{\mathsf V}(x) = \begin{cases} f_-(x) &x \le b_0^+ \\
f_+(x) &x \ge b_0^-.
\end{cases}
\ee
Since $h_-(b_t^-,t) = h_+(b_t^-,t)$, we have, 
\be \label{eq:fsup123}
\begin{aligned}
\sup_{y \in \R}[f_+(y) + \Ll(y,0;b_t^-,t)] &= \sup_{y \in \R}[f_-(y) + \Ll(y,0;b_t^-,t)] \\
&\ge \sup_{y \le b_0^-}[f_-(y) + \Ll(y,0;b_t^-,t)] \ge \sup_{y \le b_0^-}[f_+(y) + \Ll(y,0;b_t^-,t)],
\end{aligned}
\ee
where the last inequality follows because $f_+ - f_-$ is nondecreasing, and $b_0^-$ is the minimal value of $x$ such that $f_+(x) - f_-(x) = 0$. Indeed, since $f_-(y) > f_+(y)$ for $y < b_0^-$, the last inequality in \eqref{eq:fsup123} is an equality only when the maximum of $f_+(y) + \Ll(y,0;b_t^-,t)$ is attained at $y = b_0^-$. Therefore, we either have
\[
\sup_{y \in \R}[f_+(y) + \Ll(y,0;b_t^-,t)]  > \sup_{y \le b_0^-}[f_+(y) + \Ll(y,0;b_t^-,t)],
\]
or 
\[
\sup_{y \in \R}[f_+(y) + \Ll(y,0;b_t^-,t)]  =  \sup_{y \le b_0^-}[f_+(y) + \Ll(y,0;b_t^-,t)] = f_+(b_0^-) + \Ll(y,0;b_t^-,t).
\]
In either case, we have
\be \label{eq1}
h_+(b_t^-,t) = \sup_{y \in \R}[f_+(y) + \Ll(y,0;b_t^-,t)] = \sup_{y \ge b_0^-}[f_+(y) + \Ll(y,0;b_t^-,t)].
\ee
Then, 
\be \label{ge0}
\begin{aligned}
0 &= h_+(b_t^-,t) - h_-(b_t^-,t) \\ &= \sup_{y \ge b_0^-}[f_+(y) + \Ll(y,0;b_t^-,t)] - \sup_{y \in \R}[f_-(y) + \Ll(y,0;b_t^-,t)]  \\
&\le \sup_{y \ge b_0^-}[f_+(y) + \Ll(y,0;b_t^-,t)] - \sup_{y \le b_0^-}[f_-(y) + \Ll(y,0;b_t^-,t)] \\
&= \sup_{y \ge b_0^-}[f_{\mathsf V}(y) + \Ll(y,0;b_t^-,t)] - \sup_{y \le b_0^-}[f_{\mathsf V}(y) + \Ll(y,0;b_t^-,t)] = d_{b_0^-}(b_t^-,t;f_{\mathsf V}).
\end{aligned}
\ee
In the second equality, we used \eqref{eq1}, and in the penultimate equality, we used \eqref{eq:f_mont}. The last equality is the definition of $d_p$ from \eqref{eq:dp_def}. 

Now, by definition of $b_t^-$ and since $h_+(\cdot,t) - h_-(\cdot,t)$ is nondecreasing (because $(h_-,h_+) \in \mathcal X(\theta)$), for $x < b_t^-$, we have $h_+(x,t) < h_-(x,t)$, which, by \eqref{eq:f_mont}, implies
\begin{align*}
\sup_{y \ge b_0^-} [f_-(y) + \Ll(y,0;x,t)] &\le \sup_{y \ge b_0^-} [f_+(y) + \Ll(y,0;x,t)]  \\
&\le \sup_{y \in \R} [f_+(y) + \Ll(y,0;x,t)] < \sup_{y \in \R}[f_-(y) + \Ll(y,0;x,t)],
\end{align*}
and therefore,
\be \label{eq3}
\sup_{y \in \R}[f_-(y) + \Ll(y,0;x,t)] = \sup_{y \le b_0^-} [f_-(y) + \Ll(y,0;x,t)].
\ee
Then, for $x < b_t^-$,
\begin{align*}
    d_{b_0^-}(x,t;f_{\mathsf V}) &= \sup_{y \ge b_0^-}[f_{\mathsf V}(y) + \Ll(y,0;x,t)] - \sup_{y \le b_0^-}[f_{\mathsf V}(y) + \Ll(y,0;x,t)]  \\
    &=\sup_{y \ge b_0^-}[f_{+}(y) + \Ll(y,0;x,t)] - \sup_{y \le b_0^-}[f_{-}(y) + \Ll(y,0;x,t)] \\
    &\le \sup_{y \in \R}[f_{+}(y) + \Ll(y,0;x,t)] - \sup_{y \in \R}[f_{-}(y) + \Ll(y,0;x,t)] = h_+(x,t) - h_-(x,t) < 0,
\end{align*}
where in the second equality, we used \eqref{eq:f_mont}, and in the inequality, we used \eqref{eq3}. Combining this with \eqref{ge0}, we see that 
\[
b_t^- = \inf\{x \in \R: d_{b_0^-}(x,t;f_{\mathsf V}) \ge 0\} = I_{b_0^-}^-(t,f_{\mathsf V}). \qedhere
\]
\end{proof}
The next lemma follows immediately from the definition of $\pi_{\mathrm{Sh};\pm}$ (recall \eqref{eq:piShdef}).
\begin{lemma} \label{lem:pi_constant}
    For $\theta > 0$ and $\underline f = (f_-,f_+) \in \mathcal X(\theta)$, we have  $\pi_{\mathrm{Sh};\pm}[\underline{g}]=\pi_{\mathrm{Sh};\pm}[\underline{g}+(c,c)]$
for any scalar constant $c\in\mathbb{R}$.
\end{lemma}

We next prove the following lemma, which gives a shift-invariance result for the shift maps.
\begin{lemma} \label{lem:shift_idempotent}
    Let $\theta > 0$, and let $\underline f = (f_-,f_+)$ be random initial data, independent of $\F_{\ge 0}$,  such that $\underline f \in \mathcal X(\theta)$ almost surely. Then, for $t > 0$,
    \[
    \Law\Bigl(\pi_{\mathrm{Sh};\pm}[ \underline h(\cdot,t \mid \underline f)]\Bigr)  = \Law\Bigl(\pi_{\mathrm{Sh};\pm}\bigl[\underline h(\cdot,t \mid \pi_{\mathrm{Sh};\pm}[\underline f]) \bigr] \Bigr),
    \]
    where the $\pm$ is used to indicate two separate distributional equalities. 
    \end{lemma}
\begin{proof}
     Let $b_0^{\pm} = \mathfrak b_\pm[\underline f]$ and $b_t^\pm = \mathfrak b_\pm[\underline h(\cdot,t \mid \underline f)]$. Then
     \be \label{eq:hshift_deq}
    \begin{aligned}
         &\underline h(\cdot,t\mid \pi_{\mathrm{Sh};\pm}[\underline f]) \\
         &\qquad= \Bigl(\sup_{y \in \R}[f_-(y + b_0^\pm) - f_-(b_0^\pm) + \Ll(y,0;\cdot,t)], \sup_{y \in \R}[f_+(y + b_0^\pm) - f_+(b_0^\pm) + \Ll(y,0;\cdot,t)]  \Bigr) \\
         &\qquad= \Bigl(\sup_{y \in \R}[f_-(y) + \Ll(y - b_0^\pm,0;\cdot,t)], \sup_{y \in \R}[f_+(y) + \Ll(y - b_0^\pm,0;\cdot,t)]  \Bigr) - (f_-(b_0^\pm),f_+(b_0^\pm)) \\
         &\qquad\deq \Bigl(\sup_{y \in \R}[f_-(y) + \Ll(y,0;\cdot + b_0^\pm,t)], \sup_{y \in \R}[f_+(y) + \Ll(y,0;\cdot+ b_0^\pm,t)]  \Bigr) - (f_-(b_0^\pm),f_+(b_0^\pm)) \\
         &\qquad= \underline h(\cdot + b_0^{\pm}, t \mid \underline f)- (f_-(b_0^\pm),f_+(b_0^\pm)),
    \end{aligned}
    \ee
    where the distributional equality follows from the independence of $f$ and $\Ll$, and the spatial stationarity in Lemma \ref{lm:landscape_symm}.  
    Recalling the definition of $\mathfrak b_{\pm}$ \eqref{eq:bdef}, we see that
    \[
     \mathfrak b_{\pm}[ \underline h(\cdot + b_0^{\pm}, t \mid \underline f)] = \mathfrak b_{\pm}[ \underline h(\cdot, t \mid \underline f)] - b_0^{\pm} = b_t^{\pm} - b_0^{\pm}, 
    \]
    which implies that 
    \[
\pi_{\mathrm{Sh};\pm}\bigl[\underline h(\cdot + b_0^{\pm},t\mid \underline f)\bigr] = \underline h(\cdot + b_t^{\pm} - b_0^{\pm} +  b_0^{\pm},t\mid \underline f)-\underline h(b_t^{\pm} - b_0^{\pm} +  b_0^{\pm},t\mid \underline f) = \pi_{\mathrm{Sh};\pm}[ \underline h(\cdot,t \mid \underline f)].
    \]
    Then, since $f_-(b_0^{\pm}) = f_+(b_0^{\pm})$ by definition of $b_0^\pm$, the statement of the proposition now follows from Lemma \ref{lem:pi_constant} and the distributional equality in \eqref{eq:hshift_deq}.
\end{proof}

We now record the following observation. This was stated as an equality in distribution in \cite{KPZfixed}, and follows immediately from the variational description from the directed landscape.
\begin{lemma} \label{lem:KPZFP_max}
    Let $f_1,f_2 \in \mathcal C_{\mathrm{FP}}$. Then, for all $t > 0$ and $x \in \R$,
    \[
    h(x,t \mid f_1 \vee f_2) = h(x,t \mid f_1) \vee h(x,t \mid f_2).
    \]
\end{lemma}

The following proposition, which is analogous to \cite[Proposition~2.1]{Dunlap-Sorensen-24} in the positive temperature case, shows that appropriately recentered versions of the KPZ dynamics are Markovian and have the Feller property.
\begin{proposition} \label{prop:Markov_feller}
Let $N\in\mathbb{N}$ and let $g\colon\mathcal{C}_{\mathrm{FP}}^{N}\to\mathbb{R}^{N}$
be a continuous linear map such that $g[x\mapsto g[\underline{f}]]= g[\underline{f}]$
for all $\underline{f}\in\mathcal{C}_{\mathrm{FP}}^{N}$ \rm{(}Here, $x\mapsto g[\underline{f}]$
denotes the constant function with value $g[\underline{f}]$\rm{)}. Define
$\pi\colon\mathcal{C}_{\mathrm{FP}}^{N}\to\mathcal{C}_{\mathrm{FP}}^{N}$
by $\pi[\underline{f}](x)=\underline{f}(x)-g[\underline{f}]$.
\begin{enumerate} [label=\rm(\roman{*}), ref=\rm(\roman{*})]  \itemsep=3pt
\item For any $s\in\mathbb{R}$ and any $\underline{f}\in\mathcal{C}_{\mathrm{FP}}^{N}$
independent of $\mathcal{F}_{\ge s}$, the process $(\pi[h(\cdot,t\mid s,\underline{f})])_{t\ge s}$
is a Markov process with state space $\mathcal{C}_{\mathrm{FP}}^{N}$.
\item For $F\in\mathcal{C}_{\mathrm{b}}(\mathcal{C}_{\mathrm{FP}}^{N})$,
$t\ge0$, and $\underline{f}\in\mathcal{C}_{\mathrm{FP}}^{N}$, let
$P_{t}^{\pi}F(\underline{f})\coloneqq\mathbb{E}[F(\pi[\underline{h}(\cdot,t\mid0,\underline{f})])]$.
Then the Markov semigroup $(P_{t}^{\pi})_{t\ge0}$ has the Feller
property.
\end{enumerate}
\end{proposition}

\begin{proof}
Without loss of generality, we can assume in the proof of the first
part that $s=0$. Let $\underline{h}(x,t)=\underline{h}(x,t \mid0,\underline{f})$.
For $0\le s<t$ and any $x\in\mathbb{R}$, we have
\begin{align*}
\underline{h}(x,t) & =\max_{y\in\mathbb{R}}\left\{ \underline{h}(y,s)+\mathcal{L}(y,s;x,t)\right\} \\
 & =\max_{y\in\mathbb{R}}\left\{ \pi[\underline{h}(\cdot,s)](y)+g[\underline{h}(\cdot,s)]+\mathcal{L}(y,s;x,t)\right\} \\
 & =\max_{y\in\mathbb{R}}\left\{ \pi[\underline{h}(\cdot,s)](y)+\mathcal{L}(y,s;x,t)\right\} +g[\underline{h}(\cdot,s)].
\end{align*}
Therefore, we have
\begin{align*}
\pi[\underline{h}(\cdot,t)] & =\pi\left[x\mapsto\max_{y\in\mathbb{R}}\left\{ \pi[\underline{h}(\cdot,s)](y)+\mathcal{L}(y,s;x,t)\right\} +g[\underline{h}(\cdot,s)]\right]\\
 & =\pi\left[x\mapsto\max_{y\in\mathbb{R}}\left\{ \pi[\underline{h}(\cdot,s)](y)+\mathcal{L}(y,s;x,t)\right\} \right].
\end{align*}
Hence we see that $\pi[\underline{h}(\cdot,t)]$ depends only on $\pi[\underline{h}(\cdot,s)]$
and $\mathcal{L}(\cdot,s;\cdot,t)$, which implies that $(\pi[\underline{h}(\cdot,t)])_{t\ge0}$
is indeed a Markov process by the independent temporal increments property of $\Ll$ (Definition \ref{def:sigma_alg}).

The fact that the semigroup $(P_{t}^{\pi})_{t\ge0}$ has the Feller
property is a consequence of the Feller property for $P_{t}$ proved in Proposition~\ref{prop:Feller}.
and the assumed continuity of $g$.
\end{proof}

\subsection{Stationary horizon}
Here, we recall the definition of the stationary horizon from \cite{Busani-2021,Seppalainen-Sorensen-21b}. Define the map $\Phi$ by
\be \label{Phialt}
\Phi(f,g)(y) = g(y) + \sup_{-\infty <x \le y }\{f(y) - f(x) - [g(y) - g(x)]\} - \sup_{-\infty < x \le 0}\{g(x) - f(x)\}.
\ee
This map extends to maps $\Phi^k:\mathcal C(\R)^k \to \mathcal C(\R)^k$ as follows: \begin{enumerate}
    \item $\Phi^1(f_1)(x) = f_1(x)$. 
    \item $\Phi^2(f_1,f_2)(x) = [f_1(x),\Phi(f_1,f_2)(x)]$.
    \item For $k\ge 3$, \[\Phi^k(f_1,\ldots,f_k)(x) = [f_1(x),\Phi(f_1,[\Phi^{k - 1}(f_2,\ldots,f_k)]_1)(x),\ldots,\Phi(f_1,[\Phi^{k -1}(f_2,\ldots,f_k)]_{k - 1})(x)].\]
\end{enumerate}

\begin{definition} \label{def:SH}
The stationary horizon $\{G_\theta\}_{\theta \in \R}$ is a process with state space $\mathcal C(\R)$ and with paths in the Skorokhod space $D(\R,\mathcal C(\R))$ of right-continuous functions $\R \to \mathcal C(\R)$ with left limits. The space $\mathcal C(\R)$ has the Polish topology of uniform convergence on compact sets. The law of the stationary horizon is characterized as follows: for real numbers $\theta_1 < \cdots < \theta_k$, the $k$-tuple  $(G_{\theta_1},\ldots,G_{\theta_k})$ of continuous function  has the same law as $\Phi^k(f_1,\ldots,f_k)$, where $f_1,\ldots,f_k$ are independent, and the $i$th component has law $\BM(2\theta_i,\sqrt 2)$. In particular, $G_\theta \sim \BM(2\theta,\sqrt 2)$ for each $\theta \in \R$ \cite{brownian_queues}.    
\end{definition}

\begin{definition} \label{nutheta_def}
    For $\theta > 0$, we define $\nu_\theta := \Law(G_{-\theta},G_{\theta})$, where $G$ is the stationary horizon.  
\end{definition}

We now cite a result from \cite{Busa-Sepp-Sore-22a}:
\begin{proposition} \cite[Theorem~2.1, Theorem~5.1(viii)]{Busa-Sepp-Sore-22a} \label{prop:invariance_of_SH}
Let $(\Omega,\F,\Pp)$ be a  probability space on which the stationary horizon $G=\{G_\theta\}_{\theta \in \R}$ and directed landscape $\Ll$ are defined, and such that  the processes $\{\Ll(y,0;x,t):x,y \in \R, t > 0\}$ and $G$ are independent. For each $\theta \in \R$, let $G_\theta$ evolve under the KPZ fixed point in the same environment $\Ll$, i.e., for each $\theta \in \R$,
\[
h(x,t \mid G_\theta) = \sup_{y \in \R}\{G_\theta(y) + \Ll(y,0;x,t)\},\qquad\text{for all } x\in\R \text{ and } t > 0.
\]
Then the following statements hold:
  \paragraph{(Invariance)}
For each $t > 0$,   the equality in distribution  $\bigl(h(\cdot,t\mid G_\theta) - h(\cdot,t\mid G_\theta)\bigr)_{\theta \in \R}$ $\deq G$  holds between  random elements of $D(\R,\mathcal C(\R))$.

\paragraph{(Attractiveness)}
There exists a process 
\[
\Bigl(W^{\theta}(y,s;x,t): (y,s;x,t) \in \R^4\Bigr)
\] 
measurable with respect to $\Ll$,
such that, for each $t \in \R$, the law of $(W^\theta(0,t;\cdot,t))_{\theta \in \R}$ is the law of the stationary horizon, and moreover that the following holds. 
Let $k \in \N$ and $\theta_1 < \cdots < \theta_k$ in $\R$. Let $(f_1,\ldots,f_k)$ be a $k$-tuple of upper-semicontinuous functions $\R \to \R \cup\{-\infty\}$ that are not equal to $-\infty$ everywhere, coupled with $\Ll$ {\rm arbitrarily},  and that almost surely satisfy \eqref{eqn:drift_assumptions} for $(f, \theta) = (f_i, \theta_i)$  for each  $i\in\{1,\dotsc,k\}$.  Then, with probability $1$, for every $C > 0$ and $t \in \R$, there exists $S < t$ such that, for all $s < S, x \in [-C,C]$, and $i \in \{1,\ldots,k\}$,
\[
h(x,t \mid s, f_i) - h(0,t \mid s, f_i) = W^{\theta_i}(0,t;x,t).
\]
\end{proposition}
The process $W^\theta$ in Proposition~\ref{prop:invariance_of_SH} is known as the \textit{Busemann process} and is discussed more in Proposition~\ref{prop:Buse_basic_properties}. To match the notation, the process $W^\theta$ is the process $W^{\theta +}$ from Proposition~\ref{prop:Buse_basic_properties}. 

 In the following, we upgrade the convergence in  Proposition~\ref{prop:invariance_of_SH} to convergence with respect to the topology on $\C_{\mathrm{FP}}$.  We prove Proposition~\ref{prop:uniform_upbd} in Section~\ref{sec:converge_to_SH}.
 \begin{proposition} \label{prop:uniform_upbd}
 Let $\theta \in \R$, and assume $f$ is an upper semi-continuous function $\R \to \R \cup \{-\infty\}$ that is not equal to $-\infty$ everywhere and satisfies \eqref{eqn:drift_assumptions}. Then, for each $t \in \R$, with probability $1$,  there exist random $A,B,S > 0$ (depending on $t$ and $f$) such that, for all $x \in \R$ and all $s < S$,
 \be \label{eq:linear_upbd}
 h(x,t \mid s, f) - h(0,t \mid s,f) \le A + B|x|
 \ee
In particular, combined with Proposition~\ref{prop:invariance_of_SH}, for $f \in \CFP$ satisfying \eqref{eqn:drift_assumptions},  as $s \to -\infty$, the process
\[
x \mapsto h(x,t \mid s, f) - h(0,t \mid s,f)
\]
converges, with respect to the topology on $\CFP$, to $x \mapsto W^\theta(0,t;x,t)$. Consequently,
for each $\theta_1 < \cdots < \theta_k$,  the law of $\bigl(G_{\theta_1}, \dotsc, G_{\theta_k})$ is the unique invariant measure of the KPZ fixed point on the space $\C_{\mathrm{FP}}^k$,  such that for each  $i\in\{1,\dotsc,k\}$  the condition~\eqref{eqn:drift_assumptions} holds for $(f, \theta) = (G_{\theta_i}, \theta_i)$ almost surely.
 \end{proposition}

\section{Invariant measure from the perspective of a shock in the full-line and open settings} \label{sec:inv_meas_shock_proofs}

As in \cite{Dunlap-Ryzhik-2020,Dunlap-Sorensen-24}, the building
blocks of shock profiles are jointly invariant measures for the KPZ
dynamics with different asymptotic slopes. 
Let
\[
\mathcal{M}f(x)=\sup_{y\in(-\infty,x]}f(y),
\]
and define
\begin{equation}
\Phi^{\diamond}[\overline{W},\widetilde{W}](x)\coloneqq\left(\overline{W}(x)-\widetilde{W}(x),\overline{W}(x)+2\mathcal{M}\widetilde{W}(x)-2\mathcal{M}\widetilde{W}(0)-\widetilde{W}(x)\right).\label{eq:Phidiamonddef}
\end{equation}

\begin{proposition}
\label{prop:nuthetainvariant}Let $(\overline{W},\widetilde{W})\sim\BM(0)\otimes\BM(2\theta)$
and $\underline{f}=(f_{-},f_{+})=\Phi^{\diamond}[\overline{W},\widetilde{W}]$.
Then, $
\nu_{\theta}= \Law(\underline{f})\label{eq:nuthetadef-1}$,
and for each (deterministic) $t\ge0$, we have $\pi_{0}[\underline{h}(\cdot,t\mid0,\underline{f})]\sim\nu_{\theta}$
as well.
\end{proposition}

\begin{proof}
This is a slight reformulation of Proposition~\ref{prop:invariance_of_SH}:
the measure we consider is (a particular case of) the same one considered
there, but we are constructing it in a slightly different way, amounting
to a change of coordinates in the parametrization of $\Phi^{\diamond}$.
We recall, for paths $W_{1},W_{2}\colon\mathbb{R}\to\mathbb{R}$ such
that $W_{1}(0)=W_{2}(0)=0$, the map $\Phi$ from \eqref{Phialt}, which can be equivalently expressed as
\[
\Phi[W_{1},W_{2}](x)=\left(W_{1}(x),W_{1}(x)+\mathcal{M}[W_{2}-W_{1}](x)-\mathcal{M}[W_{2}-W_{1}](0)\right).
\] 
We can define
\[
W_{1}=\overline{W}-\widetilde{W},\qquad W_{2}=\overline{W}+\widetilde{W},\qquad\text{so}\qquad\overline{W}=\frac{W_{1}+W_{2}}{2},\qquad\widetilde{W}=\frac{W_{2}-W_{1}}{2}.
\]
Moreover, $W_{1} \sim\BM(-2\theta,\sqrt 2)$ and $W_{2} \sim \BM(2\theta,\sqrt 2)$ are independent. It
is straightforward to check that
\[
\Phi[W_{1},W_{2}]=\Phi^{\diamond}[\overline{W},\widetilde{W}]=\underline{f}.
\]
Then, the conclusion of the proposition follows from the invariance statement of
Proposition~\ref{prop:invariance_of_SH}.
\end{proof}

\subsection{The shock frame}

We need to understand how the shock frame interacts with $\Phi^{\diamond}$.
From the definitions \eqref{eq:Phidiamonddef} and \eqref{eq:bdef},
we see, for $\zeta\ge0$, that

\begin{align}
\mathfrak{b}_{-}[\Phi^{\diamond}[\overline{W},\widetilde{W}]+(\zeta,0)] & =\min\left\{ x\in\mathbb{R}\st\mathcal{M}\widetilde{W}(x)=\mathcal{M}\widetilde{W}(0)+\zeta/2\right\} \label{eq:b-Phidiamond-withM}\\
 & =\min\left\{ x\in\mathbb{R}\st\widetilde{W}(x)=\mathcal{M}\widetilde{W}(0)+\zeta/2\right\} \label{eq:b-Phidiamond}
\end{align}
 and also
\begin{align}
\mathfrak{b}_{+}[\Phi^{\diamond}[\overline{W},\widetilde{W}]+(\zeta,0)] & =\max\left\{ x\in\mathbb{R}\st\mathcal{M}\widetilde{W}(x)=\mathcal{M}\widetilde{W}(0)+\zeta/2\right\} .\label{eq:b+Phidiamond}
\end{align}

\subsection{The shock-frame invariant measure}

Let $\overline{Z}\sim\BM$. We also construct a process $\widetilde{Z}$,
independent of $\overline{Z}$, which is informally given by a Brownian
motion with drift $2\theta$ and $\widetilde{Z}(0)=0$ \emph{conditioned}
such that $\widetilde{Z}(x)\le0$ for all $x\le0$. This is not a
rigorous description, of course, because the conditioning is singular,
but we interpret it in the usual way via the Bessel process. To be
precise, let $\widetilde{Z}_{+}\sim\BM(2\theta)$, and let $\widetilde{Z}_{-}\sim\BES^{3}(2\theta)$
be independent of each other and also of $\overline{Z}$. Then we
define
\begin{equation}
\widetilde{Z}(x)=\begin{cases}
-\widetilde{Z}_{-}(-x), & x\le0;\\
\widetilde{Z}_{+}(x), & x\ge0.
\end{cases}\label{eq:Ztildedef}
\end{equation}
Now we define
\begin{equation}
\widehat{\nu}_{\theta}\coloneqq\Law(\Phi^{\diamond}[\overline{Z},\widetilde{Z}]).\label{eq:nuthetahatdef}
\end{equation}
We note that the form of $\Phi^{\diamond}$ simplifies somewhat in
this case because $\sup\limits_{y\le0}\widetilde{Z}(y)=0$: for these
choices of $\overline{Z},\widetilde{Z}$, the definition~\eqref{eq:Phidiamonddef}
reduces to
\begin{equation}
\Phi^{\diamond}[\overline{Z},\widetilde{Z}]=\left(\overline{Z}-\widetilde{Z},\overline{Z}+2\mathcal{M}\widetilde{Z}-\widetilde{Z}\right),\qquad\text{almost surely.}\label{eq:Phidiamondfornuhat}
\end{equation}
Define
\begin{equation}
\mathcal{M}_{0}f(x)=\sup_{y\in[0,x]}f(y)\qquad\text{for }x\ge0.\label{eq:M0def}
\end{equation}
We note that the process $\left(2\mathcal{M}\widetilde{Z}(x)-\widetilde{Z}(x)\right)_{x\ge0}=\left(2\mathcal{M}_{0}\widetilde{Z}(x)-\widetilde{Z}(x)\right)_{x\ge0}$
appearing on the right side of \eqref{eq:Phidiamondfornuhat}
is distributed like a $\BES^{3}(2\theta)$ process on the right half-line
by Proposition~\ref{prop:2M-X}\ref{enu:2M-X-fwd}. On the other
hand, the process $\left(2\mathcal{M}\widetilde{Z}(x)-\widetilde{Z}(x)\right)_{x\le0}$
is distributed like a $\BM(2\theta)$ process on the left half-line
by Proposition~\ref{prop:2M-X}\ref{enu:2M-x-backward}.
\begin{lemma} \label{nutheta_fact}
    For $\theta > 0$, let $(f_-,f_+) \sim \widehat \nu_\theta$. Then, with probability $1$, $(f_-,f_+) \in \mathcal X(\theta)$, and $x = 0$ is the unique value of $x$ such that $f_-(x) = f_+(x)$. Furthermore, the function $f_{\mathsf V}(x) := f_-(x) \vee f_+(x)$ has the law of $B + 
    Z$, where $B \sim \BM$ and $Z \sim \BES^3(2\theta)$ are independent. 
\end{lemma}
\begin{proof}
    By definition of the law $\widehat \nu_\theta$ given above, we may write 
    \[
    (f_-,f_+) = \left(\overline{Z}-\widetilde{Z},\overline{Z}+2\mathcal{M}\widetilde{Z}-\widetilde{Z}\right),
    \]
    where $\overline Z \sim \BM$, independent of the process $\wt Z$, which is described above in \eqref{eq:Ztildedef}. Since $\bigl(\wt Z(-x)\bigr)_{x \ge 0}$ is the negative of a $\BES^3(2\theta)$ process on the right half-line, and $\bigl(\wt Z(x)\bigr)_{x \ge 0}$ is a Brownian motion with drift, we have $\mathcal M \wt Z(x) = 0$ if and only if $x = 0$. Hence, $x = 0$ is the unique value of $x$ where $f_-(x) = f_+(x)$, as desired. Furthermore, the process $x \mapsto \mathcal M \wt Z(x)$ is nondecreasing, so $\mathcal M \wt Z(x) = \mathcal M_0 \wt Z(x)$ for $x \ge 0$, and $f_{\mathsf V} = B + Z$, where 
\[
B(x) = \overline Z(x)\quad \text{and}\quad  Z(x) = \begin{cases}
    -\wt Z(x), &x \le 0 \\
    2\mathcal{M}_{0}\widetilde{Z}(x)-\widetilde{Z}(x), &x \ge 0.
    \end{cases}
    \]
 Because $\overline Z$ and $\wt Z$ are independent, we see that $B$ and $Z$ are independent. It remains to show that $Z \sim \BES^3(2\theta)$. By definition of $\wt Z$, the processes $(\wt Z(x))_{x \le 0}$ and $(\wt Z(x))_{x \ge  0}$ are independent, so $(-\wt Z(x))_{x \le 0}$ is independent of $ \bigl(2\mathcal{M}_{0}\widetilde{Z}(x)-\widetilde{Z}(x)\bigr)_{x \ge 0}$. The process $(-\wt Z(x))_{x \le 0}$ is distributed as a $\BES^3(2\theta)$ process on the left-hand line by definition, and as noted above the statement of the lemma, the process $\bigl(2\mathcal{M}_{0}\widetilde{Z}(x)-\widetilde{Z}(x)\bigr)_{x \ge 0}$
is distributed as a $\BES^{3}(2\theta)$ process on the right half-line. Hence, $Z\sim \BES^3(2\theta)$, as desired.  
\end{proof}

\begin{lemma}
\label{lem:bplusminussame}Let $\underline{f}\sim\nu_{\theta}$. With
probability $1$, for all but countably many $\zeta\in\R$, we have
$\mathfrak{b}_{-}[\underline{f}+(\zeta,0)]=\mathfrak{b}_{+}[\underline{f}+(\zeta,0)]$.
\end{lemma}
\begin{proof}
By \eqref{eq:nuthetadef-1}, we write $\underline{f}=\Phi^{\diamond}[\overline{W},\widetilde{W}]$
with $(\overline{W},\widetilde{W})\sim\BM(0)\otimes\BM(2\theta)$.
By \eqref{eq:b-Phidiamond-withM} and \eqref{eq:b+Phidiamond}
along with the fact that $\mathcal{M}\widetilde{W}$ is nondecreasing,
we see that the intervals $I_{\zeta}\coloneqq\left(\mathfrak{b}_{-}[\underline{f}+(\zeta,0)],\mathfrak{b}_{+}[\underline{f}+(\zeta,0)]\right)$
must be pairwise disjoint across all $\zeta$. Hence $I_{\zeta}$
can be nonempty for at most countably many $\zeta$ with probability
$1$, and this completes the proof.
\end{proof}
The following lemma is an analogue of \cite[Lemma~2.14]{Dunlap-Sorensen-24},
which in turn was based on an argument used in the proof of \cite[Proposition~4.4]{Dunlap-Ryzhik-2020}.
That lemma was proved using ergodicity properties. In the present
setting, the measure $\widehat{\nu}_{\theta}$ is not absolutely continuous
with respect to $\nu_{\theta}$. To see this, Lemma \ref{nutheta_fact} states that the measure $\widehat{\nu_\theta}$ is supported on tuples of functions $(f_-,f_+)$ such that $f_-(x) = f_+(x)$ if and only if $x = 0$, while from the description of the measure $\nu_\theta$ in Definition \ref{nutheta_def}, we can see that $\nu_\theta$ is supported on tuples of functions $(f_-,f_+)$ such that $f_-(x) = f_+(x)$ in a nondegenerate interval around the origin. Because of this singularity, we must use a different strategy to show convergence, based
on a Brownian path decomposition results of Rogers and Pitman \cite{Rogers-Pitman-81}, recalled in the present paper
as Proposition~\ref{prop:besselrev}.
\begin{lemma}
\label{lem:convergetonuhat}Suppose that $\underline{f}=(f_{-},f_{+})\sim\nu_{\theta}$.
For $L\in(0,\infty)$, let $\zeta_{L}\sim\Uniform([0,L])$ be independent
of everything else. Then we have
\begin{equation}
\Law\left(\pi_{\mathrm{Sh};\pm}[\underline{f}+(\zeta_{L},0)]\right)\xrightarrow[L\to\infty]{}\widehat{\nu}_{\theta}\label{eq:pishonbigintervalconvinlawtonuhat}
\end{equation}
weakly, with respect to the topology on $\CFP$.
\end{lemma}

\begin{proof}
By Lemma~\ref{lem:bplusminussame}, it suffices to show the statement for the shift $\pi_{\mathrm{Sh};-}[\underline{f}+(\zeta_{L},0)]$. By Proposition~\ref{prop:nuthetainvariant}, we can represent $\underline f$ as 
\begin{equation}
\underline{f}=\Phi^{\diamond}[\overline{Q},\widetilde{Q}],\label{eq:funderlinedef}
\end{equation}
where $(\overline{Q},\widetilde{Q})\sim\mathrm{BM}(0)\otimes\BM(2\theta)$.
For $L > 0$, let $\xi_L$ be a random variable such that 
\begin{equation}
\mathbb{P}(\xi_L>0)=1,\text{ and }\xi_L\text{ is independent of }(\widetilde{Q}(x))_{x\ge0}.\label{eq:xiind}
\end{equation}
We will specifically define $\xi_L$ later in the proof. Given this random variable, define
the stopping time
\begin{equation}
\tau_{\xi_L}=\min\{x\in\mathbb{R}\st\widetilde{Q}(x)=\xi_L\},\label{eq:tauxidef}
\end{equation}
and the process
\begin{equation}
x \mapsto \widetilde{Q}_{\xi_L}(x)\coloneqq\widetilde{Q}(x+\tau_{\xi_L})-\widetilde{Q}(\tau_{\xi_L})=\widetilde{Q}(x+\tau_{\xi_L})-\xi_L.\label{eq:qtildexidef}
\end{equation}
Since $(\widetilde{Q}(x))_{x\ge0}$ has drift $2\theta>0$, we see
that $\tau_{\xi_L}<\infty$ with probability $1$. Also, since $\tau_{\xi_L}$
is a stopping time with respect to the usual filtration of $(\widetilde{Q}(x))_{x\ge0}$,
the strong Markov property tells us that the process $(\widetilde{Q}_{\xi_L}(x))_{x\ge0}$
is another standard Brownian motion with drift $2\theta$, which moreover
is independent of $(\tau_{\xi_L},(\widetilde{Q}(x))_{0\le x\le\tau_{\xi_L}})$.
Now by Proposition~\ref{prop:besselrev}, in an appropriate coupling, we have a random variable $\widetilde{R}_{-}\sim\BES^{3}(2\theta)$
such that, if we define the last-passage time
\[
\sigma_{\xi_L}\coloneqq\sup\{x\ge0\st\widetilde{R}_{-}(x)=\xi_L\},
\]
then 
\[
\sigma_{\xi_L}=\tau_{\xi_L}\qquad\text{and}\qquad\widetilde{Q}_{\xi_L}(x)=-\widetilde{R}_{-}(-x)\qquad\text{for all }x\in[-\tau_{\xi_L},0].
\]
Moreover, we can (and do) choose $\widetilde{R}_{-}$ to be independent
of $(\widetilde{Q}_{\xi_L}(x))_{x\ge0}$. Next, define
\be \label{Ztilde_def}
\widetilde{Z}_{\xi_L}(x)=\begin{cases}
-\widetilde{R}_{-}(-x) & x\le0, \\
\widetilde{Q}_{\xi_L}(x) & x\ge0,
\end{cases}
\ee
so that
\begin{equation}
\widetilde{Z}_{\xi_L}\deq \widetilde{Z},\label{eq:ZtildeZ}
\end{equation}
with $\widetilde{Z}$ defined in \eqref{eq:Ztildedef}, and
\begin{equation}
\widetilde{Z}_{\xi_L}|_{[-\sigma_{\xi_L},\infty)}=\widetilde{Q}_{\xi_L}|_{[-\tau_{\xi_L},\infty)},\qquad\text{almost surely.}\label{eq:ZtildematchesQtilde}
\end{equation}
Now, as alluded to previously, we define the random variable $\xi_L$ as
\begin{equation}
\xi_L\coloneqq\sup_{y\in(-\infty,0]}\widetilde{Q}(y)+\zeta_{L}/2.\label{eq:xichoice}
\end{equation}
This choice of $\xi_L$ clearly satisfies the assumption \eqref{eq:xiind}.
Moreover, comparing \eqref{eq:b-Phidiamond} and \eqref{eq:tauxidef},
we see that
\begin{equation}
\mathfrak{b}_{-}[\underline{f}+(\zeta_{L},0)]\overset{\eqref{eq:funderlinedef}}{=}\mathfrak{b}_{-}[\Phi^{\diamond}[\overline{Q},\widetilde{Q}]+(\zeta_{L},0)]\overset{\eqref{eq:b-Phidiamond}}{=}\tau_{\xi_L}.\label{eq:bintermsoftau}
\end{equation}
We note that $\xi_L\to\infty$ in probability as $L\to\infty$, and
therefore that
\begin{equation}
\tau_{\xi_L}\to\infty\text{ in probability as }L\to\infty.\label{eq:tauxitoinfty}
\end{equation}
Analogously to \eqref{eq:qtildexidef}, define
\[
\overline{Q}_{\xi_L}(x)=\overline{Q}(x+\tau_{\xi_L})-\overline{Q}(\tau_{\xi_L}),
\]
so $\overline{Q}_{\xi_L}\sim\BM$ by the strong Markov property, and so, by \eqref{eq:ZtildeZ} (recalling the definition \eqref{eq:nuthetahatdef}), we have
\begin{equation}
\Phi^{\diamond}[\overline{Q}_{\xi_L},\widetilde{Z}_{\xi_L}]\sim\widehat{\nu}_{\theta}.\label{eq:Phidiamonddistnuhat}
\end{equation}
 Recalling \eqref{eq:qtildexidef}, we can compute
\be \label{eq:MxiL_tilde}
\mathcal{M}\widetilde{Q}(x+\tau_{\xi_L})  =\mathcal{M}[\widetilde{Q}(\cdot+\tau_{\xi_L})](x)=\mathcal{M}\widetilde{Q}_{\xi_L}(x)+\xi_L,
\ee
and hence we have
\begin{align}
  \pi_{\mathrm{Sh};-}[\underline{f}+(\zeta_{L},0)](x)\overset{\eqref{eq:bintermsoftau}}&{=}\underline{f}(x + \tau_{\xi_L})-\underline{f}(\tau_{\xi_L})\nonumber \\
 & =\Phi^{\diamond}[\overline{Q},\widetilde{Q}](x+\tau_{\xi_L})-\Phi^{\diamond}[\overline{Q},\widetilde{Q}](\tau_{\xi_L})\nonumber \\
 &=\left(\overline{Q}_{\xi_L}(x)-\widetilde{Q}_{\xi_L}(x),\overline{Q}_{\xi_L}(x)-\widetilde{Q}_{\xi_L}(x)+2\mathcal{M}\widetilde{Q}(x+\tau_{\xi_L})-2\mathcal{M}\widetilde{Q}(\tau_{\xi_L})\right)\nonumber \\
 \overset{\eqref{eq:MxiL_tilde}}&{=}\left(\overline{Q}_{\xi_L}(x)-\widetilde{Q}_{\xi_L}(x),\overline{Q}_{\xi_L}(x)-\widetilde{Q}_{\xi_L}(x)+2\mathcal{M}\widetilde{Q}_{\xi_L}(x)\right),\label{eq:fintermsofQs}
\end{align}
where in the last step, we have used the fact that $\mathcal M \wt Q_{\xi_L}(0) = 0$, which follows because $\wt Q_{\xi_L}(0) = 0$, and $\wt Q_{\tau_{\xi_L}} = \wt Q(x + \tau_{\xi_L}) -\wt Q(\tau_{\xi_L}) < 0 $ for $x < 0$, since $\tau_{\xi_L}$ is a first-passage time. For $x\in\mathbb{R}$, we define the event
\begin{align}
E_{x}^L & \coloneqq\left\{ \tau_{\xi_L}\ge-x,\ \sup_{y\in(-\infty,0]}\widetilde{Q}(y)\le\sup_{y\in[0,x + \tau_{\xi_L}]}\widetilde{Q}(y),\text{ and }\sup_{y\in(-\infty,-\tau_{\xi_L}]}\widetilde{Z}_{\xi_L}(y)\le\sup_{y\in[-\tau_{\xi_L},x]}\widetilde{Z}_{\xi_L}(y)\right\} .\label{eq:condition}
\end{align}
On the event $E_{x}^L$, we have (recalling the definition \eqref{eq:qtildexidef})
that
\[
\sup_{y\in(-\infty,-\tau_{\xi_L}]}\widetilde{Q}_{\xi_L}(y)\le\sup_{y\in[-\tau_{\xi_L},x]}\widetilde{Q}_{\xi_L}(y),
\]
and hence
\begin{equation}
\mathcal{M}\widetilde{Q}_{\xi_L}(x)=\sup_{y\in[-\tau_{\xi_L},x]}\widetilde{Q}_{\xi_L}(y)\overset{\eqref{eq:ZtildematchesQtilde}}{=}\sup_{y\in[-\tau_{\xi_L},x]}\widetilde{Z}_{\xi_L}(y)=\mathcal{M}\widetilde{Z}_{\xi_L}(x),\label{eq:supQtildeintermsofZtilde}
\end{equation}
with the last identity holding by the last condition in the definition
\eqref{eq:condition} of $E_{x}^L$. We claim that this implies
that
\begin{equation}
E_{x}^L\subseteq\left\{ \pi_{\mathrm{Sh};-}[\underline{f}+(\zeta_{L},0)](x)=\Phi^{\diamond}[\overline{Q}_{\xi_L},\widetilde{Z}_{\xi_L}](x)\right\} .\label{eq:EximpliesfisPhi}
\end{equation}
Indeed, on the event $E_{x}^L$, we can write
\begin{align*}
\Phi^{\diamond} & [\overline{Q}_{\xi_L},\widetilde{Z}_{\xi_L}](x)\overset{\eqref{eq:Phidiamondfornuhat}}{=}\left(\overline{Q}_{\xi_L}(x)-\widetilde{Z}_{\xi_L}(x),\overline{Q}_{\xi_L}(x)+2\mathcal{M}\widetilde{Z}_{\xi_L}(x)-\widetilde{Z}_{\xi_L}(x)\right)\\
 & =\left(\overline{Q}_{\xi_L}(x)-\widetilde{Q}_{\xi_L}(x),\overline{Q}_{\xi_L}(x)+2\mathcal{M}\widetilde{Q}_{\xi_L}(x)-\widetilde{Q}_{\xi_L}(x)\right)\overset{\eqref{eq:fintermsofQs}}{=}\pi_{\mathrm{Sh};-}[\underline{f}+(\zeta_{L},0)](x),
\end{align*}
where in the second identity we use \eqref{eq:ZtildematchesQtilde}
(along with the fact that $x\ge-\tau_{\xi_L}$ on $E_{x}^L$ by the definition
\eqref{eq:condition}), and \eqref{eq:supQtildeintermsofZtilde}.

We also note from the definition \eqref{eq:condition} of $E_{x}^L$
that
\begin{equation}
x'\ge x\implies E_{x}^L\subseteq E_{x'}^L,\label{eq:Eincreasing}
\end{equation}
and that, for any fixed $x\in\mathbb{R}$, we have
\begin{equation}
\lim_{L\to\infty}\mathbb{P}(E_{x}^L)=1\label{eq:limPEx}
\end{equation}
by \eqref{eq:tauxitoinfty}. (Recall that the event $E_{x}^L$ depends
on $L$ through the dependence \eqref{eq:xichoice} of $\xi_L$
on $L$.) Hence, for any fixed $K>0$, we have
\be \label{eq:ExL_sets}
\begin{aligned}
  \liminf_{L\to\infty} {} & \mathbb{P}\left(\pi_{\mathrm{Sh};-}[\underline{f}+(\zeta_{L},0)]\big|_{[-K,K]}=\Phi^{\diamond}[\overline{Q}_{\xi_L},\widetilde{Z}_{\xi_L}]\big|_{[-K,K]}\right)\\
 & \overset{\eqref{eq:EximpliesfisPhi}}{\ge}\liminf_{L\to\infty}\mathbb{P}\left[\bigcap_{x\in[-K,K]}E_{x}^L\right]\overset{\eqref{eq:Eincreasing}}{=}\liminf_{L\to\infty}\mathbb{P}(E_{-K}^L)\overset{\eqref{eq:limPEx}}{=}1,
\end{aligned}
\ee
and in light of \eqref{eq:Phidiamonddistnuhat} this implies the convergence  in \eqref{eq:pishonbigintervalconvinlawtonuhat}, with respect to the topology of uniform convergence on compact sets. We now upgrade to convergence on the space $\CFP$. For $C > 0$, we define the event 
\be \label{ACL}
A_{C}^L = \Bigl\{\overline Q_{\xi_L}(x) \le C + |x|, \quad  |\wt Q_{\xi_L}(x) - 2\theta x|\le C + \theta|x|,\quad \forall x \in \R     \Bigr\}
\ee
We know that $\overline Q_{\xi_L} \sim \BM$, so,  with arbitrarily high probability, by taking $C$ large enough,
\be \label{overline Q small}
\overline Q_{\xi_L}(x) \le C + |x|,\quad\forall x \in \R. 
\ee
Furthermore, from \eqref{Ztilde_def} and \eqref{eq:ZtildematchesQtilde}, since $\BM(2\theta)$ and $\BES^3(2\theta)$ both have asymptotic slope $2\theta$, we have that, with arbitrarily high probability, by taking $C$ large enough,
\be \label{eq:QxiL}
|\wt Q_{\xi_L}(x) - 2\theta x| \le C + \theta|x| ,\quad\forall  x \ge -\tau_{\xi_L}.
\ee
On the other hand, since $\wt Q_{\xi_L}(x) \overset{\eqref{eq:qtildexidef}}= \wt Q(x + \tau_{\xi_L}) - \wt Q(\tau_\xi)$ and $\wt Q \sim \BM(2\theta)$, for every $\ve > 0$, we may obtain 
\be \label{eq:Qbd2}
|\wt Q_{\xi_L}(x) - 2\theta x| \le |\wt Q(x + \tau_{\xi_L}) - 2\theta(x + \tau_{\xi_L})| + |\wt Q(\tau_{\xi_L}) - 2\theta \tau_{\xi_L}| \le C +2 \ve \xi_L + \ve |x| \quad\text{for all }x \in \R
\ee
with arbitrarily high probability, by choosing $C$ large enough.
Using \eqref{overline Q small} combined with \eqref{eq:QxiL} for $x \ge -\tau_{\xi_L}$ and \eqref{eq:Qbd2} for $x < -\tau_{\xi_L}$ with $\ve = \f{\theta}{3}$, we see from the definition of $A_C^L$ \eqref{ACL} that for any $\delta > 0$, we may choose $C > 0$ so that 
\be \label{eq:ACLto1}
\liminf_{L \to \infty} \Pp(A_C^L) \ge 1-\delta.
\ee
 On the event $A_C^L$, Equation \eqref{eq:fintermsofQs} implies that the first and second components  of $\pi_{\mathrm{Sh};-}[\underline{f}+(\zeta_{L},0)]$ are bounded above by $2C + (3\theta  +1)|x|$, and  $4C + 9\theta|x|$, respectively. Note that, whenever $A,B > 0$ and $m \in \N$, there exists a common compact set $K$ (depending on $A,B,m$ such that 
 \[
 \sup_{x \in \R}\Bigl[f(x) - \f{1}{m}x^2\Bigr] = \sup_{x \in K}\Bigl[f(x) - \f{1}{m}x^2\Bigr]
 \]
 for all continuous functions $f$ satisfying $f(0) = 0$ and  $f(x) \le A + B|x|$. Hence, \eqref{eq:ExL_sets} and \eqref{eq:ACLto1} imply that $\pi_{\mathrm{Sh};-}[\underline{f}+(\zeta_{L},0)] - \Phi^{\diamond}[\overline{Q}_{\xi_L},\widetilde{Z}_{\xi_L}]$ converges in probability (with respect to the $L^1$ metric on $\CFP^2$ induced by the metric on $\CFP$ defined in \eqref{eq:DCFP_metric}), as $L \to \infty$, to $0$. Since $\Phi^{\diamond}[\overline{Q}_{\xi_L},\widetilde{Z}_{\xi_L}] \sim \widehat \nu_\theta$ for all $L > 0$, this completes the proof.   
\end{proof}
We are now ready to prove the following analogue of \cite[Theorem~1.8]{Dunlap-Sorensen-24}
and \cite[Theorem~1.1]{Dunlap-Ryzhik-2020} in the present setting.
With Lemma~\ref{lem:convergetonuhat} in hand, the proof is the same
as before.
\begin{theorem} \label{thm:stat_from_shock_general}
For a measure $\mu$ on $\mathcal{C}_{\mathrm{FP};0}^{2}$, let $\mathbb{E}_{\mu}$
denote expectation under which $\underline{f}\sim\mu$, for appropriate coupling of directed landscape $\Ll$ and $f$, where $f$ is independent of $\F_{\ge 0}$. Let $\underline{h}(x,t)=(h_{-},h_{+})(x,t)\coloneqq\underline{h}(x,t \mid\underline{f})$.
Then we have, for any $F\in\mathcal{C}_{\mathrm{b}}(\mathcal{C}_{\mathrm{FP};0}^{2})$
and any $t > 0$, that 
\[
\mathbb{E}_{\widehat{\nu}_{\theta}}[F(\underline f)]=\mathbb{E}_{\widehat{\nu}_{\theta}}[F(\pi_{\mathrm{Sh};\pm}[\underline{h}(\cdot,t)])].
\]
\end{theorem}
\begin{proof}
Let $\zeta_L \sim \text{Unif}[0,L]$, independent of everything else. Then, we have
\begin{align*}
\mathbb{E}_{\widehat{\nu}_{\theta}}[F(\pi_{\mathrm{Sh};\pm}[\underline{h}(\cdot,t)])] & =\lim_{L\to\infty}\mathbb{E}_{\nu_{\theta}}\left[F\left(\pi_{\mathrm{Sh};\pm}\left[\underline{h}(\cdot,t\mid \pi_{\mathrm{Sh};\pm}[\underline{f}+(\zeta_L,0)])\right]\right)\right]\\
& =\lim_{L\to\infty}\mathbb{E}_{\nu_{\theta}}\left[F\left(\pi_{\mathrm{Sh};\pm}\left[\underline{h}(\cdot,t\mid \underline{f}+(\zeta_L,0))\right]\right)\right] \\ 
& =\lim_{L\to\infty}\mathbb{E}_{\nu_{\theta}}\left[F\left(\pi_{\mathrm{Sh};\pm}\left[\underline{h}(\cdot,t\mid \underline{f}) + (\zeta_L,0)\right]\right)\right] \\
 & =\lim_{L\to\infty}\mathbb{E}_{\nu_{\theta}}\left[F\left(\pi_{\mathrm{Sh};\pm}\left[\pi_{0}\left[\underline{h}(\cdot,t)\right]+\underline{h}(0,t)+(\zeta_L,0)\right]\right)\right]\\
 & =\lim_{L\to\infty}\mathbb{E}_{\nu_{\theta}}\left[F\left(\pi_{\mathrm{Sh};\pm}\left[\pi_{0}\left[\underline{h}(\cdot,t)\right]+\left(\zeta_L + h_-(0,t) - h_+(0,t),0\right)\right]\right)\right]\\
 & =\lim_{L\to\infty}\mathbb{E}_{\nu_{\theta}}\left[F\left(\pi_{\mathrm{Sh};\pm}\left[\pi_{0}\left[\underline{h}(\cdot,t)\right]+\left(\zeta_{L},0\right)\right]\right)\right]\\
 & =\lim_{L\to\infty}\mathbb{E}_{\nu_{\theta}}\left[F\left(\pi_{\mathrm{Sh};\pm}\left[\underline{f}+\left(\zeta_{L},0\right)\right]\right)\right] =\mathbb{E}_{\widehat{\nu}_{\theta}}[F(\underline f)].
\end{align*}
The first and last identities are by Lemma~\ref{lem:convergetonuhat}, the second follows by Lemma \ref{lem:shift_idempotent}, the third follows because the KPZ fixed point evolution commutes with height shifts of the initial data, 
the fourth is by the definition \eqref{eq:pix0}, the fifth is
by Lemma \ref{lem:pi_constant}, the sixth is by the fact
that, for fixed $t > 0$, the total variation distance between $\Law\left(\left(\zeta_L + h_-(0,t) - h_+(0,t),0\right)\right)$
and $\Law\left(\left(\zeta_L,0\right)\right)$ goes to $0$ as $L\to\infty$, 
and the seventh is by the invariance in Proposition~\ref{prop:nuthetainvariant}.
\end{proof}

\begin{proof}[Proof of Theorem~\ref{thm:KPZFP_invmeas_shock}]
Let $\underline f = (f_-,f_+) \sim \widehat \nu_\theta$. Theorem~\ref{thm:stat_from_shock_general} gives us the invariance 
\be \label{eq:bpm_invar}
\pi_{\mathfrak b_{\pm}[h(\cdot,t \mid \underline f)] }h(\cdot,t \mid \underline f) \deq  \underline f.
\ee
By Lemma~\ref{nutheta_fact}, $\underline f \in \mathcal X(\theta)$ almost surely, and $\mathfrak b_{\pm}[\underline f] = 0$. Then,  Proposition~\ref{prop:interfaces_are_shocks} implies that $b_t^{\pm} = \mathfrak b_{\pm}[h(\cdot,t \mid \underline f)]$, so we can restate \eqref{eq:bpm_invar} as
\be \label{eq:bpm_shift2}
\pi_{b_t^{\pm} }h(\cdot,t \mid \underline f) \deq  \underline f.
\ee
In particular, 
\be \label{eq:hbtpm_eq}
h(\cdot,t \mid \underline f) - h(b_t^-,t\mid \underline f) \deq
h(\cdot,t \mid \underline f) - h(b_t^+,t\mid \underline f).
\ee
Since $\underline f \in \mathcal X(\theta)$, Lemmas~\ref{lem:slope_conserved} and~\ref{KPZFP_attr} imply that  $h(\cdot,t \mid \underline f) \in \mathcal X(\theta)$. Hence, we see that the difference $h(\cdot,t \mid f_+) - h(\cdot,t \mid f_-)$ is nondecreasing, and $x < b_t^-$ if and only if $h(x,t \mid f_+) < h(x,t \mid f_-)$. Then, by the distributional equality in \eqref{eq:hbtpm_eq}, we must have $b_t^- = b_t^+$ with probability $1$. 

Lastly, from Lemma~\ref{lem:KPZFP_max} and \eqref{eq:bpm_shift2}, we get
\[
\pi_{b_t^{\pm}} h(\cdot,t \mid f_- \vee f_+) = \pi_{b_t^{\pm}}[h(\cdot,t \mid f_-) \vee h(\cdot,t \mid f_+)] \deq f_- \vee f_+.
\]
By Lemma~\ref{nutheta_fact}, $f_- \vee f_+$ has the law of the measure described in the statement of the theorem; and this completes the proof. 
\end{proof}

\subsection{Invariant measure from the shock for the open KPZ fixed point} \label{sec:open_to_full_conv}

Here, we prove Theorem~\ref{thm:open_converge}.
\begin{proof}[Proof of Theorem~\ref{thm:open_converge}]
We first prove \eqref{eq:AAsymptoticallyuniform}. Let $F\in\mathcal{C}([0,1])$.
Let $\mathbb{E}$ denote expectation such that $Z\sim\nu^{-\theta;L}$ and
$\overline{\mathbb{E}}$ denote expectation such that $Z\sim\BM$.
Then, we have
\begin{align}
\mathbb{E} & [F(A/L)] =\frac{1}{C_{u;L}}\overline{\mathbb{E}}\left[\exp\left\{ -2^{3/2}\theta\min_{x\in[0,L]}Z(x)+2^{1/2}\theta Z(L)\right\} F(A/L)\right]\nonumber \\
 & =\frac{1}{C_{u;L}}\overline{\mathbb{E}}\left[\exp\left\{ 2^{1/2}\theta \left(Z(0)-Z(A)\right)+2^{1/2}\theta\left(Z(L)-Z(A)\right)\right\} F(A/L)\right]\nonumber \\
 & =\frac{1}{C_{u;L}}\overline{\mathbb{E}}\left[\exp\left\{ 2^{1/2}\theta \left(Z(0)-Z(A)\right)+2^{1/2}\theta \left(Z(L)-Z(A)\right)\right\} F(A/L)\right]\nonumber \\
 & =\frac{1}{C_{u;L}}\overline{\mathbb{E}}\left[\overline{\mathbb{E}}\left[\exp\left\{ 2^{1/2}\theta \left(Z(0)-Z(A)\right)\right\} \ \middle|\ A\right]\overline{\mathbb{E}}\left[\exp\left\{ 2^{1/2}\theta \left(Z(L)-Z(A)\right)\right\} \ \middle|\ A\right]F(A/L)\right],\label{eq:develop-EFA}
\end{align}
with the last identity by the independence assertion in Proposition~\ref{prop:BM_min}\ref{enu:denisov-conditional}.
Now also by Proposition~\ref{prop:BM_min}\ref{enu:denisov-conditional}
along with the explicit description of the law of the endpoint of the Brownian meander (see, for example, \cite[page 118]{Durrett-Iglehart-Miller-1977}, we have, when $Z \sim \BM$, that conditional on $A$, 
\[
A^{-1/2}\left(Z(0)-Z(A)\right),(L-A)^{-1/2}\left(Z(L)-Z(A)\right)\sim\Rayleigh(1),
\]
where $\Rayleigh(1)$ is the law of the Rayleigh distribution with
parameter $1$, which has pdf $t\mapsto t\e^{-t^{2}/2}\mathbf{1}_{t\ge0}$.
This implies that
\begin{align}
\overline{\mathbb{E}} & \left[\exp\left\{ 2^{1/2}\theta \left(Z(0)-Z(A)\right)\right\} \ \middle|\ A\right]\nonumber \\
 & =\int_{0}^{\infty}t\e^{-t^{2}/2}\e^{t\sqrt{2A\theta^{2}}}\,\dif t=\int_{0}^{\infty}\left(-\frac{\dif}{\dif t}+\sqrt{2A\theta^{2}}\right)\e^{-t^{2}/2+t\sqrt{2A\theta^{2}}}\,\dif t\nonumber \\
 & =1+(\sqrt{2A\theta^{2}} + 1)\e^{A\theta^{2}}\int_{0}^{\infty}\e^{-\frac{1}{2}\left(t-\sqrt{2A\theta^{2}}\right)^{2}}\,\dif t=1+(\sqrt{2A\theta^{2}}+1)\e^{A\theta^{2}}\int_{-\sqrt{2A\theta^{2}}}^{\infty}\e^{-\frac{1}{2}t^{2}}\,\dif t.\label{eq:leftside}
\end{align}
Similarly, we have
\begin{equation}
\overline{\mathbb{E}}\left[\exp\left\{ 2^{1/2}\theta \left(Z(L)-Z(A)\right)\right\} \ \middle|\ A\right]=1+(\sqrt{2(L-A)\theta ^{2}} + 1)\e^{(L-A)\theta^{2}}\int_{-\sqrt{2(L-A)\theta^{2}}}^{\infty}\e^{-\frac{1}{2}t^{2}}\,\dif t.\label{eq:rightside}
\end{equation}
Using \eqref{eq:leftside} and \eqref{eq:rightside} in \eqref{eq:develop-EFA},
and also using the fact that $A\sim\Arcsine$ (see Proposition~\ref{prop:BM_min}\ref{enu:arcsine}), we see that
\begin{align*}
\mathbb{E}[F(A/L)] & =\frac{\int_{0}^{1}F(a)\prod_{z\in\{a,1-a\}}\left\{ z^{-1/2}\left(1+(\sqrt{2Lz\theta ^{2}} + 1)\e^{Lz\theta ^{2}}\int_{-\sqrt{2Lz\theta ^{2}}}^{\infty}\e^{-\frac{1}{2}t^{2}}\,\dif t\right)\right\} \,\dif a}{\int_{0}^{1}\prod_{z\in\{a,1-a\}}\left\{ z^{-1/2}\left(1+(\sqrt{2Lz\theta ^{2}} + 1)\e^{Lz\theta ^{2}}\int_{-\sqrt{2Lz\theta ^{2}}}^{\infty}\e^{-\frac{1}{2}t^{2}}\,\dif t\right)\right\} \,\dif a}\\
 & =\frac{\int_{0}^{1}F(a)\prod_{z\in\{a,1-a\}}\left\{ z^{-1/2}+(\sqrt{2L\theta^{2}} + z^{-1/2})\e^{Lz\theta^{2}}\int_{-\sqrt{2Lz\theta^{2}}}^{\infty}\e^{-\frac{1}{2}t^{2}}\,\dif t\right\} \,\dif a}{\int_{0}^{1}\prod_{z\in\{a,1-a\}}\left\{ z^{-1/2}+(\sqrt{2L\theta^{2}} + z^{-1/2})\e^{Lz\theta^{2}}\int_{-\sqrt{2Lz\theta^{2}}}^{\infty}\e^{-\frac{1}{2}t^{2}}\,\dif t\right\} \,\dif a}\\
 & =\frac{\int_{0}^{1}F(a)\prod_{z\in\{a,1-a\}} \left\{\e^{-Lz\theta ^{2}}(2Lz\theta^2)^{-1/2}+\bigl(1 +(2Lz \theta^2)^{-1/2} \bigr)\int_{-\sqrt{2Lz\theta ^{2}}}^{\infty}\e^{-\frac{1}{2}t^{2}}\,\dif t\right\} \,\dif a}{\int_{0}^{1}\prod_{z\in\{a,1-a\}} \left\{\e^{-Lz\theta^{2}}(2Lz\theta^2)^{-1/2}+\bigl(1 +(2Lz \theta^2)^{-1/2} \bigr)\int_{-\sqrt{2Lz\theta ^{2}}}^{\infty}\e^{-\frac{1}{2}t^{2}}\,\dif t\right\} \,\dif a}.
\end{align*}
As $L\to\infty$, by the dominated convergence theorem, the numerator
approaches $2\pi\int_{0}^{1}F(a)\,\dif a$ while the denominator
approaches $2\pi$. This completes the proof of \eqref{eq:AAsymptoticallyuniform}.

Now we work towards a proof of the convergence in distribution of $f_A$. It follows from the convergence of $A/L$ to the uniform distribution on $[0,1]$ that
\begin{equation}
\min\{A,L-A\}\to\infty\qquad\text{in probability as }L\to\infty.\label{eq:Aawayfromboundary}
\end{equation}
Since $A$ is independent of $B \sim \BM$, the strong Markov property gives us that \[
\Bigl(\f{B(A+2x)-B(A)}{\sqrt 2}:x \in \Bigl[-\f{A}{2},L - \f{A}{2}\Bigr]\Bigr)
\]
has the law of $\BM$, restricted to $[-\f{A}{2},L - \f{A}{2}]$. Hence, by Brownian scaling, it suffices to show that, if $Z \sim \nu^{-\theta;L}$, and we define 
\[
Z_{A}(x)=\begin{cases}
Z(A+x)-Z(A), & x\in[-A,L-A];\\
-Z(A), & x\le-A;\\
Z(L)-Z(A), & x\ge L-A,
\end{cases}
\]
then
\begin{equation}
\lim_{L\to\infty}\Law(Z_{A})=\BES^{3}(2^{1/2}\theta)\label{eq:convergetoshockmeasure}
\end{equation}
weakly with respect to the topology of uniform convergence on compact
sets. Fix $K>0$ and let $F\in\mathcal{C}_b(\mathcal{C}([-K,K]))$. We
write $F(f)=\widetilde{F}(f|_{[-K,0]},f|_{[0,K]})$. Similarly to
the computation leading to \eqref{eq:develop-EFA}, we have
\begin{align*}
\mathbb{E}[F(Z_{A})] & =\frac{1}{C_{u;L}}\overline{\mathbb{E}}\left[\exp\left\{ 2^{1/2}\theta (Z(0)-Z(A))+2^{1/2}\theta (Z(L)-Z(A))\right\} F(Z_{A})\right]\\
 & =\frac{1}{Z_{u;L}}\overline{\mathbb{E}}\left[\overline{\mathbb{E}}\left[\exp\left\{ 2^{1/2}\theta Z_{A}(-A)+2^{1/2}\theta Z_{A}(L-A)\right\} \widetilde{F}(Z_{A}|_{[-K,0]},Z_{A}|_{[0,K]})\ \middle|\ A\right]\right].
\end{align*}
Then, by Proposition
\ref{prop:BM_min}\ref{enu:denisov-conditional} and \eqref{eq:meander-girsanov}, To the left and right of $A$, $Z_A$ has the law of two independent Brownian meanders with drift, of length $A$ and $L-A$, respectively. Then, \eqref{eq:convergetoshockmeasure} follows by \eqref{eq:Aawayfromboundary} and the convergence of the Brownian meander with drift to the Bessel process with drift proved in Proposition \ref{prop:meander-to-bessel}.
\end{proof}

\section{Characterizing the extremal invariant measures} \label{sec:characterize}
In this section, we prove Theorem~\ref{thm:KPZFP_characterize}. We do this in two parts. In Section~\ref{sec:slopes}, we prove Proposition~\ref{prop:inv_meas_have_slope}, which states that all extremal invariant measures are supported on functions having fixed asymptotic slopes. Combined with Proposition~\ref{prop:invariance_of_SH}, this leads to Corollary~\ref{cor:almost_characterization}, which reduces the characterization to showing the nonexistence of V-shaped invariant measures. This nonexistence is handled in Section~\ref{sec:no_V-shape}. Below, we recall the state space $\CFP$ defined in \eqref{eq:CFP}.

\subsection{Proving extremal invariant measures  have deterministic asymptotic slopes} \label{sec:slopes} The purpose of this subsection is to prove Proposition~\ref{prop:inv_meas_have_slope} below. To do so, we introduce the concept of an eternal solution of the KPZ fixed point. Given a realization of the directed landscape $\Ll$, an \textit{eternal solution} 
is a function $b:\R^2 \to \R$ such that, for all $s < t$ and $x \in \R$,
\[
b(x,t) = \sup_{z \in \R}[b(z,s) + \Ll(z,s;x,t)].
\]
A \emph{semi-infinite geodesic}  rooted at the point $(x,t)$ is encoded by a continuous function $g:(-\infty,t] \to \R$ such that, for any $s < t$, the restriction $g|_{[s,t]}$ is a geodesic for $\Ll$. We say the semi-infinite geodesic $g$ has \emph{direction} $\theta$ if 
\[
\lim_{s \to -\infty} \f{g(s)}{|s|} = \theta.
\]
Proposition \ref{prop:all_directions}, originally from \cite{Busa-Sepp-Sore-22a}, states that all semi-infinite geodesics have a direction. For a semi-infinite geodesic $g$, define $\zeta(g)$ to be its direction. 

The following lemma essentially follows from standard results in the theory of random dynamical systems; we present the details for completeness.
\begin{lemma} \label{lem:eternal_soln_existence}
    Let $\mu$ be an invariant measure for the recentered KPZ fixed point on the space $\CFP$. Then, there exists a coupling of the directed landscape $\Ll$ with a 
    continuous eternal solution $b$ satisfying $b(0,0) = 0$ and
    \be \label{eq:b_t_law}
    \pi_0[b(\cdot,t)] \sim \mu \quad\text{for all }t \in \R.
    \ee
\end{lemma}
\begin{proof}
Let $f \sim \mu$, independent of $\Ll$. For each $T > 0$, consider the function $b_T:\R \times [-T + 1,\infty) \to \R$ defined by 
\[
b_T(x,t) = h(x,t \mid -T,f) - h(0,0 \mid -T,f).
\]
By Proposition~\ref{prop:h_pres_CFP}, $b_T$ is continuous with probability $1$; we extend $b_T$ to a continuous function $\R^2 \to \R$ in some measurable way. By the assumed invariance of $\mu$, we have, for all $t > - T+1$,
\be \label{eq:bT_law_stable}
\Law(\pi_0[b_T(\cdot,t)]) = \mu.
\ee
We view $(b_T,\Ll)$ as an element of the space $\C(\R^2,\R) \times \C(\Rup,\R)$, together with the topology of uniform convergence on compact sets. We claim that $(b_T,\Ll)$ is tight in this space, and any subsequential limit of $(b_T,\Ll)$ is a tuple $(b,\Ll)$, where $\Ll$ is the directed landscape and $b$ is an eternal solution with the desired property. First, observe that, for $-T< s < t$, by the metric composition  property of $\Ll$,
\begin{equation} \label{eq:sgroup}
\begin{aligned}
h(x,t \mid - T,f) &= \sup_{z \in \R}[f(z) + \Ll(z,-T;x,t)]  \\
&= \sup_{z,w \in \R}[f(z) + \Ll(z,-T;w,s) + \Ll(w,s;x,t)] \\
&= \sup_{w \in \R}[h(w,s\mid - T,f) + \Ll(w,s;x,t)].
\end{aligned}
\end{equation}
Then, for $-T + 1 < s < t \wedge 0$,
\begin{align*}
b_T(x,t) &= h(x,t \mid -T,f) - h(0,0 \mid - T,f) \\
&= \sup_{z \in \R}\bigl[h(w,s \mid - T,f) + \Ll(w,s;x,t) \bigr] - \sup_{z \in \R}\bigl[h(w,s \mid - T,f) + \Ll(w,s;0,0) \bigr] \\
&= \sup_{z \in \R}\bigl[b_T(w,s) - b_T(0,s) + \Ll(w,s;x,t)  \bigr] - \sup_{z \in \R}\bigl[b_T(w,s) - b_T(0,s) + \Ll(w,s;0,0)  \bigr].
\end{align*}
Note the process $w \mapsto b_T(w,s) - b_T(0,s)$ is measurable with respect to the $\sigma$-algebra generated by $\F_{\le s}$ and the random function $f$, while $w \mapsto \Ll(w,s;x,t)$ and $w \mapsto \Ll(w,s;0,0)$ are measurable with respect to $\F_{\ge s}$ (Recall Definition~\ref{def:sigma_alg}). By independence of these two $\sigma$-algebras and \eqref{eq:bT_law_stable}, we see that for each fixed $s < 0$, the law of $\{b_T(x,t): t > s, x \in \R \}$ does not depend on $T$, so long as $T >  1- s$. Hence, the process $b_T$ is tight as $T \to \infty$. Coupled with $\Ll$, we have that $(b_T,\Ll)$ is tight. By Skorokhod representation (see e.g.~\cite[Theorem~11.7.2]{dudl} or~\cite[Theorem~3.1.8]{EKbook}), we may let $(b,\Ll)$ be a subsequential limit defined on some probability space as the almost sure limit of some coupling $(b_T,\Ll_T)$, where $\Ll_T$ is a copy of the directed landscape for each $T$. By \eqref{eq:bT_law_stable}, $b$ satisfies \eqref{eq:b_t_law}. It remains to show that $b$ is an eternal solution.   

We can apply \eqref{eq:sgroup} again to get that, for $-T + 1 < s < t$,
\begin{equation} \label{eq:bT_prelim}
\begin{aligned}
b_T(x,t) &= \sup_{w \in \R}\bigl[h(w,s \mid -T,f)  + \Ll_T(w,s;x,t) \bigr]- h(0,0\mid-T,f) \\
&= \sup_{w \in \R}[b_T(w,s) + \Ll_T(w,s;x,t)].
\end{aligned}
\end{equation}
Taking limits as $T \to \infty$, we see that, for all $w \in \R$, 
\[
b(x,t) = \lim_{T \to \infty} b_T(x,t) \ge \lim_{T \to \infty}[b_T(w,s) + \Ll_T(w,s;x,t)] =b(w,s) + \Ll(w,s;x,t),
\]
and hence,
\be \label{eq:bxtge}
b(x,t) \ge \sup_{w \in \R}[b(w,s) + \Ll(w,s;x,t)].
\ee
Additionally, we have seen that for all sufficiently large $T$, the law of $w \mapsto b_T(w,s)$ and the law of $b_T(x,t)$ does not depend on $T$, so for each $t > s$ and $x \in \R$,
\begin{align*}
&\quad\,\Law\bigl(b(x,t)\bigr) = \Law(b_T(x,t)) \\
&= \Law\bigl(\sup_{w \in \R}[b_T(w,s) + \Ll_T(w,s;x,t)]\bigr) =  \Law\bigl(\sup_{w \in \R}[b(w,s) + \Ll(w,s;x,t)]\bigr).
\end{align*}
Thus, \eqref{eq:bxtge} is an equality with probability $1$ for each fixed $t > s$ and $x \in \R$. This equality immediately extends to all rational $x$ and rational pairs $t >s$ on an event of probability $1$, and then extends to all points by continuity. Hence, $b$ is an eternal solution, as desired. 
\end{proof}

In Proposition~\ref{prop:geodesics_from_b}, originally from \cite{Bhattacharjee-Busani-Sorensen-25}, it is shown that any eternal solution $b$ gives rise to a family of infinite geodesics. For such an eternal solution $b$, $(x,t) \in \R^2$, and 
$s < t$, define $g_{(x,t)}^{b,\mathrm{R}}(s)$ to be the rightmost maximizer of 
\be \label{eq:bLfun}
z \mapsto b(z,s) + \Ll(z,s;x,t)
\ee
over $z \in \R$, whenever rightmost maximizers exist. We also define $g_{(x,t)}^{b,\mathrm{R}}(t) = x$.  We now prove the following lemmas.
\begin{lemma} \label{lem:ordering}
If $b$ is an eternal solution such that rightmost maximizers in \eqref{eq:bLfun} exist, then whenever $x < y$ and $t \in \R$,
\[
g_{(x,t)}^{b,\mathrm{R}}(s) \le g_{(y,t)}^{b,\mathrm{R}}(s),\text{ for all }s \le t,\qquad\text{and therefore,}\quad \zeta(g_{(x,t)}^{b,\mathrm{R}}) \le \zeta(g_{(y,t)}^{b,\mathrm{R}}).
\]

\end{lemma}
\begin{proof}
The geodesics are continuous, and we have $g_{(x,t)}^{b,\mathrm{R}}(t) = x < y = g_{(y,t)}^{b,\mathrm{R}}(t)$ by definition. By Proposition~\ref{prop:geodesics_from_b}\ref{itm: consis}, if $w := g_{(x,t)}^{b,\mathrm{R}}(r) = g_{(y,t)}^{b,\mathrm{R}}(r)$ for some $r < t$, then, for all $s \le r$,
\[
g_{(x,t)}^{b,\mathrm{R}}(s) = g_{(w,r)}^{b,\mathrm{R}}(s) = g_{(y,t)}^{b,\mathrm{R}}(s). \qedhere
\]
\end{proof}

\begin{lemma} \label{lem:only_3_slopes}
Let $b$ be an eternal solution for $\Ll$ such that, for all $s < t$ and $x \in \R$, leftmost and rightmost maximizers of the function
\[
z \mapsto b(z,s) + \Ll(z,s;x,t)
\]
exist. For shorthand, let $g_{(x,t)} = g^{b,\mathrm{R}}_{(x,t)}$. If, for each $t \in \R$, the function $x \mapsto b(x,t)$ is independent of $\F_{\ge t}$ and the law of the function $x \mapsto \pi_0[b(\cdot ,t)](x)$ is the same for all $t \in \R$, then
\[
\Pp\bigl(0 < \zeta(g_{(x,0)})<\zeta(g_{(y,0)}) \text{ for some } x < y\bigr) = \Pp\bigl(\zeta(g_{(x,0)})<\zeta(g_{(y,0)}) < 0 \text{ for some } x < y\bigr) = 0.
\]
\end{lemma}
\begin{proof}
We just show that the first probability is $0$, since the other statement follows by a symmetric proof.

Assume, to the contrary, that 
\[
\Pp\bigl(0 < \zeta(g_{(x,0)})<\zeta(g_{(y,0)}) \text{ for some } x < y\bigr) > 0.
\]
Then, by the ordering in Lemma~\ref{lem:ordering}, there exists $0 < q_1 < q_2$ and $x_1 < x_2$ with $q_1,q_2,x_1,x_2 \in \Q$ so that 
\[
\Pp\Bigl(q_1< \zeta(g_{(x_1,0)}) < q_2  <  \zeta(g_{(x_2,0)})\Bigr) > 0.
\]
For this choice of $q_1,q_2,x_1,x_2$, and $t \in \R$, define the event 
\[
A_t := \Bigl\{ q_1< \zeta(g_{(x_1,t)}) < q_2  <  \zeta(g_{(x_2,t)})  \Bigr\}.
\]
Our assumption is that $\Pp(A_0) > 0$. Recall that, for $i = 1,2$, and $r > 0$, $g_{(x_i,t)}(t-r)$ is the rightmost maximizer of  
\[
z \mapsto b(z,t-r) + \Ll(t-r,s;x_i,t),
\]
which is the same as the rightmost maximizer of 
\[
z \mapsto \pi_0[b(\cdot,t-r)](z) + \Ll(z,t-r;x_i,t).
\]
Then, by stationarity of the increments of $b$ and shift stationarity of $\Ll$ (Lemma~\ref{lm:landscape_symm}\ref{itm:time_stat}), we have that, for each $r > 0$ the law of the tuple $(g_{(x_1,t)}(t-r),g_{(x_2,t)}(t-r))$ is the same for all $t \in \R$. Dividing by $|r|$ and taking limits as $r \to -\infty$, we have that the law of $\bigl(\zeta(g_{(x_1,t)}),\zeta(g_{(x_2,t)})\bigr)$ is the same for all $t \in \R$, and so  $\Pp(A_t) = \Pp(A_0)$ for all $t \in \R$.

If the event $A_t$ occurs for some $t \in \R$, we now claim that there exists $S < t$ so that $A_s$ does not occur for all $s \le S$. To see this, assume that $A_t$ occurs. Then, 
\[
0 < q_1 < \lim_{s \to -\infty} \f{g_{(x_1,t)} (s)}{|s|} < q_2,
\]
so there exists $S < t$ so that $g_{(x_1,t)}(s) > x_2$ for all $s \le S$. Then, for $r \le s$, by the ordering of geodesics in Lemma~\ref{lem:ordering} and consistency of the semi-infinite geodesics in Proposition~\ref{prop:geodesics_from_b}\ref{itm: consis},
\[
\
g_{(x_2,s)}(r) \le g_{(x_1,t)}(r).
\]
Dividing by $|r|$ and taking limits as $r \to -\infty$, we see that
\[
\zeta(g_{(x_2,s)}) \le \zeta(g_{(x_1,t)}) < q_2.
\]
Then, $A_s$ does not occur by definition of the event. 

Thus, we have shown that $\lim_{s \to -\infty} \ind_{A_s} = 0$ almost surely. This is because either $\ind_{A_s} = 0$ for all $s \in \R$, or else there exists some $t$ such that $\ind_{A_t} = 1$, in which case there exists some $S < t$ so that $\ind_{A_s} = 0$ for all $s \le S$. By the bounded convergence theorem, $\Pp(A_s) \to 0$ as $s \to -\infty$. But we have seen that $\Pp(A_s) = \Pp(A_0)$ for all $s \in \R$, so $\Pp(A_0) = 0$, a contradiction to our assumption.
\end{proof}

We now present the first main ingredient in the classification of invariant measures for the KPZ fixed point. That is, all extremal invariant measures are supported on functions have fixed asymptotic slopes at $\pm \infty$.
\begin{proposition} \label{prop:inv_meas_have_slope}
Let $\mu$ be an extremal (i.e., ergodic) invariant measure for the recentered KPZ fixed point evolution on the space $\CFP$. Then, there exist deterministic constants $-\infty < -\alpha \le \beta < \infty$ such that,  for $f \sim \mu$, $\mu$-almost surely,
\[
\lim_{x \to -\infty} \f{f(x)}{x} = -2\alpha,\quad\text{and}\quad \lim_{x \to +\infty} \f{f(x)}{x} = 2\beta.
\]
\end{proposition}
\begin{proof}
By Lemma~\ref{lem:slope_conserved}, asymptotic slopes are conserved quantities for the KPZ fixed point, so it suffices to show that for any invariant measure $\mu$, when $f \sim \mu$ the (possibly random) limits
\[
\lim_{x \to \pm \infty} \f{f(x)}{x}
\]
exist almost surely, and the limit to $-\infty$ is bounded above by the limit to $+\infty$. On an appropriate probability space, let $b$ be the eternal solution associated to $\Ll$ obtained in Lemma~\ref{lem:eternal_soln_existence}. Then, $f$ has the same law as $x \mapsto b(x,0) - b(0,0) = b(x,0)$, so it suffices to prove the asymptotic slopes for $b(x,0)$. By definition of the state space \eqref{eq:CFP} and the modulus of continuity bounds in Lemma~\ref{lm:landscape_symm}, leftmost and rightmost maximizers in \eqref{eq:bLfun} exist for all $s < t$ and $x \in \R$. For $(x,t) \in \R$, adopt the shorthand notation $g_{(x,t)} = g_{(x,t)}^{b,\mathrm{R}}$. By Lemma~\ref{lem:only_3_slopes}, there are at most $3$ distinct values of $\zeta(g_{(x,0)})$ as $x$ varies over $\R$ (at most one negative value and at most one positive value). By the ordering of geodesics in Lemma~\ref{lem:ordering} and Propositions~\ref{prop:Busani_N3G} and Proposition~\ref{prop:geod_ordering}, there exist (possibly random) $\alpha \le \beta$, $\sigg_1,\sigg_2 \in \{-,+\}$, and $m < M$, such that, for all $x \le m$, $g_{(x,0)}$ is a $-\alpha \sigg_1$ Busemann geodesic, and for all $x \ge M$, $g_{(x,0)}$ is a $\beta \sigg_2$ Busemann geodesic, as in Definition~\ref{def:Buse_geodesics}. Now let $x \ge M$.  By Proposition~\ref{prop:geod_ordering}, there exists $T = T(M,x)$ such that, for all $s \le T$, $g_{(M,0)}(s) = g_{(x,0)}(s)$. Call this common value $g(s)$. Then, using Equation \eqref{eqn:SIG_weight} from Proposition~\ref{prop:DL_SIG_cons_intro}\ref{itm:arb_geod_cons} in the third line below, followed by the additivity in Proposition~\ref{prop:Buse_basic_properties}\ref{itm:DL_Buse_add},
\begin{align*}
b(x,0) - b(M,0) &= \sup_{z \in \R}[b(z,T) + \Ll(z,t;x,0)] - \sup_{z \in \R}[b(z,T) + \Ll(z,t;M,0)] \\
&=  b(g(s),T) + \Ll(g(s),T;x,0) - [b(g(s),T) + \Ll(g(s),T;M,0)] \\
&= W^{\beta  \sig_2}(g(s),T;x,0) - W^{\beta \sig_2}(g(s),T;M,0) \\
&= W^{\beta \sig_2}(M,0;x,0).
\end{align*}
We observe here that, while the value $T$ depends on $M$ and $x$, the equality $b(x,0) - b(M,0) = W^{\beta \sig_2}(M,0;x,0)$ is true for all $x \ge M$. Then, using the additivity of $W^{\beta \sig_2}$ and the asymptotic slopes in Proposition~\ref{prop:Buse_basic_properties}\ref{it:Wslope},
\[
\lim_{x \to +\infty} \f{b(x,0)}{x} = \lim_{x \to +\infty} \f{b(x,0) - b(M,0)}{x} = \lim_{x \to +\infty} \f{W^{\beta \sig_2}(M,0;x,0)}{x} = 2\beta.
\]
A symmetric argument shows that 
\[
\lim_{x \to -\infty} \f{b(x,0)}{x} =-2\alpha. \qedhere
\]
\end{proof}

\begin{corollary} \label{cor:almost_characterization}
  Let $\mu$ be an extremal invariant measure for the recentered KPZ fixed point on the space $\CFP$. Then, $\mu = \BM(2\theta;\sqrt 2)$ for some $\theta \in \R$, or there exists $\theta > 0$ such that $\mu(\mathcal V_0(\theta)) = 1$, where $\mathcal{V}_0(\theta)$ is defined in \eqref{zero_spaces}. 
\end{corollary}
\begin{proof}
Assume that $\mu$ is an invariant measure and let $f \sim \mu$. Since we recenter at $0$, we must have $f(0) = 0$, $\mu$-almost surely. By Proposition~\ref{prop:inv_meas_have_slope}, there exist deterministic $-\alpha \le \beta$ such that, $\mu$-almost surely, 
\[
\lim_{x \to -\infty} \f{f(x)}{x} = -2\alpha,\quad\text{and}\quad \lim_{x \to +\infty} \f{f(x)}{x} = 2\beta.
\]
It suffices to show that if $\alpha \neq \beta$ or $\beta \le 0$, then $\mu = \BM(2\theta,\sqrt 2)$ for some $\theta \in \R$. In each of these cases, since $G_{\theta} \sim \BM(2\theta,\sqrt 2)$ (Definition~\ref{def:SH}), the $k = 1$ case of the uniqueness in Proposition~\ref{prop:uniform_upbd} implies that $\mu$ is one of  $\BM(-2\alpha,\sqrt 2),\BM(0,\sqrt 2)$, or $\BM(2\beta,\sqrt 2)$. 
\end{proof}

\subsection{Ruling out $V$-shaped invariant measures} \label{sec:no_V-shape}
In this section, we rule out the second case given in Corollary~\ref{cor:almost_characterization}. 
Recall the measure $\nu_\theta$ from Definition~\ref{nutheta_def}.  A key input is the following:
 \begin{proposition} \label{prop:hpm_dif}
    Let $\theta > 0$ and $(f_-,f_+) \sim \nu_\theta$, and define $h_{\pm}(x,t) = h(x,t \mid f_{\pm})$. Then, as $t \to \infty$, we have the distributional convergence
    \[
    \f{h_+(0,t) - h_-(0,t)}{\sqrt t}\Longrightarrow \Nor(0,4\theta).
    \]
\end{proposition}
\begin{proof}
We first show that 
\be \label{eq:piece1}
t^{-1/2}\Bigl(h_+(-\theta t,t) - h_+(0,t) - \bigl(h_-(\theta t,t) - h_-(0,t)\bigr)\Bigr) \Longrightarrow \mathcal N(0,4\theta).
\ee
To see this, recall that, because $\nu_\theta$ is jointly invariant for the KPZ fixed point (Proposition~\ref{prop:invariance_of_SH}),
\begin{align*}
&\quad \, h_+(-\theta t,t) - h_+(0,t) - \bigl(h_-(\theta t,t) - h_-(0,t)\bigr) \\
&\deq f_+(-\theta t) - f_-(\theta t) \\
&\deq W_2(-\theta t) + \mathcal S_\theta(-\theta t) - \mathcal S_\theta(0) - W_1(\theta t).
\end{align*}
We get the last equality by the explicit description of the law $\nu_\theta$ from Definitions~\ref{def:SH} and~\ref{nutheta_def}. Here, $W_1$ and $W_2$ are independent, $W_1 \sim \BM(-2\theta,\sqrt 2)$, and $W_2 \sim \BM(2\theta,\sqrt 2)$, and $x \mapsto \mathcal S_\theta(x)$ is the stationary process 
\[
x \mapsto \sup_{y \le x}[W_2(x) - W_2(y) - (W_1(x) - W_1(y))].
\]
Upon dividing by $t^{1/2}$, we get that 
\begin{align*}
&\quad\, t^{-1/2}\Bigl(h_+(-\theta t,t) - h_+(0,t) - (h_-(\theta t,t) - h_-(0,t))\Bigr)  \\
&\deq  t^{-1/2}(W_2(-\theta t) - W_1(\theta t)) + t^{-1/2}(\mathcal S_\theta(-\theta t) - \mathcal S_\theta(0)),
\end{align*}
The term $t^{-1/2}(W_2(-\theta t) - W_1(\theta t))$ has the law $\mathcal N(0,4\theta)$. The error term $t^{-1/2}(\mathcal S_\theta(-\theta t) - \mathcal S_\theta(0))$ goes to $0$ in probability by stationarity of $\mathcal S_\theta$. 

By \eqref{eq:piece1}, it now suffices to show that, as $t \to \infty$,
\[
t^{-1/2}\Bigl(h_+(-\theta t,t) - h_-(\theta t,t)   \Bigr) \to 0 \quad\text{in probability.}
\]
Using the Skew-stationarity in Lemma~\ref{lm:landscape_symm},
\begin{align*}
    h_+(-\theta t,t) &= \sup_{y \in \R}[\Ll(y,0;-\theta t,t) + h_+(0,y)] \\
    &\deq \sup_{y \in \R}[\Ll(y,0;0,t) + h_+(0,y) - 2\theta y] - t \theta^2 \\
    &\deq \sup_{y \in \R}[\Ll(y,0;0,t) + h_0(y)] - t\theta^2,
\end{align*}
where $h_0$ is a Brownian motion with diffusivity $\sqrt 2$ and drift $0$. By a symmetric argument, this is the same as the distribution of $h_-(\theta t,t)$. Then, using the rescaling invariance of Lemma~\ref{lm:landscape_symm} with $q = t^{1/3}$ along with Brownian scaling,
\begin{align*}
\sup_{y \in \R}[\Ll(y,0;0,t) + h_0(y)] &\deq t^{1/3} \sup_{y \in \R}[\Ll(t^{-2/3}y,0;0,1) + h_0(t^{-2/3}y)] \\
&= t^{1/3}\sup_{y \in \R}[\Ll(y,0;0,1) + h_0(y)].
\end{align*}
In summary, the laws of 
\[
\f{h_+(-\theta t,t) + t\theta^2}{t^{1/3}}\quad\text{and}\quad \f{h_-(\theta t,t) + t\theta^2}{t^{1/3}}
\]
are the same and do not depend on $t$. Thus,
\[
t^{-1/2}\Bigl(h_+(-\theta t,t) - h_-(\theta t,t)   \Bigr)= t^{-1/6}\Biggl(\f{h_+(-\theta t,t) + t\theta^2}{t^{1/3}} - \f{h_-(\theta t,t) + t\theta^2}{t^{1/3}}\Biggr) \to 0 \text{ in probability},
\]
as desired. 
\end{proof}
Recall the map $V:\mathcal Y(\theta) \to \mathcal V(\theta)$ given in \eqref{eq:V_map}. We show the following.
\begin{lemma}\label{lem:VAinverses}
    Let $\theta > 0$. There exists a measurable map $\mathbf A:\mathcal V(\theta) \to \mathcal Y(\theta)$ such that, for all $f \in \mathcal V(\theta)$, $V\bigl[\mathbf A[f]\bigr] = f$. 
\end{lemma}
\begin{proof}
    For $f \in \mathcal V(\theta)$, we define $\mathbf A[f] \coloneqq (f_-,f_+)$ as follows. Since $f$ is continuous and satisfies $\lim_{|x| \to \infty
} \f{f(x)}{x} = 2\theta > 0$, $f$ attains its global minimum on $\R$. Let $x_{\mathrm{min}}$ be the leftmost minimizer and set 
\be \label{A_def}
\begin{aligned}
f_-(x) &= \begin{cases}
f(x) &x \le x_{\mathrm{min}} \\
f(x_{\mathrm{min}}) - \theta(x - x_{\min}) &x \ge x_{\mathrm{min}}
    \end{cases} \\
    f_+(x) &= \begin{cases}
    f(x_{\mathrm{min}}) + \theta(x -x_{\min}) &x \le x_{\min} \\
    f(x) &x \ge x_{\min}.
    \end{cases}
\end{aligned}
\ee
It follows immediately that $(f_-,f_+) \in \mathcal Y(\theta)$ and $V[f_-,f_+](x) = f_-(x) \vee f_+(x) = f(x)$. 
\end{proof}

\begin{proposition} \label{prop:no_V_shaped}
Let $\theta > 0$. There is no probability measure $\mu$ on $\mathcal C_{\mathrm{FP}}$ such that $\mu$ is an invariant measure for the recentered KPZ fixed point evolution and $\mu(\mathcal V_0(\theta)) = 1$.  
\end{proposition}
\begin{proof}
    With analogous ingredients having been established in the present setting, the proof follows very closely that of \cite[Proposition~4.3]{Dunlap-Sorensen-24}. We reproduce the details for completeness.

    We argue by contradiction. Suppose that such a measure $\mu$ exists. Let $f_{\mathsf{V}}\sim\mu$, let $\underline{f} = \mathbf{A}[f_{\mathsf{V}}]$, and let $(\underline{h},h_{\mathsf{V}})(x,t) = (h_-,h_+,h_{\mathsf{V}})(x,t) = \underline{h}(x,t \mid (\underline{f},f_{\mathsf{V}}))$.
    Lemma~\ref{lem:VAinverses} shows that $V[\underline{f}] = f_{\mathsf{V}}$, and then
    Lemma~\ref{lem:KPZFP_max} implies that, in fact,
    \begin{equation}\label{eq:VishV}
    V[\underline{h}(\cdot,t)](x) = h_{\mathsf{V}}(x,t)\qquad\text{for all }t\ge0\text{ and }x\in\R.
    \end{equation}

    Now we let $U_T\sim \Uniform([0,T])$ (independent of all other random variables). By Proposition~\ref{prop:uniform_upbd}, we have (since $\underline{f}\in \mathcal{Y}(\theta)$ by construction) that 
    \begin{equation}\label{eq:convergence}
      \Law(\underline{h}(\cdot,t)-\underline{h}(0,t))\xrightarrow[t\to\infty]{}\nu_\theta,\qquad \text{weakly in the $\mathcal{C}_{\mathrm{FP}}^2$ topology.}
    \end{equation}
We claim that 
\begin{equation}\label{eq:hdifftightcontradiction}
  \text{for all $\eps>0$, there exists $K<\infty$ such that $\sup\limits_{T\in (0,\infty)} \mathbb{P}(h_+(0,U_T)-h_-(0,U_T)>K)<\eps$.}
\end{equation}
We prove \eqref{eq:hdifftightcontradiction} by contradiction. Suppose that there exists some $\eps>0$ and a sequence $T_k\uparrow\infty$ such that
\[
  \mathbb{P}((h_+ - h_-)(0,U_{T_k})>k)\ge \eps\qquad\text{for each }k\in\mathbb{N}.
\]
Now by \eqref{eq:convergence}, there is an $M_0<\infty$ such that if $x<-M_0$, then there exist $A(x),C(x)\in (0,\infty)$ such that
\begin{equation*}
  \sup_{k\in\mathbb{N}} \mathbb{P}\bigl(\left|(h_+-h_-)(x,U_{T_k})-(h_+-h_-)(0,U_{T_k})\right|>A(x)\bigr)<\frac\eps4,
\end{equation*}
and
\[
  \sup_{k\ge C(x)} \mathbb{P}\left(h_+(x,U_{T_k})-h_+ (0,U_{T_k})>0\right)<\frac\eps4.
\]
Combining the last three displays, we see that for all $x\le -M_0$ and $k\ge C(x)$, we have $\mathbb{P}(E_{k,x})\ge \eps/2$, where $E_{k,x}$ is the event that the following three inequalities hold: 
\begin{align}
  (h_+-h_-)(0,U_{T_k})&\ge k,\label{eq:zerofar}\\
  (h_+-h_-)(x,U_{T_k}) &\ge k-A(x),\qquad\text{and}\label{eq:xbigger}\\
  h_+(x,U_{T_k})-h_+(0,U_{T_k})&\le 0.\label{eq:xsmaller}
\end{align}
If $k\ge A(x)$, then on the event $E_{k,x}$, we have by~\eqref{eq:zerofar} and~\eqref{eq:xbigger} that $V[\underline{h}(\cdot, U_{T_k})](y) = h_+(y,U_{T_k})$ for $y\in\{0,x\}$, and hence  
\[h_{\mathsf{V}}(x,U_{T_k}) -h_{\mathsf{V}}(0,U_{T_k}) = V[\underline{h}(\cdot, U_{T_k})](x)-V[\underline{h}(\cdot, U_{T_k})](0) = h_+(x,U_{T_k}) -h_+(0,U_{T_k})\le 0, \]
with the inequality following by \eqref{eq:xsmaller}.
Thus, we in fact have
\[
  \sup_{k\ge C(x)\vee A(x)}\mathbb{P}\left(h_{\mathsf{V}}(x,U_{T_k})-h_{\mathsf{V}}(0,U_{T_k})\le 0\right)\ge \frac\eps2.
\]
But the increments of $h_{\mathsf{V}}$ were assumed to be stationary in time, so the last probability does not depend on $k$, and indeed, we obtain the uniform statement
\[
  \inf_{x\le -M_0}\mathbb{P}\left(h_{\mathsf{V}}(x,U_{T_k})-h_{\mathsf{V}}(0,U_{T_k})\le 0\right)\ge \frac\eps2\qquad\text{for all $k$}.
\]
This contradicts the assumption that $f_{\mathsf{V}}\in \mathcal{V}_0(\theta)$ $\mu$-almost surely. Hence the proof of \eqref{eq:hdifftightcontradiction} is complete, and by using a symmetrical argument, we can actually conclude that
\[
  \sup_{T\in (0,\infty)}\mathbb{P}\left(|h_+(0,U_T)-h_-(0,U_T)| >K\right)<\eps,
\]
so the family $(h_+(0,U_T)-h_-(0,U_T))$ is in fact tight. Using this along with \eqref{eq:convergence}, we see that if we define the projection
\[\tilde\pi[(f_-,f_+)](x)\coloneqq \left(f_-(x)-\frac12(f_+(0)+f_-(0)),f_+(x)-\frac12(f_+(0)+f_-(0))\right),\]
then we in fact have
that $\tilde\pi[\underline{h}(\cdot,t)]$ 
is also tight in the topology of $\mathcal{C}^2_{\mathrm{FP}}$.  Therefore, there exists a sequence $T_k\uparrow\infty$ and a limiting measure $\psi$ on $\mathcal{C}^2_{\mathrm{FP}}$ such that
\begin{equation}\label{eq:conv-to-psi}
  \lim_{k\to\infty} \Law(\tilde\pi[\underline{h}(U_{T_k},\cdot)]) = \psi\qquad\text{weakly on }\mathcal{C}^2_{\mathrm{FP}}.
\end{equation}
Now, Proposition~\ref{prop:Markov_feller} means that the semigroup for the recentered process $\tilde\pi[\underline{h}]$ has the Feller property. 
Thus, we can apply the Krylov--Bogoliubov theorem \cite[Theorem~3.1.1]{Da-Prato-Zabczyk-1996} to conclude that $\psi$ is actually invariant for this semigroup. In particular, this means that if $\tilde{\underline{h}}(0,\cdot) = (\tilde{h}_-,\tilde{h}_+)(0,\cdot)\sim \psi$ and $\tilde{\underline{h}}$ evolves as the KPZ fixed point, then the law of $\tilde\pi[\tilde{\underline{h}}(\cdot,t)]$ does not depend on $t$. Since $\tilde{h}_+(0,t)-\tilde{h}_-(0,t)$ is also the difference of the two coordinates of $\tilde\pi[\tilde{h}(\cdot,t)](0)$, this in particular means that the family $(\tilde{h}_+(0,t)-\tilde{h}_-(0,t))_t$ 
is a family of identically distributed random variables and hence tight. On the other hand, by \eqref{eq:conv-to-psi} and the continuity of $\pi_0$, along with the fact that $\pi_0\circ \tilde \pi = \pi_0$, we see that $\pi_0[\tilde{\underline{h}}(\cdot,0)]\sim\nu_\theta$, and then 
Proposition~\ref{prop:hpm_dif} tells us that $(\tilde{h}_+(0,t)-\tilde{h}_-(0,t))_t$ is not tight, a contradiction. 
\end{proof}

\begin{proof}[Proof of Theorem~\ref{thm:KPZFP_characterize}]
The measures $\BM(2\theta,\sqrt 2)$ are each invariant for the KPZ fixed point \cite{KPZfixed}. Since these measures are mutually singular (indexed by the asymptotic drift condition), the classification follows immediately from Corollary~\ref{cor:almost_characterization} and Proposition~\ref{prop:no_V_shaped}.
\end{proof}

\section{Shock fluctuations} \label{sec:shocks}
\subsection{Fluctuations of the shocks from fluctuations of the difference of two solutions}
To study the fluctuations of the shocks in Theorem \ref{thm:shock_fluctuations}, we first find the fluctuations of the quantity $h_+(0,t) - h_-(0,t)$. The following is a generalized version of \cite[Lemma~5.1]{Dunlap-Sorensen-24}, and allows us to translate between fluctuations of $h_+(0,t) - h_-(0,t)$ and fluctuations of the shock.
\begin{lemma}
\label{lem:h_to_b}Fix $\theta>0$. Let $\{\mathcal{J}(x,t)\st t\ge0,x\in\R\}$
be a real-valued stochastic process such that the following hold.
\begin{enumerate} [label=\rm(\roman{*}), ref=\rm(\roman{*})]  \itemsep=3pt
\item \label{enu:cont}For each fixed $t\ge0$, with probability $1$, $x\mapsto\mathcal{J}(x,t)$
is continuous and nondecreasing.
\item \label{enu:drift}For each fixed $t\ge0$, we have the almost sure limits $\lim\limits_{|x|\to\infty}\frac{\mathcal{J}(x,t)}{x}=4\theta$.
In particular, $\lim\limits_{x\to\pm\infty}\mathcal{J}(x,t)=\pm\infty$.
\item \label{enu:dist}For some exponent $\alpha>0$, $t^{-\alpha}\mathcal{J}(0,t)$
converges in distribution to an almost surely finite random variable
$Y$.
\item \label{enu:to0}Given the exponent $\alpha$ from Assumption~\ref{enu:dist},
for each $t\ge0$ and $\ve\in(0,4\theta)$, the random variable $t^{-\alpha}M_{t,\ve,\theta}$
converges to $0$ in probability, where
\[
M_{t,\ve,\theta}\coloneqq\sup_{x\in\R}\left[\left|\mathcal{J}(x,t)-\mathcal{J}(0,t)-4\theta x\right|-\ve|x|\right].
\]
 Note that $M_{t,\ve,\theta}$ is almost surely finite by Assumption~\ref{enu:drift}.
\end{enumerate}
For $t > 0$, define
\be \label{eq:btpm_def}
b_t^- = \inf\{x \in \R: \mathcal J(x,t) = 0\},\quad\text{and}\quad b_t^+ = \sup\{x \in \R: \mathcal J(x,t) = 0\},
\ee
noting that $-\infty < b_t^- \le b_t^+ < \infty$ by Assumptions~\ref{enu:cont} and~\ref{enu:drift}.
Then, as $t\to\infty$, $t^{-\alpha}(b_{t}^+ - b_t^-)$ converges in probability to $0$, and $t^{-\alpha} b_t^+$ converges in distribution to $-\frac{Y}{4\theta}$.
\end{lemma}
\begin{proof}
Let $\ve\in(0,4\theta)$. By the definition of $M_{t,\ve,\theta}$,
we have 
\begin{align}
-M_{t,\ve,\theta}+(4\theta-\ve)x & \le\mathcal{J}(x,t)-\mathcal{J}(0,t)\le M_{t,\ve,\theta}+(4\theta+\ve)x,\qquad x\ge0;\label{eq:positivepart}\\
-M_{t,\ve,\theta}+(4\theta+\ve)x & \le\mathcal{J}(x,t)-\mathcal{J}(0,t)\le M_{t,\ve,\theta}+(4\theta-\ve)x,\qquad x\le0.\label{eq:negativepart}
\end{align}
We consider three cases. If $\mathcal{J}(0,t)< 0$, then since $x\mapsto\mathcal{J}(x,t)$
is nondecreasing, we have $b_{t}^+\ge b_t^- > 0$. By \eqref{eq:positivepart},
this implies that 
\[
-M_{t,\ve,\theta}+(4\theta-\ve)b_{t}^\pm\le-\mathcal{J}(0,t)\le M_{t,\ve,\theta}+(4\theta+\ve)b_{t}^\pm,
\]
and so 
\[
\frac{-M_{t,\ve,\theta}-\mathcal{J}(0,t)}{4\theta+\ve}\le b_t^- \le b_{t}^+\le\frac{M_{t,\ve,\theta}-\mathcal{J}(0,t)}{4\theta-\ve}.
\]
Similarly, if $\mathcal{J}(0,t)>0$, then $b_t^- \le b_t^+ < 0$, and
\[
\frac{-M_{t,\ve,\theta}-\mathcal{J}(0,t)}{4\theta-\ve}\le b_{t}^- \le b_t^+\le\frac{M_{t,\ve,\theta}-\mathcal{J}(0,t)}{4\theta+\ve}.
\]
Lastly, when $\mathcal J(0,t) = 0$, then $b_t^- \le 0 \le b_t^+$, and
\[
-\f{M_{t,\ve,\theta}}{4\theta - \ve} \le b_t^- \le b_t^+ \le \f{M_{t,\ve,\theta}}{4\theta - \ve}.
\]
Thus, in all cases, we have
\begin{equation}
\frac{-M_{t,\ve,\theta}}{4\theta+\ve}+\min\left\{ \frac{-\mathcal{J}(0,t)}{4\theta-\ve},\frac{-\mathcal{J}(0,t)}{4\theta+\ve}\right\} \le b_{t}^- \le b_t^+\le\frac{M_{t,\ve,\theta}}{4\theta-\ve}+\max\left\{ \frac{-\mathcal{J}(0,t)}{4\theta-\ve},\frac{-\mathcal{J}(0,t)}{4\theta+\ve}\right\} .\label{eq:btbd}
\end{equation}
Now, Assumption~\ref{enu:to0} states that $t^{-\alpha}M_{t,\ve,\theta}$
converges to $0$ in probability for each fixed $\ve$, and Assumption~\ref{enu:dist}
states that $t^{-\alpha}\mathcal{J}(0,t)$ converges in distribution
to $Y$. Using these assumptions in \eqref{eq:btbd}, we see that
the collections of random variables $(t^{-\alpha}b_{t}^{\pm})_{t\ge1}$ are tight, and if $\mathfrak b^{\pm}$ are subsequential limits, then for any $\ve > 0$, we have the stochastic ordering
\[
\min\left\{ \frac{-Y}{4\theta-\ve},\frac{-Y}{4\theta+\ve}\right\} \lesssim \mathfrak b^- \lesssim \mathfrak b^+ \lesssim \max\left\{ \frac{-Y}{4\theta-\ve},\frac{-Y}{4\theta+\ve}\right\}.
\]
Letting $\ve\downarrow0$, we see that  $t^{-\alpha} b_t^{\pm}$ both converge in distribution to $-\f{Y}{4\theta}$. Since 
$b_t^- \le b_t^+$, we must have that $t^{-\alpha}(b_t^+ - b_t^-)$ converges to $0$ in probability. 
\end{proof}

In Proposition~\ref{prop:hpm_dif}, we have shown the limiting fluctuations for $h_+(0,t) - h_-(0,t)$ when $(f_-,f_+) \sim \nu_\theta$. We now prove analogous results for the two remaining cases in Theorem~\ref{thm:shock_fluctuations}. We first prove a lemma.
    \begin{lemma} \label{lem:flat_rescale}
For Let $f_0 \in \mathcal C_{\mathrm FP}$ be independent of $\F_{\ge 0}$,  and for $\theta \in \R$, let $f_\theta(x) = f_0(x) + 2\theta x$. Let $h_\theta$ denote the KPZ fixed point started from initial condition $f_\theta$ at time $0$. Then, for $t > 0$,
\[
\bigl(h_\theta(xt^{2/3},t)\bigr)_{x \in \R} \deq \bigl(t^{1/3} h(x,1 \mid \wt f_0) + 2\theta xt^{2/3} + \theta^2 t\bigr)_{x \in \R},
\]
where 
\[
\wt f_0(x) = t^{-1/3}f_0(\theta t + xt^{2/3}).
\]
In particular, when $f_0 \equiv 0$, we have $h(\cdot,1 \mid \wt f_0) = h_0(\cdot,1)$.
\end{lemma}
\begin{proof}
Below, we use the skew-stationarity, time stationarity, then rescaling from Lemma~\ref{lm:landscape_symm} to get the following distributional equalities, which each hold as processes in $x$:
\begin{align*}
    h_\theta(xt^{2/3},t) &= \sup_{y \in \R}[f_0(y) + 2\theta y + \Ll(y,0;xt^{2/3},t)] \\
    &\deq \sup_{y \in \R}[f_0(y) +\Ll(y,0;xt^{2/3} + \theta t,t)] + 2\theta xt^{2/3} + \theta^2 t \\
    &\deq \sup_{y \in \R}[f_0(y) + \Ll(y-\theta t,0;xt^{2/3},t)]+ 2\theta x t^{2/3} + \theta^2 t \\
    &= \sup_{z \in \R}[f_0(zt^{2/3} + \theta t) + \Ll(zt^{2/3},0;xt^{2/3},t)] + 2\theta xt^{2/3} + \theta^2 t \\
    &\deq t^{1/3} \sup_{z \in \R}[\wt f_0(z) + \Ll(z,0;x,1)] + 2\theta x + \theta^2 t \\
    &= t^{1/3} h_0(x,1 \mid \wt f_0) + 2\theta xt^{2/3} + \theta^2 t. \qedhere
\end{align*}
\end{proof}

\subsection{Fluctuations of the shock in the stationary case}

\begin{proposition} \label{prop:tilt_h=-_dif}
Let $(f_-,f_+) \sim \widehat \nu_\theta$, and let $h_\pm(\cdot,t) = h(\cdot,t \mid f)$. As $t\to \infty$, we have the following convergence in distribution
\[
\f{h_+(0,t) - h_-(0,t)}{t^{1/2}} \Longrightarrow \mathcal N(0,4\theta).
\]
\end{proposition}
\begin{proof}
Recall from the discussion of $\widehat \nu_\theta$ below \eqref{eq:M0def} that $(f_-(-x))_{x \ge 0}$ and $(f_+(x))_{x \ge 0}$ are independent and equal in law to $(W(x) + Y(x))_{x \ge 0}$, where $W \sim \BM$ and $Y \sim \BES^3(2\theta)$ are independent. One can also readily see the almost sure limits
\be\label{eq:fpmlim}
\lim_{|x| \to \infty} \f{f_{\pm}(x)}{x} = \pm 2\theta.
\ee
Let $f_1,f_2$ be defined by 
\[
f_1(x) = f_-(x) + 2\theta x,\qquad\text{and}\qquad f_2(x) = f_+(x) - 2\theta x.
\]
It suffices to show that
\begin{equation} \label{eq:f2f1lim}
\lim_{t \to \infty} \f{f_2(\theta t) - f_1(-\theta t)}{t^{1/2}} \Longrightarrow \mathcal N(0,4\theta),
\end{equation}
and then show that 
\begin{equation} \label{h+lim}
\f{h_+(0,t) - t\theta^2 - f_2(\theta t)}{t^{1/2}} \to 0 \qquad\text{and}\qquad  \f{h_-(0,t) - t\theta^2 - f_1(-\theta t)}{t^{1/2}} \to 0 \qquad\text{in probability}.
\end{equation}
We start with \eqref{eq:f2f1lim}: We have seen that $f_1(-\theta t)$ and $f_2(\theta t)$ are independent and equal in law, so it suffices to show that 
\begin{equation} \label{eq:f2lim}
\f{f_2(\theta t)}{t^{1/2}} \Longrightarrow \mathcal N(0,2\theta).
\end{equation}
Since $\theta t > 0$, we may write
\[
f_2(\theta t) = f_+(\theta t) - 2\theta^2 t = W(\theta t) + Y(\theta t) - 2\theta^2 t, 
\]
where $W$ is a standard Brownian motion, and $Y$ is an independent $
\BES^{3}(2\theta)$ process. We see immediately that $t^{-1/2}W(\theta t) \sim \mathcal N(0,\theta)$. Furthermore, by Equation \eqref{eq:YtoNor} of Lemma~\ref{lem:Bess_lin_bd}, 
\begin{align*}
\f{Y(\theta t) - 2\theta^2 t}{t^{1/2}} \Longrightarrow \mathcal N(0,\theta).
\end{align*}
Thus, we have proved \eqref{eq:f2lim} and therefore also \eqref{eq:f2f1lim}. 

We turn to proving \eqref{h+lim}. We prove the first statement, and the proof of the second is symmetric.  By Lemma~\ref{lem:flat_rescale}, 
\begin{equation} \label{eq:hplus}
\begin{aligned}
h_+(0,t) - \theta^2t - f_2(\theta t) 
\deq t^{1/3} \sup_{y \in \R}[ \wt f(y,t) + \Ll(y,0;0,1)],
\end{aligned}
\end{equation}
where we define
\[
\wt f(y,t) := t^{-1/3}\bigl(f_2(t^{2/3} y + \theta t) - f_2(\theta t)\bigr).
\]
Dividing by $t^{1/2}$, we get 
\begin{equation} \label{eq:1/6fty}
\f{h_+(0,t) - t\theta^2 - f_2(\theta t)}{t^{1/2}} \deq t^{-1/6} \sup_{y \in \R}[\wt f(y,t) + \Ll(y,0;0,1)].
\end{equation}
Let $A_{1,t}'$ be the event where 
\[
h_+(0,t) = \sup_{y \ge 0}[f_2(y) + 2\theta y + \Ll(y,0;0,t)],
\]
and let $A_{1,t}$ be the event on which
\begin{equation} \label{eq:A1tp}
\sup_{y \in \R}[\wt f(y,t) + \Ll(y,0;0,1)] = \sup_{y \ge -\theta t^{1/3}}[\wt f(y,t) + \Ll(y,0;0,1) ].
\end{equation}
From the scaling in Lemma~\ref{lem:flat_rescale} that gives us \eqref{eq:hplus}, we see that $\Pp(A_{1,t}) = \Pp(A_{1,t}')$. By \eqref{eq:fpmlim} and Lemma~\ref{lem:unq}, $\Pp(A_{1,t}') \to 1$ as $t \to \infty$. For $C,\delta > 0$, 

\begin{equation} \label{eq:A2tp}
\text{Let $A_{t,2}(C,\delta)$ be the event on which, for all $y \ge -\theta t^{1/3}$},\quad \wt f(y,t) \le C + \delta |y|.
\end{equation}
Note first that, by setting $y = 0$ and noting that $\wt f(0,t) = 0$, 
\begin{equation} \label{eq:fL1}
\sup_{y \in \R}[ \wt f(y,t) + \Ll(y,0;0,1)] \ge \Ll(0,0;0,1).
\end{equation}
Furthermore, on the event $A_{1,t} \cap A_{2,t}(C,\delta)$, we see that
\begin{equation} \label{eq:fL2}
\begin{aligned}
  \sup_{y \in \R}\Bigl[ \wt f(y,t) + \Ll(y,0;0,1)\Bigr] &\le \sup_{y \ge - \theta t^{1/3}}\Bigl[C + \delta |y| + \Ll(y,0;0,1) \Bigr] \\
  &\le \sup_{y \in \R}\Bigl[C + \delta |y|  - y^2 + C'\log\bigl(2\sqrt{y^2 + 1} + 2\bigr) \log^{2/3}(\sqrt{y^2 + 1} + 2) \Bigr],
  \end{aligned}
\end{equation}
where $C'$ is the random constant from Lemma~\ref{lem:Landscape_global_bound}. Combining \eqref{eq:fL1} and \eqref{eq:fL2}, we see that there exists an almost surely finite random variable  $X(C,\delta)$ (whose law does not depend on $t$) such that, on the event $A_{1,t} \cap A_{2,t}(C,\delta)$, 
\begin{equation} \label{eq:fL3}
\Biggl|\sup_{y \in \R}\Bigl[ \wt f(y,t) + \Ll(y,0;0,1)\Bigr]\Biggr| \le X(C,\delta).
\end{equation}
Then, combining \eqref{eq:1/6fty} and \eqref{eq:fL3}, for $\ve > 0$,
\begin{align*}
\Pp\Biggl(\Biggl|\f{h_+(0,t) - t\theta^2 - f_2(\theta t)}{t^{1/2}}\Biggr| \ge \ve \Biggr) \le \Pp(A_{1,t}^c) + P(A_{2,t}(C,\delta)^c) + \Pp(X(C,\delta) \ge \ve t^{1/6}).
\end{align*}
We have seen already that $\Pp(A_{1,t}^c) \to 0$, and since the law of $X(C,\delta)$ does not depend on $t$, we have  $\Pp(X(C,\delta) \ge \ve t^{1/6}) \to 0$. Hence, we complete the proof by showing that, for every $\eta > 0$, we may choose $C > 0$ and $\delta > 0$ so that 
\begin{equation} \label{eq:PA2to1}
\limsup_{t \to \infty} \Pp(A_{2,t}(C,\delta)^c) \le \eta.
\end{equation} First recall that $f_2$ for $y \ge 0$ has the law of $W(y) + Y(y) - 2\theta y$, where $W$ is a standard Brownian motion, and $Y$ is an independent $\BES^{3}(2\theta)$ process. Then, as a process in $y$ for $y \ge -\theta t^{1/3}$,
\begin{align*}
\wt f(y,t) &= t^{-1/3}\bigl(f_2(t^{2/3} y + \theta t) - f_2(\theta t)\bigr) \\
&\deq  t^{-1/3} \Bigl( W(t^{2/3} y + \theta t) - W(\theta t) + Y(t^{2/3} y + \theta t) - Y(\theta t)  - 2\theta t^{2/3} y\Bigr).
\end{align*}
By scaling and shift invariance of Brownian motion, the law of $ t^{-1/3} \bigl( W(t^{2/3} y + \theta t) - W(\theta t)\bigr)$ is the same for all $t$. By sub-linearity Brownian motion and Equation \eqref{eq:Ybd} of Lemma~\ref{lem:Bess_lin_bd}, we may choose $\delta = 1$ and sufficiently large $C$ so that $\Pp(A_{2,t}(C,\delta))\ge 1-\eta$, completing the proof. 
\end{proof}

We now prove Condition~\ref{enu:to0} of Lemma~\ref{lem:h_to_b} for the conditions where $(f_-,f_+) \sim \nu_\theta$ or $\widehat \nu_\theta$. 

\begin{lemma} \label{lem:nnutheta_Mto0}
    Let $(f_-,f_+) \sim \nu_\theta$, and let $h_\pm$ be the KPZ fixed point started from these initial conditions. Then, for each $\ve > 0$, as $t \to \infty$,
    \[
    t^{-1/2}\sup_{x \in \R} \Bigl[|h_+(x,t) - h_+(0,t) - (h_-(x,t) -h_-(0,t)) - 4\theta x|-\ve|x|\Bigr]\to 0 \quad\text{ in probability}.
    \]
\end{lemma}
\begin{proof}
By the triangle inequality, it suffices to show that 
\begin{align*}
&t^{-1/2}\sup_{x \in \R} \Bigl[|h_+(x,t) - h_+(0,t) - 2\theta x|-\f{\ve}{2}|x|\Bigr]\to 0,\quad\text{and}\\
&t^{-1/2}\sup_{x \in \R} \Bigl[|h_-(x,t) - h_-(0,t) + 2\theta x|-\f{\ve}{2}|x|\Bigr]\to 0, \quad\text{in probability}
\end{align*}
We prove the first limit, the second following a symmetric proof. By invariance of Brownian motion with arbitrary drift and diffusion coefficient $\sqrt 2$ under the KPZ fixed point, we have that 
\begin{align*}
&\quad \,t^{-1/2}\sup_{x \in \R} \Bigl[|h_+(x,t) - h_+(0,t) - 2\theta x|-\f{\ve}{2}|x|\Bigr] \deq t^{-1/2}\sup_{x \in \R} \Bigl[|B(x)|-\f{\ve}{2}|x|\Bigr],
\end{align*}
where $B \sim \BM(0,\sqrt 2)$. The quantity $\sup_{x \in \R} \Bigl[|B(x)|-\f{\ve}{2}|x|\Bigr]$ is almost surely finite since $B(x) = o(|x|)$ as $|x| \to \infty$, and this completes the proof. 
\end{proof}

\begin{lemma} \label{hatnutheta_Mto0}
     Let $(f_-,f_+) \sim \widehat \nu_\theta$, and let $h_\pm$ be the KPZ fixed point started from these initial conditions. Then, for each $\ve > 0$, as $t \to \infty$,
    \[
    t^{-1/2}\sup_{x \in \R} \Bigl[|h_+(x,t) - h_+(0,t) - (h_-(x,t) -h_-(0,t)) - 4\theta x|-\ve|x|\Bigr]\to 0 \quad\text{ in probability}.
    \]
\end{lemma}
\begin{proof}
    As in the proof of Lemma~\ref{lem:nnutheta_Mto0}, we prove that 
    \be \label{eq:goal1}
    t^{-1/2} \sup_{x \in \R}\Bigl[ |h_+(x,t) - h_+(0,t) - 2\theta x| - \f{\ve}{2} |x| \Bigr] \to 0,\quad\text{in probability},
    \ee
    and the analogous statement for $h_-$ is proved similarly. Define the function $f_0$ by $f_0(x) = f_+(x) - 2\theta x$. By a change of variable, followed by an application of Lemma~\ref{lem:flat_rescale}, we have 
    \be \label{eq:ineq123}
    \begin{aligned}
        &\quad \, t^{-1/2} \sup_{x \in \R}\Bigl[ |h_+(x,t) - h_+(0,t) - 2\theta x| - \f{\ve}{2} |x| \Bigr] \\
        &=  \sup_{x \in \R}\Bigl[ t^{-1/2}|h_+(xt^{2/3},t) - h_+(0,t) - 2\theta xt^{2/3}| - \f{\ve}{2} |x| t^{1/6} \Bigr] \\
        &\deq \sup_{x \in \R}\Bigl[t^{-1/6}|h(x,1\mid \wt f^t) - h(0,1 \mid \wt f^t)| - \f{\ve}{2} |x| t^{1/6}  \Bigr] \\
        &\le \sup_{x \in \R}\Bigl[t^{-1/6}|h(x,1\mid \wt f^t)| +t^{-1/6} |h(0,1 \mid \wt f^t)| - \f{\ve}{2} |x| t^{1/6}  \Bigr],
        \end{aligned}
    \ee
    where 
    \[
    \wt f^t(y) := t^{-1/3}\bigl(f_0(yt^{2/3} + \theta t) - f_0(\theta t)\bigr) = t^{-1/3}\bigl(f_+(\theta t + yt^{2/3}) - f_+(\theta t) - 2\theta yt^{2/3}\bigr).
    \] 
    Note that, after the direct application of Lemma~\ref{lem:flat_rescale}, we have added and subtracted the term $t^{-1/3} f_0(\theta t)$. Then, from the description of the measure $\widehat \nu_\theta$, we have that, as a process in $y$, 
    \[
    \wt f^t(y) \deq t^{-1/3}\Bigl(W(\theta t + yt^{2/3}) - W(\theta t) + Y(\theta t + yt^{2/3}) - Y(\theta t) - 2\theta yt^{2/3} \Bigr),
    \]
    where $W \sim \BM$, and $Y = (Y(x))_{x \ge 0}$ is and independent process that is a $\BES^3(2\theta)$ process for $x \ge 0$ and a $\BM(2\theta)$ process for $x < 0$. Then, by Brownian rescaling and Lemma~\ref{lem:Bess_lin_bd},  for any $\eta > 0$, we may choose $C > 0$ large enough so that $\Pp(|\wt f^t(y)| \le C + |y|,\;\;\forall\, y \in \R) \ge 1 - \eta$. Then, by Lemma~\ref{lem:KPZ_linear_preserve}, we may choose $A,B > 0$ such that 
    \[
    \Pp\Bigl(|h(x,1 \mid \wt f^t)| \le A + B|x|,\;\forall\, x \in \R) \ge 1 - 2\eta.
    \]
    Then, with probability at least $1 - 2\eta$, we can bound the last term in \eqref{eq:ineq123} as follows:
    \[
    \sup_{x \in \R}\Bigl[t^{-1/6}|h(x,1\mid \wt f^t)| +t^{-1/6} |h(0,1 \mid \wt f^t)| - \f{\ve}{2} |x| t^{1/6}  \Bigr] \le \sup_{x \in \R}\Bigl[ t^{-1/6}(2A + B|x|) - \f{\ve}{2}|x|t^{1/6}\Bigr] \overset{t \to \infty}{\longrightarrow} 0.
    \]
    Since \eqref{eq:goal1} is nonnegative, this completes the proof. 
\end{proof}

\subsection{Fluctuations of the shock in the flat case}

 \begin{proposition} \label{prop:flat_h_dif}
     Let $h_\pm(x,t)$ be the KPZ fixed point started from initial data $\pm 2\theta x$. Then, as $t \to \infty$,
     \[
     \f{h_+(0,t) - h_-(0,t)}{t^{1/3}} \Longrightarrow \f{X_1 - X_2}{2^{2/3}},
     \] 
     where $X_1$ and $X_2$ are independent Tracy-Widom GOE random variables. 
 \end{proposition}
\begin{proof}
  We will first create a coupling of three copies of the directed landscape, which we call $\Ll,\Ll_-,\Ll_+$. We start with approximation via Poisson last-passage percolation,  similar to the approach in the proof of in \cite[Proposition~2.6]{Dauvergne-2024}. Let $X_-,X_+$ be independent rate-$1$ Poisson processes in the plane $\R^2 = \{(x,t): x,t \in \R\}$. Let $X$ be the Poisson process in  that uses points of $X_-$ to the left of the main diagonal line $\{x= t\}$, and uses points of $X_+$ to the right of the main diagonal. Let $L^N_-,L^N_+,L^N$ be the associated rescaled versions of Poisson last-passage percolation in \eqref{eq:rescale_PLPP}. By Proposition~\ref{prop:Poisson_to_DL}\ref{itm:LPP_conv}, each of $L^N_-,L^N_+,L^N$ converge in distribution, with respect to the topology of uniform convergence on compact sets of $\Rup$, to the directed landscape. Hence, the joint law $(L^N_-,L^N_+,L^N,X_-,X_+,X)$ is tight with respect to the product topology, where the topology on the first three components in uniform on compact sets of $\Rup$, and the last three components are given the standard topology on locally finite point processes. Let  $(\Ll_-,\Ll_+,\Ll,X_-,X_+,X)$ be some subsequential limit,  and it follows immediately that $\Ll,\Ll_-,\Ll_+$ are each marginally directed landscapes and $\Ll_-$ and $\Ll_+$ are independent. By Skorokhod representation (see e.g.~\cite[Theorem~11.7.2]{dudl} or~\cite[Theorem~3.1.8]{EKbook}), we can find fields $(\overline L^N_-,\overline L^N_+,\overline L^N,\overline X_-^N,\overline X^N_+,\overline X^N)$ on some probability space $(\Omega,\Ff,\Pp)$  such that $(\overline L^N_-,\overline L^N_+,\overline L^N,\overline X_-^N,\overline X^N_+,\overline X^N) \deq ( L^N_-,L^N_+,L^N,X_-, X_+,X)$ and 
  such that the convergence $\overline L^N_{\pm} \to \Ll_{\pm}$ and $\overline L^N \to \Ll$ holds almost surely in the sense of uniform convergence on compact sets of $\Rup$. For notational simplicity, since we never use the coupling of $( L^N_-,L^N_+,L^N,X_-, X_+,X)$ between different values of $N$, we will drop the bar notation and write $( L^N_-, L^N_+, L^N, X_-^N, X^N_+, X^N) = (\overline L^N_-,\overline L^N_+,\overline L^N,\overline X_-^N,\overline X^N_+,\overline X^N)$.

Fix $\ve  \in (0, \theta)$. Let $A_{t,1}$ be the event on which 
\[
\begin{aligned}
h_-(0,t) &= \sup_{y \in [(-\theta - \ve)t,(-\theta + \ve)t ]}[f_-(y) + \Ll(y,0;0,t)],\quad \text{and} \\
h_+(0,t) &= \sup_{y \in [(\theta - \ve)t,(\theta + \ve)t]}[f_+(y) + \Ll(y,0;0,t)].
\end{aligned},
\]
where we recall $f_{\pm}(x) = \pm 2\theta x$. 
By Lemma~\ref{lem:unq}, $\lim_{t \to \infty} \Pp(A_{t,1}) = 1$. Recalling the definition of the function $d_p$ from \eqref{eq:dp_def},  it follows immediately from definition that on the event $A_{t,1}$, we have $d_{(-\theta + \ve)t}(0,t;f_-) \le 0$ and $d_{(\theta -\ve)t}(0,t;f_+) \ge 0$. By spatial monotonicity of the function $d_p$, we have $d_{(-\theta + \ve)t}(-t^{1/2},t;f_-) \le 0$ and $d_{(\theta -\ve)t}(t^{1/2},t;f_+) \ge 0$ as well, so that, on $A_{t,1}$,
\begin{equation} \label{eq:h-pt12}
\begin{aligned}
h_-(-t^{1/2},t) &= \sup_{y \le (-\theta + \ve)t}[f_-(y) + \Ll(y,0;-t^{1/2},t)],\quad \text{and} \\
h_+(t^{1/2},t) &= \sup_{y \ge  (\theta - \ve)t}[f_+(y) + \Ll(y,0;t^{1/2},t)].
\end{aligned}
\end{equation}
Next, define $A_{t,2}$ to be the event on which the following hold. 
\begin{enumerate}
\item For both $\Ll_-$ and $\Ll$, there is a unique geodesic from $((-\theta + \ve)t,0)$ to $(-t^{1/2},t)$, and the geodesics for both $\Ll_-$ and $\Ll$ lie entirely in the set $\{(x,s): x \le -t^{1/4}\}$.
\item For both $\Ll_+$ and $\Ll$, there is a unique geodesic from $((\theta - \ve)t,0)$ to $(t^{1/2},t)$, and the geodesics for both $\Ll_+$ and $\Ll$ lie entirely in the set $\{(x,s): x \ge t^{1/4}\}$. 
\end{enumerate}
Recall here that $\ve  \in (0,\theta)$ has been fixed, so we exclude it from the notation for $A_{t}$. By Lemmas~\ref{lem:geodesics} and~\ref{lem:geodesic_max}, 
\[
\liminf_{t \to \infty} \Pp(A_{t,2})  = 1.
\]
By Proposition~\ref{prop:Poisson_to_DL}\ref{itm:geod_conv}, on the event $A_{t,2}$, the following holds for all sufficiently large $N$.
\begin{enumerate}
\item For both $L^N_-$ and $L^N$, the rightmost geodesic (in rescaled coordinates) from $((-\theta + \ve)t,0)$ to $(-t^{1/2},t)$ lies entirely in the set $\{(x,s): x \le 0\}$. By monotonicity of geodesics, all geodesics for $L^N_-$ and $L^N$ from $(y,0)$ to $(-t^{1/2},t)$ for $y \le (-\theta + \ve)t$ also lie in the set $\{(x,s): x \le 0\}$. Thus, the unscaled geodesic lies in the set $\{(x,s): x \le s\}$ and both use only points of $X_-^N$.
\item For both $L^N_+$ and $L^N$, the leftmost geodesic (in rescaled coordinates) from $((\theta - \ve)t,0)$ to $(t^{1/2},t)$ lies entirely in the set $\{(x,s): x \ge 0\}$. By monotonicity of geodesics, all geodesics for $L^N_+$ and $L^N$ from $(y,0)$ to $(t^{1/2},t)$ for $y \ge ((\theta -\ve)t,0)$ lie in the set $\{(x,s): x \ge 0\}$. Thus, the unscaled geodesic lies in the set $\{(x,s): x \ge s\}$, and both use only points of $X_+^N$.
\end{enumerate}
Then,  on the event $A_{t,2}$, for all sufficiently large $N$, we have 
\begin{align*}
L^N_-(y,0;-t^{1/2},t) &= L^N(y,0;-t^{1/2},t) \quad\text{for all}\quad y \le (-\theta + \ve)t, \quad\text{and} \\
L^N_+(y,0;t^{1/2},t) &= L^N(y,0;t^{1/2},t)\quad\text{for all}\quad y \ge (\theta -\ve)t.
\end{align*}
Taking limits as $N \to \infty$, on the event $A_{t,2}$, we have 
\begin{align*}
\Ll_-(y,0;-t^{1/2},t) &= \Ll(y,0;-t^{1/2},t) \quad\text{for all}\quad y \le (-\theta + \ve)t, \quad\text{and} \\
\Ll_+(y,0;t^{1/2},t) &= \Ll(y,0;t^{1/2},t)\quad\text{for all}\quad y \ge (\theta -\ve)t.
\end{align*}
so by \eqref{eq:h-pt12}, on the high probability event $A_{t,1} \cap A_{t,2}$, we have
\be \label{eq:hpm_phipm}
\begin{aligned}
h_-(-t^{1/2},t) &= \phi_-(-t^{1/2},t) := \sup_{y \le (-\theta + \ve)t}[f_-(y) + \Ll_-(y,0;-t^{1/2},t)],\quad \text{and} \\
h_+(t^{1/2},t) &= \phi_+(t^{1/2},t) := \sup_{y \ge  (\theta - \ve)t}[f_+(y) + \Ll_+(y,0;t^{1/2},t)],
\end{aligned}
\ee
and we note that $\phi_-(-t^{1/2},t)$ and $\phi_+(t^{1/2},t)$ are independent.

Next, by Lemma~\ref{lem:flat_rescale}, we have 
\be \label{eq:hpm_dif}
\begin{aligned}
&h_+(t^{1/2},t) - h_+(0,t) \deq t^{1/3}[h_0(t^{-1/6},1) - h_0(0,1) ] + 2\theta t^{1/2},\quad\text{and} \\
&h_-(-t^{1/2},t) - h_-(0,t) \deq t^{1/3}[h_0(-t^{-1/6},1) - h_0(0,1)] + 2\theta t^{1/2},
\end{aligned}
\ee
where $h_0$ is the KPZ fixed point started from $0$ (flat) initial data. By continuity of the KPZ fixed point (Proposition~\ref{prop:h_pres_CFP}), combined with \eqref{eq:hpm_phipm}, since $\Pp(A_{t,1} \cap A_{t,2}) \to 1$ as $t \to \infty$, we have that 
\be \label{phipmdif}
\f{\phi_+(t^{1/2},t) -h_+(0,t) - 2\theta t^{1/2}}{t^{1/3}}\to 0,\quad\text{and}\quad \f{\phi_-(-t^{1/2},t) -h_-(0,t) - 2\theta t^{1/2}}{t^{1/3}} \to 0,
\ee
in probability. It therefore suffices to show that 
\[
\f{\phi_+(t^{1/2},t) - \phi_-(-t^{1/2},t)}{t^{1/3}} \Longrightarrow \f{X_1 - X_2}{2^{2/3}},
\]
where $X_1$ and $X_2$ are independent Tracy-Widom GOE random variables. Since $\phi_-$ and $\phi_+$ are independent, it further suffices to show that 
\[
\f{\phi_{\pm}(\pm t^{1/2},t) - \theta^2 t - 2\theta t^{1/2}}{t^{1/3}}
\]
each marginally converge to $2^{-2/3} X$, where $X$ is a Tracy-Widom GOE random variable. This follows by \eqref{phipmdif} and the fact that $t^{-1/3}(h_{\pm}(0,t) - \theta^2 t)$ both have the distribution of $2^{-2/3} X$ in the prelimit. Indeed, by Lemma~\ref{lem:flat_rescale}, we have that 
\[
t^{-1/3}(h_{\pm}(0,t) - \theta^2 t) \deq h_0(0,1) = \sup_{y \in \R}[\Ll(y,0;0,1)] \deq \sup_{y \in \R}[\Ll(y,-1;0,0)] ]\deq \sup_{y \in \R}[\Ll(0,0;y,1)],
\]
where the last two steps are the temporal reflection and shift-stationarity of the directed landscape in Lemma \ref{lm:landscape_symm} (originally from \cite{Directed_Landscape}).

By the construction of the directed landscape (see, specifically, \cite[Definitions 8.1(1) and 10.1(1)]{Directed_Landscape}), the process $y \mapsto \Ll(0,0;y,1)$ is distributed as the $\text{Airy}_2$ process (sometimes also called the $\text{Airy}_2$ process minus a parabola).
Then, by \cite[Corollary 1.3]{Joansson-03}, $h_0(0,1)$ has the law of $2^{-2/3}X$ (note the scaling factor $2^{-2/3}$ was missing in \cite{Joansson-03}; this was corrected in \cite[Section 2]{Corwin-Quastel-Remenik-2013}). 
\end{proof}

\begin{lemma} \label{lem:max_to_infty}
For $\theta > 0$, let $h_\pm$ denote the KPZ fixed point, started from initial data $\pm 2\theta x$. Then, for each $\ve > 0$, as $t \to \infty$
\[
t^{-1/3} \sup_{x \in \R}\bigl[|h_+(x,t) - h_+(0,t) - (h_-(x,t) - h_-(0,t)) - 4\theta x| - \ve |x|\bigr]\to 0\quad\text{in probability.}
\]
\end{lemma}
\begin{proof}
By the triangle inequality, it suffices to prove that 
\begin{align*}
& t^{-1/3} \sup_{x \in \R}\Bigl[|h_+(x,t) - h_+(0,t) - 2\theta x| - \f{\ve}{2} |x|\Bigr]\to 0 \quad\text{and}\\
&t^{-1/3} \sup_{x \in \R}\Bigl[|h_-(x,t) - h_-(0,t) + 2\theta x| - \f{\ve}{2} |x|\Bigr] \to 0, \quad\text{in probability}.
\end{align*}
 We prove the first limit, with the second following a symmetric proof. 
Applying a change of variable in the first line below, followed by an application of Lemma~\ref{lem:flat_rescale}, 
\begin{align*}
&\quad\, t^{-1/3}\sup_{x \in \R}\Bigl[|h_+(x,t) - h_+(0,t) - 2\theta x| - \f{\ve}{2} |x|\Bigr] \\
&= t^{-1/3}\sup_{x \in \R}\Bigl[|h_+(xt^{2/3},t) - h_+(0,t) - 2\theta x t^{2/3}| - \ve |x|t^{2/3}\Bigr] \\
&\deq \sup_{x \in \R}\bigl[|h_0(x,1) - h_0(0,1)| - \ve|x|t^{1/3}  \bigr].
\end{align*}
It therefore suffices to show that this last quantity converges to $0$ almost surely. By evaluating at $x = 0$, we get
\be \label{eq:h0bd1}
\sup_{x \in \R}\bigl[|h_0(x,1) - h_0(0,1)| - \ve|x|t^{1/3}  \bigr] \ge 0
\ee
Next, let $\delta > 0$ be arbitrary. By the spatial continuity of $h_0$ (Proposition~\ref{prop:h_pres_CFP}), we may choose a random $a > 0$ sufficiently small so that $|h_0(x,1) - h_0(0,1)| \le \delta$ for all $|x| \le a$. By the preservation of slopes in Lemma~\ref{lem:slope_conserved}, we have $\lim_{|x| \to \infty} \f{h_0(x,1)}{x} = 0$, so there exist random $A,B > 0$ so that $|h_0(x,1) - h_0(0,1)| \le A + B|x|$. Hence, for all sufficiently large $t$, 
\[
\sup_{|x| > a}[|h_0(x,1) - h_0(0,1)| - \ve|x|t^{1/3} ] < 0.
\]
Combined with \eqref{eq:h0bd1}, we have, for sufficiently large $t$,
\[
0 \le \sup_{x \in \R}\bigl[|h_0(x,1) - h_0(0,1)| - \ve|x|t^{1/3}  \bigr] \le \sup_{|x| \le a}[\delta  - \ve |x|t^{1/3}] = \delta,
\]
and this completes the proof since $\delta > 0$ is arbitrary. 
\end{proof}

\subsection{Proof of Theorem~\ref{thm:shock_fluctuations}}
Here, we put all of the pieces together for the proof of Theorem~\ref{thm:shock_fluctuations}. Each item follows by an application of Lemma~\ref{lem:h_to_b} for the function $\mathcal J(x,t) = h_+(x,t) - h_-(x,t)$, where $h_{\pm}$ is the KPZ fixed point from initial conditions $\underline f = (f_-,f_+)$ distributed according to three options. For Item~\ref{itm:stat_shock}, we take $\underline f \sim \widehat \nu_\theta$, for Item~\ref{itm:flat_shock}, we take $f_\pm(x) = \pm 2\theta x$, and for Item~\ref{itm:joint_stat_shock}, we take $\underline f \sim \nu_\theta$. In each case, $\underline f \in \mathcal X(\theta)$ almost surely (Proposition~\ref{prop:nuthetainvariant} and Lemma~\ref{nutheta_fact}), and by the preservation of the space $\mathcal X(\theta)$ under the KPZ fixed point (Lemmas~\ref{lem:slope_conserved}--\ref{KPZFP_attr}), Conditions~\ref{enu:cont} and~\ref{enu:drift} of Lemma~\ref{lem:h_to_b} are satisfied. 

Define $b_t^{\pm}$ as in \eqref{eq:btpm_def}. In the case $f_{\pm} = \pm 2 \theta x$, it is immediate that $b_t^- = b_t^+ = 0$, and in the case $(f_-,f_+) \sim \widehat \nu_\theta$, Lemma~\ref{nutheta_fact} implies that $b_0^- = b_0^+ = 0$, almost surely. Hence, in all cases, Proposition~\ref{prop:interfaces_are_shocks} implies that this notion of $b_t^{\pm}$ matches the definition in the statement of the theorem.

\noindent For Item~\ref{itm:stat_shock}, $\alpha = \f{1}{2}$, and Conditions~\ref{enu:dist}--\ref{enu:to0} of Lemma~\ref{lem:h_to_b} are satisfied by Proposition~\ref{prop:tilt_h=-_dif} and Lemma~\ref{hatnutheta_Mto0}. 

\noindent For Item~\ref{itm:flat_shock}, $\alpha = \f{1}{3}$, and Conditions~\ref{enu:dist}--\ref{enu:to0} are satisfied by Proposition~\ref{prop:flat_h_dif} and Lemma~\ref{lem:max_to_infty}.

\noindent For Item~\ref{itm:joint_stat_shock}, $\alpha = \f{1}{2}$, and Conditions~\ref{enu:dist}--\ref{enu:to0} are satisfied by Proposition~\ref{prop:hpm_dif} and Lemma~\ref{lem:nnutheta_Mto0}. \qed

\appendix
\section{The directed landscape and KPZ fixed point} \label{sec:DLKPZFP}
Here, we collect some auxiliary results regarding the directed landscape and KPZ fixed point.

\begin{lemma}\cite[Lemma~10.2]{Directed_Landscape},  \cite[Proposition~1.23]{Dauvergne-Virag-21}  \label{lm:landscape_symm}
As a random continuous function of $(y,s;x,t) \in \Rup$, the directed landscape $\Ll$ satisfies the following distributional symmetries, for all  $r,c \in \R$ and $q > 0$.
\begin{enumerate} [label=\rm(\roman{*}), ref=\rm(\roman{*})]  \itemsep=3pt
    \item {\rm(Space-time stationarity)}  \label{itm:time_stat} \ \ $\Ll(y,s;x,t) \deq \Ll(y+c,s + r;x+c,t + r).
    $
    \item {\rm(Skew stationarity)} \label{itm:skew_stat}
    \ \ $
    \Ll(y,s;x,t) \deq \Ll(y + cs,s;x + ct,t) -2c(y - x) + (t- s)c^2.  
    $
    \item \label{itm:DL_reflect} {\rm(Spatial and temporal reflections)} 
    \ \ $
    \Ll(y,s;x,t) \deq \Ll(-y,s;-x,t) \deq \Ll(x,-t;y,-s).
    $
    \item \label{itm:DL_rescaling} {\rm(Rescaling)} 
    \ \ $
    \Ll(y,s;x,t) \deq q\Ll(q^{-2}y,q^{-3}s;q^{-2}x,q^{-3}t).
    $
\end{enumerate}
\end{lemma}

\begin{lemma}\cite[Corollary 10.7]{Directed_Landscape},  \label{lem:Landscape_global_bound}
There exists a random constant $C$ satisfying
\[
\Pp(C > m) \le ce^{-dm^{3/2}}
\]
for universal constants $c,d$, such that for all $v = (y,s;x,t) \in \Rup$, we have 
\[
\Bigl|\Ll(y,s;x,t) + \f{(y - x)^2}{t - s}\Bigr| \le C (t - s)^{1/3} \log^{4/3} \Bigl(\f{2(\|v\| + 2)}{t - s}\Bigr)\log^{2/3}(\|v\| + 2),
\]
where $\|v\|$ is the Euclidean norm.
\end{lemma}

\begin{lemma} \cite[Lemma~B.9]{Busa-Sepp-Sore-22arXiv} \cite[Lemma B.7]{Busa-Sepp-Sore-22a} \label{lem:slope_conserved}
Let $f \in \CFP$, independent of $\{\Ll(y,0;x,t):y,x \in \R, t > 0\}$. Then, for all $t > 0$,
\[
\lim_{x \to \pm \infty} \f{h(x,t \mid f)}{x} = \lim_{x \to \pm \infty} \f{f(x)}{x},
\]
whenever both of the limits on the right-hand side exist in $\R$. 
\end{lemma}

We cite the following lemma.
\begin{lemma} \cite[Proposition~4]{Pimentel-21a} \label{KPZFP_attr}
If $f_-(y) - f_-(x) \le f_+(y) - f_+(x)$ for all $y \ge x$, then $h_-(y,t) - h_-(x,t) \le h_+(y,t) - h_+(x,t)$ for all $t > 0$ and $y \ge x$, where $h_{\pm}(x,t) = h(x,t \mid f_{\pm})$.
\end{lemma}

\begin{lemma} \cite[Lemma B.6]{Busa-Sepp-Sore-22arXiv}, \cite[Lemma B.5]{Busa-Sepp-Sore-22a}\label{lem:unq}
	 Fix $\xi\in \R$ and $a>0$. Consider the KPZ fixed point $\h$ starting at time $s$ from a function $f \in \CFP$. 
For $t > s$, let $Z_f^{a,s,t}\subseteq\R$ denote the set of exit points from the time horizon $s$ of the geodesics associated with $\h$ from $\{t\}\times [-a,a]$. That is,
	\begin{equation} \label{exitpt}
		Z_f^{a,s,t}=\bigcup_{x\in [-a,a]} \argmax_{y\in\R} [f(y)+\Ll(y,s;x,t)].
	\end{equation}
	Then, on the full probability event of Lemma~\ref{lem:Landscape_global_bound}, whenever $f\in \CFP$ satisfies condition~\eqref{eqn:drift_assumptions}, and when  $\ve>0$, $a > 0$, and $s \in \R$, there exists  a random $t_0 = t_0(\ve,a,s) > s \vee 0$ such that for any $t> t_0$,
	\be \label{upexit}
	Z_f^{a,s,t}\subset \big[(\theta-\ve)t,(\theta +\ve)t\big].
	\ee
	In particular, if $f$ is a random function almost surely satisfying condition~\eqref{eqn:drift_assumptions}, then this random $t_0$ exists almost surely, and
\[
	\lim_{t \to \infty} \Pp\Big(Z_f^{a,s,t}\subset \big[(\theta-\ve)t,(\theta +\ve)t\big]\Big) = 1.
\]  
Furthermore, an analogous statement holds on the same full-probability event if $t$ is held fixed and $s \to -\infty$. That is, there exists a random $s_0 = s_0(\ve,a,t)< t \wedge 0 $ such that for any $s < s_0$,
\begin{equation} \label{downexit}
Z_f^{a,s,t}\subset \big[-(\theta-\ve)s,-(\theta +\ve)s\big]
\end{equation}
\end{lemma}

The following statement is weaker than that stated in \cite{Busa-Sepp-Sore-22a}, but we only include the statement we need for brevity.
\begin{lemma} \cite[Lemma~B.8(ii)]{Busa-Sepp-Sore-22arXiv}, \cite[Lemma B.6(ii)]{Busa-Sepp-Sore-22a} \label{lem:KPZ_linear_preserve}
   Assume $f \in \CFP$ satisfies $|f(x)| \le a + b|x|$ for all $x \in \R$, for some constants $a,b > 0$. Then, for any $0 < S < T$, there exists $A = A(a,S,T),B = B(b,S,T) > 0$ (measurable with respect to $\Ll$) such that $|h(x,t \mid f)| \le A + B|x|$ for all $t \in [S,T]$ and $x \in \R$.
\end{lemma}

In the following, for points $(y,s),(x,t)$, let $\{\gamma_{(x,s),(y,t)}(u): s \le u \le t\}$ denote the almost surely unique directed landscape geodesic between the points (as in Lemma~\ref{lem:geodesics}). 
\begin{lemma} \label{lem:geodesic_max}
Let $\alpha  \in (0,\f{1}{2})$ and $r > 0$. Then, 
\[
\liminf_{t \to \infty} \Pp\Bigl(\max_{s \in [0,t]}\gamma_{(-rt,0),(-t^{1/2},t)}(s) < -t^\alpha \Bigr) = 1.
\]
\end{lemma}
\begin{proof}
Let $\gamma$ denote $\gamma_{(-rt,0),(-t^{1/2},t)}$ for shorthand. We will now use invariance of $\Ll$ from Lemma~\ref{lm:landscape_symm} applied to the geodesics. By shift invariance, followed by shear invariance, then rescaling, $\gamma$ has the same law as 
\[
t^{2/3}\wh \gamma(s/t) + (r - t^{-1/2})s - rt,
\]
where $\wh \gamma:[0,1] \to \R$ is the directed landscape geodesic from $(0,0)$ to $(0,1)$. Hence,
\be \label{eq:gamma1}
\begin{aligned}
\Pp\Bigl(\max_{s \in [0,t]}\gamma(s) \ge -t^\alpha \Bigr) &= \Pp\Bigl(\max_{s \in [0,t]}\bigl[t^{2/3}\wh \gamma(s/t) + (r - t^{-1/2})s - rt\bigr] \ge -t^\alpha \Bigr) \\
&= \Pp\Bigl(\max_{u \in [0,1]} [\wh \gamma(u) + (r - t^{-1/2})t^{1/3} u] \ge -t^{\alpha - 2/3} + rt^{1/3}  \Bigr).
\end{aligned}
\ee
Let $\ve > 0$ be a small number chosen sufficiently small at a later point. Since $\wh \gamma(1) = 0$, by \cite[Proposition~12.3]{Directed_Landscape}, there exists a random constant $C$ so that, for all $u \in [0,1]$,
\[
|\wh \gamma(u)| \le C (1-u)^{2/3 - \ve}.
\]
Hence, 
\be \label{gamma2}
\begin{aligned}
&\quad \, \Pp\Bigl(\max_{u \in [0,1]} [\wh \gamma(u) + (r - t^{-1/2})t^{1/3} u] \ge -t^{\alpha - 2/3} + rt^{1/3}  \Bigr) \\
&\le \Pp\Bigl(\max_{u \in [0,1]}[C(1-u)^{2/3 - \ve} + (r - t^{-1/2})t^{1/3} u] \ge - t^{\alpha- 2/3} + rt^{1/3}\Bigr)
\end{aligned}
\ee
Now, let $\beta > 0$ be a small number to be determined. If $u \le 1 - t^{-\beta}$, then we have
\begin{align*}
&\quad \,C(1-u)^{2/3 - \ve} + (r - t^{-1/2})t^{1/3} u \\
&\le C + rt^{1/3} u \le C + rt^{1/3} - t^{1/3 - \beta},
\end{align*}
and so 
\be \label{Cge1}
\text{If}\quad C(1-u)^{2/3 - \ve} + (r - t^{-1/2})t^{1/3} u \ge - t^{\alpha- 2/3} + rt^{1/3}, \quad \text{then}\quad  C \ge t^{1/3 - \beta} - t^{\alpha - 2/3}.
\ee
On the other hand, if $ 1 - t^{-\beta} \le u \le 1$, then 
\begin{align*}
&\quad \, C(1-u)^{2/3 - \ve} + (r - t^{-1/2})t^{1/3} u  \\
&\le C(1-u)^{2/3 - \ve} + rt^{1/3} - t^{-1/6}u \\
&\le C t^{-\beta(2/3 - \ve)} + rt^{1/3} - t^{-1/6} + t^{-1/6 - \beta},
\end{align*}
so
\be \label{Cge2}
\begin{aligned}
\text{If}\quad &C(1-u)^{2/3 - \ve} + (r - t^{-1/2})t^{1/3} u \ge - t^{\alpha- 2/3} + rt^{1/3},\quad\text{then} \\
&C \ge  t^{\beta(2/3 - \ve)}\Bigl(t^{-1/6} - t^{\alpha - 2/3} - t^{-1/6 - \beta}\Bigr) \ge \f{1}{2}t^{\beta(2/3 - \ve) - 1/6},
\end{aligned}
\ee
where the last inequality holds for all sufficiently large $t$, noting that  $-\f{1}{6} > \alpha - \f{2}{3}$ since $\alpha < \f{1}{2}$ so the term in the parenthesis is positive and of order $t^{-1/6}$.
Hence, combining \eqref{eq:gamma1}, \eqref{gamma2}, \eqref{Cge1}, and \eqref{Cge2}, we see that 
\be \label{probgamma}
\Pp\Bigl(\max_{s \in [0,t]}\gamma(s) \ge -t^\alpha \Bigr) \le \Pp\Biggl(C \ge \min\Bigl(\f{1}{2}t^{\beta(2/3 - \ve) - 1/6},t^{1/3 - \beta} - t^{\alpha - 2/3}\Bigr)\Biggr).
\ee
Then, if we choose, for example, $\ve = \f{1}{9}$, and $\f{3}{10} < \beta < \f{1}{3}$, then since $\alpha \in (0,1/2)$, $\beta < 1-\alpha$, and we have
\[
\f{1}{3} - \beta > \max\Bigl(1,\alpha - \f{2}{3}\Bigr),\quad\text{and}\quad \beta\Bigl(\f{2}{3} - \ve\Bigr) > \f{1}{6}.
\]
Then, the probability in \eqref{probgamma} tends to $0$, as desired. 
\end{proof}

We finish this section by including a result from \cite{Dauvergne-Virag-21} about Poisson last-passage percolation. Let $X$ be a rate-$1$ Poisson process in the plane. For $\mbf p \le \mbf q \in \R^2$ (in coordinate-wise ordering), we let $d(\mbf p;\mbf q)$ be the maximal number of points of $X$ that can be collected in an up-right polygonal path from $\mbf p$ to $\mbf q$. We will consider geodesics in this path as maximal paths that are linear between any two points of $X$. 

\begin{proposition} \cite[Theorem~1.7]{Dauvergne-Virag-21}, \cite[Theorem~13.1]{Directed_Landscape} \label{prop:Poisson_to_DL}
Let $d$ denote Poisson last-passage percolation in the plane, built from the Poisson point process $X$ as described above. The following hold.
\begin{enumerate} [label=\rm(\roman{*}), ref=\rm(\roman{*})]  \itemsep=3pt
\item \label{itm:LPP_conv} There exist constants $\alpha,\beta,\tau,\chi > 0$ such that, if we define
\be \label{eq:rescale_PLPP}
L^N(y,s;x,t) := \f{1}{\chi N^{1/3}} \Bigl(d(sN + yN^{2/3}, sN; tN + xN^{2/3},tN) - \alpha N(t-s) - \beta N^{2/3}(x-y)\Bigr),
\ee
then there exists a coupling of copies of $L^N$ and $\Ll$ such that, as $N \to \infty$, $L^N$ converges to $\Ll$ with respect to the topology of uniform convergence on compact sets of $\Rup$.
\item \label{itm:geod_conv} In any such coupling where the convergence $L^N \to \Ll$ holds, let $\pi_N$ be the image of an arbitrary geodesic for $L$ under the linear map $F_N(x,t) := \bigl(N^{-2/3}(x-t),N^{-1}t\bigr)$. If, for $(\mbf p;\mbf q) \in \Rup$, the endpoints of $\pi_N$ converge to $\mbf p$ and $\mbf q$, then the sequence of geodesics $\pi_N$ (seen as a subset of $\R^2)$ is precompact in the Hausdorff topology on $\R^2$, and if there is a unique geodesic $\gamma$ from $\mbf p$ to $\mbf q$ for $\Ll$, then $\pi_N$ converges to $\gamma$ in the Hausdorff topology.
\end{enumerate}

\end{proposition}

\section{Busemann process and semi-infinite geodesics in the directed landscape} \label{sec:SIG}
\begin{proposition}\cite[Lemma~A.12]{Bhattacharjee-Busani-Sorensen-25}\label{prop:geodesics_from_b}
Let $b:\R^2 \to \R$ be a  function satisfying, for all $s < t$, and $x \in \R$
\be \label{eq:global_appx}
b(x,t) = \sup_{z \in \R} [b(z,s) + \Ll(z,s;x,t)], 
\ee
and assume that for each $t > s$ and $x \in \R$, leftmost and rightmost maximizers in the equation above exist in $\R$. Define $g_{(x,t)}^{b,L/R}(s)$ to be the leftmost/rightmost maximizers, and set 
\[
g_{(x,t)}^{b,L/R}(t) = x.
\]
Then the following statements hold:
\begin{enumerate} [label=\rm(\roman{*}), ref=\rm(\roman{*})]  \itemsep=3pt
\item \label{itm:geod} $g_{(x,t)}^{b,L}:(-\infty,t] \to \R$ \rm{(}resp. $g_{(x,t)}^{b,\mathrm{R}}:(-\infty,t] \to \R$\rm{)} is a semi-infinite geodesic that is the leftmost (resp. rightmost) directed landscape geodesic between any two of its points. In particular, $g_{(x,t)}^{b,L}$ and $g_{(x,t)}^{b,\mathrm{R}}$ are continuous functions $(-\infty,t] \to \R$.
\item \label{itm: consis} Let $s < r < t$ and $x \in \R$, and let $
w = g_{(x,t)}^{b,L}(r)$.
Then,
\[
g_{(x,t)}^{b,L}(s) = g_{(w,r)}^{b,L}(s),
\]
and the same holds when replacing $L$ with $R$. 
\end{enumerate}
\end{proposition}

\begin{proposition} \cite[Theorem~2.5(i)]{Busa-Sepp-Sore-22a} \label{prop:all_directions}
With probability $1$, every semi-infinite geodesic $g$ in the directed landscape has a direction $\theta \in \R$. That is,
\[
\lim_{s \to -\infty} \f{g(s)}{|s|} = \theta.
\]
\end{proposition}

\begin{proposition} \cite[Theorem~5.1, Theorem~5.3, Lemma~5.12]{Busa-Sepp-Sore-22a} \label{prop:Buse_basic_properties}
On the probability space of the directed landscape $\Ll$, there exists a process
\[
\Bigl(W^{\theta \sig}(y,s;x,t): \theta \in \R, \sigg \in \{-,+\}, (y,s;x,t) \in \R^4\Bigr)
\]
satisfying the following properties:
\begin{enumerate}[label=\rm(\roman{*}), ref=\rm(\roman{*})]  \itemsep=3pt
\item{\rm(Continuity)} \label{itm:general_cts}  As an $\R^4 \to \R$ function,  $(y,s;x,t) \mapsto \W^{\theta \sig}(y,s;x,t)$ is  continuous. 
 \item {\rm(Additivity)} \label{itm:DL_Buse_add} For all $p,q,r \in \R^2$, 
    $\W^{\theta \sig}(p;q) + \W^{\theta \sig}(q;r) = \W^{\theta \sig}(p;r)$.   In particular, \\ $\W^{\theta \sig}(p;q) = -\W^{\theta \sig}(q;p)$ and $\W^{\theta \sig}(p;p) = 0$.
    \item {\rm(Monotonicity along a horizontal line)}
    \label{itm:DL_Buse_gen_mont} Whenever $\theta_1< \theta_2$, $x < y$, and $t \in \R$,
    \[
    \W^{\theta_1 -}(x,t;y,t) \le \W^{\theta_1 +}(x,t;y,t) \le \W^{\theta_2 -}(x,t;y,t) \le \W^{\theta_2 +}(x,t;y,t).
    \]
    \item {\rm(Evolution as the KPZ fixed point)}\label{itm:Buse_KPZ_description} For 
    all $x,y \in \R$ and $s < t$,
    \be\label{W_var}
    \W^{\theta \sig}(y,s;x,t) = \sup_{z \in \R}\{\W^{\theta \sig}(y,s;z,s) +
 \Ll(z,s;x,t)\}.
    \ee
    \item \label{it:Wslope} {\rm(Law and asymptotic slope)} For fixed $\theta \in \R$, $W^{\theta} := W^{\theta +} = W^{\theta -}$ with probability $1$, and for each $t \in \R$, the law of the process $(x \mapsto W^{\theta +}(0,t;x,t))_{\theta \in \R}$ is the stationary horizon (Definition~\ref{def:SH}). In particular, for each $\theta \in \R$, the random function $x \mapsto W^{\theta}(0,t;x,t)$ is a two-sided Brownian motion with variance $2$ and drift $2\theta$. Furthermore, on a single event of full probability, for all $t \in \R,\theta \in \R$, and $\sigg \in \{-,+\}$,
    \[
    \lim_{|x| \to \infty} \f{W^{\theta \sig}(0,t;x,t)}{x} = 2\theta.
    \]
\end{enumerate}
\end{proposition}

The next cited result concerns the existence of infinite geodesics for the directed landscape. 
\begin{proposition}  \cite[Theorem~5.9]{Busa-Sepp-Sore-22a}\label{prop:DL_SIG_cons_intro}
 The following hold on a single event of probability $1$ across all initial points $(x,s) \in \R^2$, times $s < t$, directions $\theta  \in \R$, and signs $\sigg \in \{-,+\}$.
 \begin{enumerate} [label=\rm(\roman{*}), ref=\rm(\roman{*})]  \itemsep=3pt
 \item \label{itm:intro_SIG_bd} 
 All maximizers of $z \mapsto \W^{\theta \sig}(0,s;z,s)+ \Ll(z,s;x,t)$ are finite. Furthermore, as  $x,s,t$ vary over a compact set $K\subseteq \R$ with $s < t$, the set of all maximizers is bounded.
    \item \label{itm:arb_geod_cons}  
    Let $t = s_0 > s_1 > s_2 > \cdots$ be an arbitrary decreasing sequence with $\lim_{n \to \infty} s_n = -\infty$. Set $g(t) = x$, and for each $i \ge 1$, let $g(s_i)$ be \textit{any} maximizer of $z \mapsto W^{\theta \sig}(0,s_{i};z,s_{i}) + \Ll(z,s_i; g(s_{i - 1}),s_{i - 1})$ over $z \in \R$. Then, pick \textit{any} point-to-point geodesic of $\Ll$ from $(g(s_{i}),s_{i})$ to $(g(s_{i-1}),s_{i-1})$, and for $s_{i} < s < s_{i-1}$, let $g(s)$ be the location of this geodesic at time $s$. Then, regardless of the choices made at each step, the following hold.
    \begin{enumerate} [label=\rm(\alph{*}), ref=\rm(\alph{*})]
        \item \label{itm:g_is_geod} The path $g:(-\infty,t]\to \R$ is a semi-infinite geodesic.
        \item \label{itm:weight_of_geod} For all  $ s < u \le t$,
    \be \label{eqn:SIG_weight}
    \Ll(g(s),s;g(u),u) = \W^{\theta \sig}(g(s),s;g(u),u).
    \ee
    \item \label{itm:maxes} For all  $ s < u$ in $(-\infty,s]$, $g(u)$ maximizes $z \mapsto \W^{\theta \sig}(0,s;z,s) + \Ll(z,s;g(u),u) $ over $z \in \R$. 
    \item \label{itm:geo_dir} The geodesic $g$ has direction $\theta$, i.e., $g(s)/|s| \to \theta$ as $s \to -\infty$. 
    \end{enumerate}
    \end{enumerate}
    \end{proposition}
\begin{definition} \label{def:Buse_geodesics}
We refer to the  geodesics constructed  in Proposition~\ref{prop:DL_SIG_cons_intro}\ref{itm:arb_geod_cons} as $\theta \sig$ \textit{Busemann geodesics}, or simply {\it $\theta \sig$ geodesics}. 
\end{definition}

\begin{proposition} \cite[Theorem 1.7]{Busani-2023} \label{prop:Busani_N3G}
On a single event of probability $1$, for all $\theta \in \R$, every $\theta$-directed semi-infinite geodesic is either a $\theta-$ or $\theta +$ Busemann geodesic. 
\end{proposition}

\begin{proposition} \cite[Theorem~6.3, Theorem~7.1]{Busa-Sepp-Sore-22a} \label{prop:geod_ordering}
On a single event of probability $1$, for all $\theta \in \R$ and compact sets $K \subseteq \R^2$, there exists a time $T$ such that, for any two $\theta-$ geodesics $g_1^-,g_2^-$ and any two $\theta +$ geodesics, $g_1^+,g_2^+$ rooted at points in $K$, 
\[
g_1^-(s) = g_2^-(s) \le g_1^+(s) = g_2^+(s),\quad\text{for all }s \le T.
\]
If $\theta$ is point of continuity of the Busemann process, the inequality is an equality. If not, then there exists a $T\in\R$ such that strict inequality holds.  
\end{proposition}

\section{The state space $\CFP$} \label{appx:state_space}
In this section, we prove topological details about the space $\CFP$ defined in Section~\ref{sec:state_space}.

The next lemma is an alternative description of the convergence in the space $\CFP$.
\begin{lemma} \label{lem:alt_conv}
    For a sequence $(f_n)_n$ taking values in  $\CFP$, we have $f_n\to f \in \CFP$ if and only if the following hold:
    \begin{enumerate} [label=\rm(\roman{*}), ref=\rm(\roman{*})]  
    \item \label{it:fntof_unif} $f_n$ converges to $f$ uniformly on compact sets.
    \item \label{it:sup_bd} For each $a > 0$, 
\[
\sup_{n \ge 1, x \in \R}[f_n(x) \vee f(x)  - ax^2]  < \infty.
\]
\end{enumerate}
\end{lemma}
\begin{proof}
    Recall that, by definition, $f_n \to f$ in $\CFP$ if the following two conditions hold:
    \begin{enumerate} [label=\rm(\arabic{*}), ref=\rm(\arabic{*})] 
        \item \label{it:fntof1} $f_n$ converges to $f$ uniformly on compact sets.
    \item \label{it:sup_conv} For each $a > 0$, 
\[
\lim_{n \to \infty}\sup_{x \in \R}[f_n(x)  - ax^2] = \sup_{x \in \R}[f(x)  - ax^2].
\]
\end{enumerate}
    Condition~\ref{it:sup_conv} clearly implies Condition~\ref{it:sup_bd}, so it suffices to prove that Conditions~\ref{it:fntof_unif}--\ref{it:sup_bd} together imply Condition~\ref{it:sup_conv}.

    Assume that Conditions~\ref{it:fntof_unif}--\ref{it:sup_bd} hold, and let $a > 0$.  Since $f_n(0) \to f(0)$, we have that
    \be \label{eq:Mdef}
    M := \inf_n f_n(0) > -\infty.
    \ee
    By Condition~\ref{it:sup_bd} applied with $\f{a}{2}$ in place of $a$, there exists $A > 0$ so that 
    \[
    f_n(x) \vee f(x) \le A + \f{a}{2}x^2,\quad\text{for all } x \in \R, n \ge 1.
    \]
    Hence, there exists $C > 0$ so that, for all $n \ge 1$,
    \[
    \sup_{x \notin [-C,C]} [f_n(x) \vee f(x) - ax^2] < M.
    \]
    By definition \eqref{eq:Mdef} of $M$, this implies that
    \be \label{eq:f-ax2_restrict}
    \sup_{x \in \R}[f_n(x) - ax^2] = \sup_{x \in [-C,C]}[f_n(x) - ax^2],\quad\text{and}\quad \sup_{x \in \R}[f(x) - ax^2] = \sup_{x \in [-C,C]}[f(x) - ax^2].
    \ee
The conclusion for Condition~\ref{it:sup_conv} now follows from the uniform convergence of $f_n$ to $f$. 
\end{proof}

\begin{lemma} \label{lem:restrict_max}
    Assume that $f_n \to f$ in $\CFP$. Then, with probability $1$, for each $s \in \R$ and each compact set $K \subseteq \R \times \R_{>s}$, there exists $C > 0$ such that, for all $(x,t) \in K$ and $n \ge 1$,
    \[
    h(x,t \mid s, f_n) = \sup_{z \in [-C,C]}[f_n(z) + \Ll(z,s;x,t)],\quad\text{and}\quad h(x,t \mid s, f) = \sup_{z \in [-C,C]}[f(z) + \Ll(z,s;x,t)].
    \]
\end{lemma}
\begin{proof}
    By the fact that $f_n(0) \to f(0)$ and the continuity of the directed landscape (Lemma~\ref{lem:Landscape_global_bound}),  we have that 
    \be \label{eq:largeM_Def}
    M := \inf_{(x,t) \in K, n \ge 1} [f_n(0) + \Ll(0,s;x,t)] \wedge \inf_{(x,t) \in K}[f(0) + \Ll(0,s;x,t)]  > -\infty.
    \ee

    By compactness of $K$ we may choose $a \in \bigl(0, \inf\{\f{1}{t-s}: (x,t) \in K\}\bigr)$. By the convergence $f_n \to f$ in $\CFP$ and Lemma~\ref{lem:alt_conv}\ref{it:sup_bd}, there exists $A > 0$ so that 
    \[
    f_n(z) \vee f(z) \le A + az^2,\quad\forall z \in \R, n \ge 1.
    \]
    From Lemma~\ref{lem:Landscape_global_bound}, we have that $\Ll(z,s;x,t) \sim - \f{(z-x)^2}{t - s}$, with an explicit error bound. Since $a < \f{1}{t-s}$ for all $(x,t) \in K$, the error bounds in Lemma~\ref{lem:Landscape_global_bound} imply there exists $C > 0$ such that, for all $n \ge 1$,
    \be \label{eq:lessM}
    \sup_{z \notin [-C,C], (x,t) \in K}[f_n(z) \vee f(z) + \Ll(z,s;x,t)]  < M.
    \ee
  Then, from the definition of $M$ in \eqref{eq:largeM_Def}, for all $n \ge 1$, and $(x,t) \in K$,
 \[
 h(x,t \mid s, f_n) = \sup_{z \in \R}[f_n(z) + \Ll(z,s;x,t)] \ge M \overset{\eqref{eq:lessM}}{\Longrightarrow} h(x,t \mid s, f_n) = \sup_{z \in [-C,C]}[f_n(z) + \Ll(z,s;x,t)],
 \]
 and the same holds if we replace $f_n$ with $f$. 
\end{proof}

\begin{lemma} \label{lem:h-ax2_bd}
    Assume that $f_n \to f$ in $\CFP$. Then, with probability $1$, for each $t > s$ and $a > 0$,  
    \[
    \sup_{n \ge 1,x \in \R}[h(x,t \mid s, f_n) \vee h(x,t \mid s, f) - ax^2] < \infty.
    \]
\end{lemma}
\begin{proof}
    By Lemma~\ref{lem:alt_conv}\ref{it:sup_bd} and the assumption that $f_n \to f$ in $\CFP$, for each $b > 0$, there exists $C_b > 0$ such that
    \be \label{Cb_bd}
    f_n(x) \vee f(x) \le C_b + b x^2,\quad\text{for all }x \in \R, n \ge 1.
    \ee
    Then, by the growth bounds for $\Ll$ in Lemma~\ref{lem:Landscape_global_bound},  for $s < t$ and $\ve > 0$,  there exists $C_\ve' = C_\ve'(s,t)$ such that
    \[
    \Ll(z,s;x,t)\le  C_\ve' -\f{(x-z)^2}{t-s} + \ve|x| + \ve|z|,\quad\text{for all }x,z \in \R.
    \]
    Then, combined with \eqref{Cb_bd}, we have, if $b < \f{1}{t-s}$, 
    \begin{align*}
        h(x,t \mid s,f_n) \vee h(x,t \mid s,f) &\le C_b + C_\ve' + \sup_{z \in \R}\Bigl[bz^2 -\f{(x-z)^2}{t-s} + \ve|x| + \ve|z|\Bigr] \\
        &\le C_b + C_\ve' + \ve|x| +  \f{1}{4(1 - b(t-s))} (4bx^2 + 4\ve|x| + (t-s)\ve^2).
    \end{align*}
    For $a > 0$, we then obtain the conclusion by choosing $b < \f{1}{(t-s)}$ sufficiently small so that 
    \[
    \f{b}{1-b(t-s)} < a. \qedhere
    \]
\end{proof}

The following gives continuity of the KPZ fixed point from arbitrary intial data. It is known  from \cite[Theorem~4.13]{KPZfixed}, but we provide a proof since we work with a larger state space. 
\begin{proposition} \label{prop:h_pres_CFP}
    Assume $f \in \CFP$. Then, for each $t > s$, $h(\cdot,t \mid s, f) \in \CFP$ with probability $1$.  Moreover, the function
    \[
    (x,t) \mapsto h(x,t \mid s, f)
    \]
    is a continuous function $\R \times \R_{>s} \to \R$. 
\end{proposition}
\begin{proof}
     It suffices to show the stated continuity, and that, for each $a > 0$ and $t > s$,
    \be \label{eq:h-ax2}
    \sup_{x \in \R}[h(x,t \mid s, f) - ax^2] < \infty.
    \ee
    The finiteness in \eqref{lem:h-ax2_bd} follows from Lemma~\ref{lem:h-ax2_bd} by taking $f_n = f$, so we turn to the continuity.
    Let $(x,t) \in \R \times \R_{>s}$, and let $x_n \to x$ and $t_n \to t$. By Lemma~\ref{lem:restrict_max}, there exists $C > 0$ so that, for all $n \ge 1$,
    \be \label{eq:h_bd_max}
    h(x_n,t_n \mid s, f) = \sup_{z \in [-C,C]}[f(z) + \Ll(z,s;x_n,t_n)],\quad\text{and}\quad h(x,t \mid s, f) = \sup_{z \in [-C,C]}[f(z) + \Ll(z,s;x,t)].
    \ee
   By continuity of $\Ll$ and $f$, the convergence of the first term in \eqref{eq:h_bd_max} to the last term in \eqref{eq:h_bd_max} follows.
\end{proof}

The next lemma shows continuity of the solution map of the KPZ fixed point in the space $\CFP$. In particular, it implies the Feller property for the Markov semigroup associated to the KPZ fixed point evolution. It was previously shown in \cite{KPZfixed} that the Feller property holds for the KPZ fixed point on the space of upper semi-continuous functions satisfying that are bounded from above by the absolute value of an affine function.  
\begin{proposition} \label{prop:Feller}
    Assume that $f_n \to f$ in $\CFP$. Then, for any $t > s$, $h(\cdot,t \mid s, f_n) \to h(\cdot,t \mid s, f)$ in $\CFP$. 
\end{proposition}
\begin{proof}
    We use the characterization of convergence given in Lemma~\ref{lem:alt_conv}. Condition~\ref{it:fntof_unif} follows from Lemma~\ref{lem:restrict_max} the uniform convergence of $f_n \to f$. Condition~\ref{it:sup_bd} follows from Lemma~\ref{lem:h-ax2_bd}. 
\end{proof}

We now prove the following.
\begin{lemma} \label{lem:Polish}
    The space $\CFP$ is Polish.
\end{lemma}
\begin{proof}
We consider the metric $D_{\CFP}$ on $\CFP$ given by 
\be \label{eq:DCFP_metric}
\begin{aligned}
   D_{\CFP}(f,g) &= \sum_{k = 1}^\infty 2^{-k} \Biggl(1 \wedge \sup_{x \in [-k.k]}|f(x) - g(x)|\Biggr) \\
   &\quad+ \sum_{m = 1}^\infty 2^{-m} \Biggl(1 \wedge \Bigl(\Bigl|\sup_{x \in \R}\Bigl[f(x) - \f{1}{m}x^2\Bigr]- \sup_{x \in \R}\Bigl[g(x) - \f{1}{m}x^2\Bigr]\Bigr|\Bigr)\Biggr).
\end{aligned}
\ee
We first show that convergence in this metric is equivalent to convergence in $\CFP$. It is immediate that convergence in $\CFP$ implies convergence with respect to $D_{\CFP}$. For the other direction,  the only thing that needs to be verified for this is that, if $f_n \to f$ uniformly on compact sets and
\[
\lim_{n \to \infty} \sup_{x \in \R} \Bigl[f_n(x) - \f{1}{m}x^2\Bigr] = \sup_{x \in \R} \Bigl[f(x) - \f{1}{m}x^2\Bigr],\quad\text{for all }m \ge 1,
\]
then the same convergence holds when $\f{1}{m}$ is replace with any $a > 0$. This follows by monotonicity and the alternate characterization of convergence in Lemma~\ref{lem:alt_conv}.

Now, let $(f_n)_n$ be a Cauchy sequence with respect to the metric $D_{\CFP}$. We show that the sequence has a limit in $\CFP$. Since $f_n$ is Cauchy with respect to $D_{\CFP}$, it is Cauchy with respect to the metric of uniform convergence on compact sets, and thus has a limit $f$ such that $f_n \to f$ uniformly. 

 To show that $f \in \CFP$ and $f_n \to f$ with respect to the topology on $\CFP$, by~\ref{lem:alt_conv}, it suffices to show that, for each $a > 0$,
 \be \label{fnf_bd1}
 \sup_{x \in \R, n \ge 1}[f_n(x) \vee f(x) - ax^2] < \infty.
 \ee
Suppose, by way of contradiction, that this fails for some $a > 0$. Pick $m \in \N$ such that $m > a^{-1}$. Since $(f_n)_n$ is a Cauchy sequence with respect to $D_{\CFP}$, there exists $N \in \N$ so that, for all $n \ge N$,
\be \label{eq:leT}
\sup_{x \in \R}[f_n(x) - ax^2] \le \sup_{x \in \R}\Bigl[f_n(x) - \f{1}{m}x^2\Bigr] < \sup_{x \in \R}\Bigl[f_N(x) - \f{1}{m}x^2\Bigr] + \ve =: T < \infty.
\ee
Then, by assumption that \eqref{fnf_bd1} fails, there must exist $w \in \R$ so that $f(w) - aw^2 > T+1$. Then, by the uniform-on-compact convergence of $f_n \to f$, we have $f_n(w) \to f(w)$, and so for all sufficiently large $n$,
\[
\sup_{x \in \R}[f_n(x) - ax^2] \ge f_n(w) - aw^2 \ge f(w) - aw^2 - 1 \ge T,
\]
a contradiction to \eqref{eq:leT}. Thus, $\CFP$ is complete under the metric $D_{\CFP}$. 

We now show that $\CFP$ is separable by showing that the following countable set is dense in $\CFP$:
\begin{align*}
\mathcal D &:= \bigcup_{n \in \N} \mathcal D_n,\quad\text{where} \\
\mathcal D_n &:=
\{f \in \CFP: f \text{ is constant on } (-\infty,-2^n] \text{ and } [2^n,\infty),\\
&\qquad\text{ and for all }j\in \Z, f(j2^{-n}) \in \Q \text{ and }f \text{ is linear on }[j2^{-n},(j+1)2^{-n}] \}.
\end{align*}
A function $f \in \D_n$ is determined by its values $f(j2^{-n})$ for $j \in \{-2^{2n},-2^{2n + 1}\ldots,2^{2n}\}$, so we see that each  $\mathcal D_n$ is in bijection with $\Q^{2^{2n+1} + 1}$ and is therefore countable. 

For $f \in \CFP$, define a sequence $(f_n)_n$ such that, $f_n \in \D_n$ for all $n$, and for all $j \in \{-2^{-2n},-2^{-2n} + 1,\ldots,2^{2n}\}$, we have $|f_n(j2^{-n}) - f(j2^{-n})| \le 2^{-n}$. It is readily seen that $f_n \to f$ uniformly on compact sets. Since we assumed that $f \in \CFP$, to show that $f_n \in \D$ for all $n$ and $f_n \to f$ in $\CFP$, by Lemma~\ref{lem:alt_conv}, it suffices to show that, for each $a > 0$, 
\be \label{fn-ax2}
\sup_{n}\sup_{x \in \R}[f_n(x) - ax^2] < \infty.
\ee
Since $f_n$ is constant outside $[-2^n,2^n]$, we see that
\[
\sup_{n}\sup_{x \in \R}[f_n(x) - ax^2] = \sup_n \sup_{x \in [-2^{n},2^n]}[f_n(x) - ax^2].
\]
Suppose, by way of contradiction, that \eqref{fn-ax2} fails for some $a > 0$. Then, there exist sequences $n_k \to \infty$ and $x_{n_k} \in [-2^{n_k},2^{n_k}]$ such that, for all $k \ge 1$, $
f_{n_k}(x_{n_k}) - a x_{n_k}^2 \ge k$.
For $k \ge 1$, let $j_k \in \Z$ be the unique integer so that $x_{n_k} \in [j_k 2^{-n_k},(j_k + 1)2^{-n_k})$.
Then, for all $k \ge 1$.
\begin{align*}
   k &\le  f_{n_k}(x_{n_k}) - a x_{n_k}^2 \\&\le \max\Bigl(f_{n_k}(j_k 2^{-n_k}),f_{n_k}((j_k+1) 2^{-n_k}) \Bigr) - a x_{n_k}^2  \\
   &\le \max\Bigl(f(j_k 2^{-n_k}) ,f((j_k+1) 2^{-n_k}) \Bigr) +  2^{-n_k} - a 2^{-2n_k}\min\Bigl(j_k^2,(j_k+1)^2\Bigr)  \\
   &\le \max\Bigl(f(j_k 2^{-n_k}) - \f{a}{2}2^{-2n_k}j_k^2 ,f((j_k+1) 2^{-n_k})- \f{a}{2}2^{-2n_k}(j_k+1)^2 \Bigr) +  (a+1)2^{-n_k} \\
   &\le \sup_{x \in \R}\Bigl[f(x) - \f{a}{2}x^2\Bigr] + (a+1)2^{-n_k},
\end{align*}
where in the penultimate step, we used the bound $\min(x^2,(x+1)^2) \ge  \f{\max((x+1)^2,x^2)}{2} -1$. Since $f \in \CFP$, taking $k$ sufficiently large gives a contradiction.
\end{proof}

\subsection{Proof of Proposition~\ref{prop:uniform_upbd}} \label{sec:converge_to_SH}
Before proving Proposition~\ref{prop:uniform_upbd}, we give two intermediate lemmas. The first below is a corollary of \cite[Proposition~10.5]{Directed_Landscape}.  
\begin{lemma} \label{lem:Ldif1}
    Let $\theta \in \R$ and $\ve > 0,\delta > 0$. Then, there exists a random constant $C > 0$ (depending only on $\delta$) such that, for all $x,z \in \R, s < -1$, and $\Bigl|\f{z}{|s|} - \theta\Bigr| < \ve$,
    \[
    \Ll(z,s;x,0) - \Ll(z,s;0,0) \le 2\max\{(\theta - \ve)x,(\theta + \ve)x\} + C|x|^{1/2 + \delta}.
    \]
\end{lemma}
\begin{proof}
    The result \cite[Proposition~10.5]{Directed_Landscape} implies that,
    for each $b \ge 2$, there exists a constant $C_b$ satisfying $\Pp(C_b> m) \le cb^{10}e^{-dm^{3/2}}$ for universal constants $c,d > 0$ such that, for all $|x|,|z|,|s| \le b$, and $s < -1$, 
    \begin{align*}
        \Biggl|\Ll(z,s;x,0)  - \Ll(z,s;0,0) + \f{(x-z)^2}{|s|} + \f{z^2}{|s|}\Bigr| &\le  C_b|x|^{1/2} \log\Bigl(\f{4b}{|x|}\Bigr).
    \end{align*}
    Therefore, for each $m \ge 0$,
\begin{align*}
    &\Pp\Biggl(\sup_{x \in \R, z \in \R, s < -1, x \neq 0}\f{\Bigl|\Ll(z,s;x,0)  - \Ll(z,s;0,0) + \f{(x-z)^2}{|s|} + \f{z^2}{|s|}\Bigr|}{|x|^{1/2 + \delta}} \ge m\Biggr) \\
    &\le \sum_{b= 1}^\infty\Pp\Biggl(\sup_{|x|,|z|,|s| \le b, s < - 1, x \neq 0}\f{\Bigl|\Ll(z,s;x,0)  - \Ll(z,s;0,0) + \f{(x-z)^2}{|s|} + \f{z^2}{|s|}\Bigr|}{|x|^{1/2} \log(4b/|x|)} \ge m b^{\delta} \log(4)\Biggr) \\
    &\le \sum_{b = 1}^\infty \Pp(C_b \ge m b^{\delta} \log(4)) \le \sum_{b = 1}^\infty cb^{10}e^{-d \log(4)^{3/2} b^{\delta} m^{3/2}},
\end{align*}
and this last term converges to $0$ as $m \to \infty$ to $0$. Hence, there exists a random constant $C = C(\delta) > 0$ such that, for all $x,z \in \R$,  $s < -1$, and $\Bigl|\f{z}{|s|} - \theta\Bigr| < \ve$,
\begin{align*}
 \Ll(z,s;x,0)  - \Ll(z,s;0,0)  &\le - \f{(x-z)^2}{|s|} - \f{z^2}{|s|} + C_\delta|x|^{1/2 + \delta} \\
&= -\f{x^2}{|s|} + \f{2zx}{|s|} + C_\delta |x|^{1/2 + \delta} \\
&\le 2\max\{(\theta - \ve)x,(\theta + \ve)x\} + C|x|^{1/2 + \delta}. \qedhere
\end{align*}
\end{proof}

The next result is a weaker version of Lemma~\ref{lem:Landscape_global_bound}, which will be easier to work with for our purposes.
\begin{lemma} \label{lem:crude_DL_global}
    For each $\ve > 0$, there exists a random constant $C'> 0$ such that, for all $s < 0$ and $z,s \in \R$, 
    \[
\Bigl|\Ll(z,s;x,0) + \f{(z-x)^2}{|s|}\Bigr| \le C'(|s|^{1/2} + |s|^{1/3}) + \ve|z| + \ve |x|.
\]
\end{lemma}
\begin{proof}
    From Lemma~\ref{lem:Landscape_global_bound}, we need to show that the following quantity is bounded over $z,x \in \R$, and $s< 0$. 
    \begin{align*}
        \f{C \log^{4/3}\Bigl(\f{2\sqrt{x^2 + z^2 + s^2}  +4}{|s|}\Bigr)\log^{2/3} \Bigl(\sqrt{x^2 + z^2 + s^2} + 2\Bigr)}{|s|^{1/6} + 1 + \f{\ve |x|}{|s|^{1/3}} +  \f{\ve|z|}{|s|^{1/3}}}.
    \end{align*}
    Clearly, this quantity is bounded for $|s| < 1$. When $|s| > 1$, we rewrite it as
    \begin{align*}
&\f{C \log^{4/3}\Bigl(2\sqrt{\f{x^2}{|s|^2} + \f{z^2}{|s|^2} + 1}  +\f{4}{|s|}\Bigr)\log^{2/3} \Bigl(|s|\sqrt{\f{x^2}{|s|^2} + \f{z^2}{|s|^2} + 1} + 2\Bigr)}{1 + |s|^{1/6}\Bigl( \f{\ve |x|}{|s|^{1/2}} +  \f{\ve|z|}{|s|^{1/2}}\Bigr)} \\
&\le \f{C \log^{4/3}\Bigl(2\sqrt{(\f{|x|}{|s|^{1/2}})^2 + (\f{|z|}{|s|^{1/2}})^2 + 1}  +4\Bigr)\log^{2/3} \Bigl(|s|\sqrt{(\f{|x|}{|s|^{1/2}})^2 + (\f{|z|}{|s|^{1/2}})^2 + 1} + 2\Bigr)}{1 + |s|^{1/6}\Bigl( \f{\ve |x|}{|s|^{1/2}} +  \f{\ve|z|}{|s|^{1/2}}\Bigr)},
    \end{align*}
    and the boundedness readily follows from here. 
\end{proof}

\begin{proof}[Proof of Proposition~\ref{prop:uniform_upbd}]

We prove \eqref{eq:linear_upbd}. The convergence to $W^\theta(0,t;x,t)$ in the topology on $\CFP$ then follows by the characterization of convergence in Lemma~\ref{lem:alt_conv}. Specifically, Condition~\ref{it:fntof_unif} is met by the attractiveness in Proposition~\ref{prop:invariance_of_SH}, and Condition~\ref{it:sup_bd} is met by \eqref{eq:linear_upbd}.

By the time stationarity in Lemma~\ref{lm:landscape_symm}, it suffices to take $t = 0$. We will also take $\theta > 0$; the other cases are handled similarly. By the assumption on the asymptotic slope \eqref{eqn:drift_assumptions}, we may  choose $\ve > 0$ and $\gamma <  0$ satisfying the conditions
\be \label{ve_bounds}
0 < \theta  - 2\sqrt{\ve \theta} < \theta -\ve < \theta + \ve < \theta  + 2\sqrt{\ve \theta},
\ee
\be \label{vetheta_bds}
\theta^2 - 3\ve \theta + 2\ve \sqrt{\ve \theta} < \Bigl(\theta - \f{\ve}{2}\Bigr)^2,
\ee
 and
\be \label{eq:gammabd}
-2\theta + 2\ve < 2\gamma < \liminf_{x \to -\infty} \f{f(x)}{x}.
\ee
The assumption
\[
\lim_{x \to +\infty} \f{f(x)}{x} = 2\theta,
\]
along with \eqref{eq:gammabd} implies that there exist random constants $C = C(\ve,\gamma) > 0$ and $R = R(\ve) > 0$ such that 
\be \label{eq:f_up_bd}
\begin{aligned}
f(x) &\le C + \Bigl(2\theta + \f{\ve}{2}\Bigr)x,\quad\text{for }x \ge 0, \\
f(x) &\le C + \Bigl(2\gamma - \f{\ve}{2}\Bigr)x,\quad\text{for }x \le 0, \\
f(x) &\ge \Bigl(2\theta - \f{\ve}{2}\Bigr)x,\quad\text{for }x \ge R.
\end{aligned}
\ee
Note that we only get the lower bound for sufficiently large $x$ because of the weaker assumption that $f$ is upper-semi-continous, instead of continuous. We now consider two cases

\medskip \noindent \textbf{Case 1: $|x| > \f{\ve |s|}{2}$.}
Combining \eqref{eq:f_up_bd} with Lemma~\ref{lem:crude_DL_global} (applied with $\ve$ replaced by $\ve/2$), we get the following bound:
\be \label{eq:hx_bd1}
\begin{aligned}
&\quad \, h(x,0 \mid s,f)  \\
&\le C +\f{\ve}{2}|x| + C'(|s|^{1/2} + |s|^{1/3}) +  \sup_{z \ge 0}\Bigl[ (2\theta +\ve)z - \f{(x-z)^2}{|s|}   \Bigr] \vee \sup_{z \le 0}\Bigl[(2\gamma - \ve)z - \f{(x-z)^2}{|s|}   \Bigr]  \\
&\le C +\f{\ve}{2}|x| + C'(|s|^{1/2} + |s|^{1/3}) + \max\Bigl((2\theta + \ve)x + \Bigl(\theta +\f{\ve}{2}\Bigr)^2|s|, (2\gamma - \ve) x + \Bigl(\gamma - \f{\ve}{2}\Bigr)^2|s|  \Bigr).
\end{aligned}
\ee
On the other hand, since $\theta > 0$,  Lemma~\ref{lem:unq} implies that, for all sufficiently large $|s|$ such that $(\theta - \ve)|s| \ge R$,
\be \label{eq:h0_lb}
\begin{aligned}
h(0,0 \mid s,f) &= \sup_{z \ge R} [f(z) + \Ll(z,s;0,0)]  \\
&\ge C'(|s|^{1/2} + |s|^{1/3}) + \ve|x| +  \sup_{z \ge R}\Bigl[(2\theta - \ve)z - \f{z^2}{|s|}\Bigr] \\
&= C'(|s|^{1/2} + |s|^{1/3}) + \ve|x| + (2\theta - \ve)x + \Bigl(\theta - \f{\ve}{2}\Bigr)^2|s|.
\end{aligned}
\ee
Comparing~\eqref{eq:hx_bd1} and~\eqref{eq:h0_lb}, we see that there exist random constants $A',B',D' > 0$ so that, for all $s < 0$ and $x \in \R$,
\be \label{eq:hs_crude_bd}
h(x,0 \mid s,f) - h(0,0 \mid s,f) \le A' + B'|x| + D'|s|.
\ee
Then, by the assumption $|x| > \f{\ve|s|}{2}$, we have 
\[
h(x,0 \mid s,f) - h(0,0 \mid s,f) \le A' + (B' + 2D'\ve^{-1})|x|,
\]
giving us a sufficient bound in this case.

\medskip \noindent \textbf{Case 2: $|x| \le \f{\ve |s|}{2}$.} Define the interval 
\[
I_s = \Bigl((\theta - 2\sqrt{\ve \theta})|s|,(\theta + 2\sqrt{\ve \theta})|s| \Bigr).
\]
Now, by \eqref{ve_bounds}, we may choose $|s|$ sufficiently large so that $R \ge (\theta - 2\sqrt{\ve \theta})|s|$ \eqref{eq:f_up_bd}. Then, by \eqref{eq:f_up_bd} and Lemma~\ref{lem:crude_DL_global} (applied with $\ve$ replaced by $\ve/2$), we have the bound
\be \label{eq:zinI}
\begin{aligned}
&\quad\, \sup_{z \in I_s} [f(z) + \Ll(z,0;x,0)] \\
&\ge -C-\f{\ve}{2}|x| - C(|s|^{1/3} + |s|^{1/2})   + \sup_{z \in I_s}\Bigl[(2\theta - \ve)z - \f{(z-x)^2}{|s|}\Bigr] \\
&= -C -\f{\ve}{2}|x| - C(|s|^{1/3} + |s|^{1/2}) + (2\theta - \ve)x + \Bigl(\theta - \f{\ve}{2}\Bigr)^2|s|,
\end{aligned}
\ee
where the last equality holds by a direct computation. Specifically, the function inside the supremum over $z \in \R$ is maximized at $z = x + (\theta - \f{\ve}{2})|s|$, which lies in the set $I_s$ by the assumption $|x| \le \f{\ve|s|}{2}$ and \eqref{ve_bounds}.

By the same computation as in \eqref{eq:hx_bd1}, we obtain
\be \label{eq:zle0bd}
\begin{aligned}
    \sup_{z \le 0}[f(z) + \Ll(z,0;x,0)] \le C + \f{\ve}{2}|x| + C'(|s|^{1/2} + |s|^{1/3}) + (2\gamma -\ve)x + \Bigl(\gamma - \f{\ve}{2}\Bigr)^2|s|.
\end{aligned}
\ee
Lastly, we follow a similar procedure to obtain the bound
\begin{align}
    &\sup_{z \ge 0, z \notin I_s}[f(z) + \Ll(z,s;x,0)] \label{eq:zpos_notinI}\\
    &\le  C + \f{\ve}{2}|x| + 
    C'(|s|^{1/2} + |s|^{1/3}) +\sup_{z \notin I_s}\Bigl[(2\theta + \ve)z - \f{(z-x)^2}{|s|}   \Bigr] \nonumber \\
    &= C+ \f{\ve}{2}|x| + 
    C'(|s|^{1/2} + |s|^{1/3}) \nonumber   \\
    &+ \max \Bigl((2\theta + \ve)(\theta - 2\sqrt{\ve \theta})|s| -\f{((\theta - 2\sqrt{\ve \theta})|s| - x)^2}{|s|} ,(2\theta + \ve)(\theta + 2\sqrt{\ve \theta})|s| -\f{((\theta + 2\sqrt{\ve \theta})|s| - x)^2}{|s|}\Bigr)\nonumber  \\
    &= C + \f{\ve}{2}|x| + C'(|s|^{1/2} + |s|^{1/3}) - \f{x^2}{|s|} \nonumber   \\
    &+ \max\Bigl((\theta^2 - 3\ve \theta - 2\ve \sqrt{\ve \theta})|s| + 2(\theta - 2\sqrt{\ve\theta})x ,(\theta^2 - 3\ve \theta + 2\ve \sqrt{\ve \theta})|s| +  2(\theta + 2\sqrt{\ve\theta})x  \Bigr). \nonumber 
\end{align}
The first equality below holds because the maximizer of the quadratic function inside the supremum on the second line over $z \in \R$ is $x + (\theta + \f{\ve}{2})$. Our assumption $|x| \le \f{\ve|s|}{2}$ along with \eqref{ve_bounds} implies this maximizer does not lie in the set $I_s$; hence the maximum occurs at one of the endpoints.  

Now, by the condition \eqref{vetheta_bds}, we have \[
\theta^2 - 3\ve \theta - 2\ve\sqrt{\ve \theta} <\theta^2 - 3\ve \theta + 2\ve\sqrt{\ve \theta} < \Bigl(\theta - \f{\ve}{2}\Bigr)^2.
\]
Furthermore, \eqref{eq:gammabd} gives $\gamma > -\theta + \ve$, so combined with the fact that $\gamma < 0$, we have 
\[
\Bigl(\gamma - \f{\ve}{2}\Bigr)^2 < \Bigl(-\theta + \f{\ve}{2}\Bigr)^2 = \Bigl(\theta - \f{\ve}{2}\Bigr)^2.
\]
Then, comparing \eqref{eq:zinI},\eqref{eq:zle0bd}, and \eqref{eq:zpos_notinI}, followed by an application of Lemma~\ref{lem:Ldif1}, there exist random constants $A,B > 0$ (changing from the first inequality to the second line below) such that, for all sufficiently large $|s|$ and all $|x| \le \f{\ve|s|}{2}$,
\begin{align*}
h(x,0 \mid s,f) &\le A + B|x| +  \sup_{z \in I_s}[f(z) + \Ll(z,s;x,0)] \\
&\le A + B|x| + \sup_{z \in I_s}[f(z) + \Ll(z,s;0,0)] \\
&\le A + B|x| + h(0,0 \mid s,f),
\end{align*}
completing the proof. As stated above, the other cases $\theta < 0$ and $\theta = 0$ are handled similarly. The $\theta < 0$ case is symmetric, while for $\theta = 0$, the assumptions \eqref{eqn:drift_assumptions} imply
\[
f(x) \le C + \f{\ve}{2}|x|,\quad\text{ for all }x \in \R.
\]
Along with the assumption that $f(x) > -\infty$ for some $x \in \R$, we can confine all maximizers to the interval 
\[
\Bigl(- 2\sqrt{\ve \theta}|s|, 2\sqrt{\ve \theta}|s| \Bigr)
\]
for all sufficiently large $|s|$.
\end{proof}

\section{Bessel processes and Brownian meander} \label{sec:Bessel}

We will need the following two key results from \cite{Rogers-Pitman-81}
on the relationship between $\BM(\theta)$ and $\BES^{3}(\theta)$.
The first is a version with drift of Pitman's famous ``$2M-X$''
theorem \cite{Pitman1975}, but for Brownian motion with drift.
\begin{proposition}\cite[Theorem~1, Corollary 2(ii)]{Rogers-Pitman-81}
\label{prop:2M-X}Let $\theta>0$, $B\sim\mathrm{BM}(\theta)$ and
$Z\sim\mathrm{BES}^{3}(\theta)$.
\begin{enumerate} [label=\rm(\roman{*}), ref=\rm(\roman{*})]  \itemsep=3pt
\item \label{enu:2M-X-fwd}Let $M(x)=\max\limits_{y\in[0,x]}B(x)$ and $Y(x)=2M(x)-B(x)$.
Then $(Y(x))_{x\ge0}\sim(Z(x))_{x\ge0}$.
\item \label{enu:2M-x-backward}Let $F(x)=\inf\limits_{y\in[x,\infty)}Z(x)$
and $W(x)=2F(x)-Z(x)$. Then $(W(x))_{x\ge0}\sim(B(x))_{x\ge0}$.
\end{enumerate}
\end{proposition}

The second says that a Brownian motion with drift, considered until
its first-passage time at a given level, is identical in law to the
horizontal and vertical reflection of a Bessel-$3$ process with the
same drift, considered until its last-passage time at the same level.
\begin{proposition}\cite[Corollary 1(i,iv)]{Rogers-Pitman-81}
\label{prop:besselrev}Let $B\sim\mathrm{BM}(\theta)$, $Z\sim\mathrm{BES}^{3}(\theta)$,
and $\zeta>0$. Let 
\begin{equation}
\tau\coloneqq\inf\{x>0\st B(x)=\zeta\}\qquad\text{and}\qquad\sigma\coloneqq\sup\{x>0\st Z(x)=\zeta\}.\label{eq:tau-sigma-def}
\end{equation}
Then we have
\[
(\sigma,(Z_{t})_{0\le t\le\sigma})\overset{\mathrm{law}}{=}(\tau,(\zeta-B_{\tau-s})_{0\le s\le\tau}).
\]
\end{proposition}

 We say that a random variable $J$
has the arcsine distribution on $[0,L]$, and write $J\sim\Arcsine(0,L)$,
if
\[
\mathbb{P}(J\le t)=\frac{2}{\pi}\arcsin(\sqrt{t/L}).
\]
Furthermore, let $\Bm(\theta;L)$ denote the law of a Brownian meander on $[0,L]$
with drift $\theta$. Informally, this is a Brownian motion $B$ on
$[0,L]$ with drift $\theta$ such that $B(0)=0$, conditioned to be
nonnegative on $[0,L]$. It can be obtained from the $\Bm(0;L)$, the law of the standard
Brownian meander on $[0,L]$, by a Girsanov change of measure:
\begin{equation}
\frac{\dif\Bm(\theta;L)}{\dif\Bm(0;L)}(B)\propto\e^{\theta B(L)}.\label{eq:meander-girsanov}
\end{equation}
We note that the law $\Bm(0;L)$ of standard Brownian mander on $[0,L]$
can be obtained from the law of the standard Brownian meander on $[0,1]$ (as constructed in \cite{Durrett-Iglehart-Miller-1977})
by Brownian scaling.

We will use the following facts about the minimum value of a
Brownian motion on an interval, as well as the Brownian motion seen
from its minimum on the interval.
\begin{proposition}
\label{prop:BM_min}Let $B\sim\BM$ and fix $L>0$. Let $A=\argmin\limits_{x\in[0,L]}B(x)$.
Then
\begin{enumerate} [label=\rm(\roman{*}), ref=\rm(\roman{*})]  \itemsep=3pt
\item \label{enu:arcsine}$A\sim\Arcsine(0,L)$.
\item \label{enu:denisov-conditional} Conditional on $A$, the processes
\[
B_{-}(x)\coloneqq W(A-y)-W(A)\qquad\text{and}\qquad B_{+}(x)\coloneqq W(A+y)-W(A)
\]
are independent Brownian meanders of lengths $A$ and $L-A$, respectively.
\end{enumerate}
\end{proposition}

Part~\ref{enu:arcsine} is one of the celebrated ``arcsine laws''
for Brownian motion; see e.g.~\cite[Theorem~5.26(b)]{morters_peres_2010}.
Part~\ref{enu:denisov-conditional} is known as \emph{Denisov's path
decomposition} and was proved in \cite{Denisov-1983-russian}.
\begin{proposition}
\label{prop:meander-to-bessel}Let $\theta>0$. As $L\to\infty$,
$\Bm(\theta;L)$ converges to the restriction of $\BES^{3}(\theta)$ to $x \in [0,\infty)$, weakly with respect
to the topology of uniform convergence on compact sets.
\end{proposition}

\begin{proof}
Let $K<\infty$ and let $F\in\mathcal{C}_b(\mathcal{C}([0,K]))$. As explained in \cite[page 578]{Rogers-Pitman-81}, we may represent $\BES^3(\theta)$ as the $\ve \downarrow 0$ limit of $\BM(\theta)$ conditioned to stay above $-\ve$ on $[0,\infty)$. Thus, we have
\begin{align}
\mathbb{E}_{B\sim\BES^{3}(\theta)}[F(B)] & =\lim_{\ve\downarrow0}\mathbb{E}_{B\sim\BM(\theta)}[F(B)\mid B(x)> - \ve\text{ for all }x\in[0,\infty)]\nonumber \\
 & =\lim_{\ve\downarrow0}\frac{\mathbb{E}_{B\sim\BM(\theta)}[F(B);B(x)> - \ve\text{ for all }x\in[0,\infty)]}{\mathbb{P}_{B\sim\BM(\theta)}[B(x)> - \ve\text{ for all }x\in[0,\infty)]},\label{eq:Bessel3-weak}
\end{align}
and, for $L\ge K$,  by \cite[Theorem (2.1)]{Durrett-Iglehart-Miller-1977} and the change of measure in \eqref{eq:meander-girsanov}, we have
\begin{equation}
\mathbb{E}_{B\sim\Bm(\theta;L)}[F(B)]=\lim_{\ve\downarrow0}\mathbb{E}_{B\sim\BM(\theta)}[F(B)\mid B(x)> - \ve\text{ for all }x\in[0,L]].\label{eq:BmL-weak}
\end{equation}
Now we can write
\be \label{eq:FBe1}
\begin{aligned}
\mathbb{E}_{B\sim\BM(\theta)} & [F(B);B(x)> - \ve\text{ for all }x\in[0,\infty)]\\
 & =\mathbb{E}_{B\sim\BM(\theta)}[F(B);B(x)> - \ve\text{ for all }x\in[0,L];B(x)> - \ve\text{ for all }x\in[L,\infty)]\\
 & =\mathbb{E}_{B\sim\BM(\theta)}[F(B)q_\ve(B(L));B(x)> - \ve \text{ for all }x\in[0,L]].
\end{aligned}
\ee
where $q_\ve(x)$ is the probability that a Brownian motion with drift
$\theta$, started at $x$, remains above $-\ve$ for all time. For $\ve =0$ we use the shorthand $q(x) = q_0(x)$. Combining \eqref{eq:FBe1} with \eqref{eq:Bessel3-weak}, we get
\be \label{eq:BESlim1}
\mathbb{E}_{B\sim\BES^{3}(\theta)}[F(B)]  =\lim_{\ve\downarrow0}\frac{\mathbb{E}_{B\sim\BM(\theta)}[F(B)q_\ve(B(L));B(x)> - \ve\text{ for all }x\in[0,L]]}{\mathbb{E}_{B\sim\BM(\theta)}[B(x)> - \ve\text{ for all }x\in[0,L]]}.
\ee
For each fixed $x$, the function $\ve \mapsto q_\ve(x)$ is nondecreasing in $\ve$. Applying this in \eqref{eq:BESlim1}, for any $\delta > 0$, we have 
\begin{align*}
&\liminf_{\ve \downarrow 0}\frac{\mathbb{E}_{B\sim\BM(\theta)}[F(B)q(B(L));B(x)> - \ve\text{ for all }x\in[0,L]]}{\mathbb{E}_{B\sim\BM(\theta)}[B(x)> - \ve\text{ for all }x\in[0,L]]} \\
&\le   \mathbb{E}_{B\sim\BES^{3}(\theta)}[F(B)]\\
&\le \limsup_{\ve \downarrow 0}\frac{\mathbb{E}_{B\sim\BM(\theta)}[F(B)q_\delta(B(L));B(x)> - \ve\text{ for all }x\in[0,L]]}{\mathbb{E}_{B\sim\BM(\theta)}[B(x)> - \ve\text{ for all }x\in[0,L]]}.
\end{align*}
Then, by two applications of \eqref{eq:BmL-weak}  with $F$
replaced by $B\mapsto F(B)q(B(L))$ and $B \mapsto F(B)q_\delta(B(L))$, we have, for any $\delta > 0$ that 
\[
\mathbb{E}_{B\sim\Bm(\theta;L)}[F(B)q(B(L))] \le \mathbb{E}_{B\sim\BES^{3}(\theta)}[F(B)] \le \mathbb{E}_{B\sim\Bm(\theta;L)}[F(B)q_\delta(B(L))].
\]
Then, since $\lim_{\delta \downarrow 0} q_\delta(x) = q(x)$ for all $x > 0$, the bounded convergence theorem implies that
\be \label{eq:FBFqB}
\mathbb{E}_{B\sim\BES^{3}(\theta)}[F(B)] = \mathbb{E}_{B\sim\Bm(\theta;L)}[F(B)q(B(L))].
\ee
Now, for any $t<\infty$, we
have $\lim\limits_{L\to\infty}\mathbb{P}_{B\sim\Bm(\theta;L)}(B(L)\le t)=0$
(as may be seen from the scaling properties of the Brownian meander),
and this implies that, for any $p<1$, we have 
\[
\lim\limits_{L\to\infty}\mathbb{P}_{B\sim\Bm(\theta;L)}(q(B(L))\le p)=0
\]
 as well. Using this along with \eqref{eq:FBFqB}, we see that
\[
\mathbb{E}_{B\sim\BES^{3}(\theta)}[F(B)]=\lim_{L\to\infty}\mathbb{E}_{B\sim\Bm(\theta;L)}[F(B)q(B(L))]=\lim_{L\to\infty}\mathbb{E}_{B\sim\Bm(\theta;L)}[F(B)],
\]
completing the proof. In fact, the first identity above is true even without taking the limit
in $L$.
\end{proof}

\begin{lemma} \label{lem:Bess_lin_bd}
Let $\theta > 0$, and let $Y\sim \BES^{3}(2\theta)$. Then, as $t \to \infty$,
\begin{equation} \label{eq:YtoNor}
\f{Y(\theta t) - 2\theta^2 t}{t^{1/2}} \Longrightarrow \mathcal N(0,\theta).
\end{equation}
Furthermore, for every $\eta,\delta > 0$, there exists a constant $C' <\infty$ so that 
\begin{equation} \label{eq:Ybd}
\liminf_{t \to \infty}\Pp\Biggl(t^{-1/3}\bigl( Y(\theta t +t^{2/3} y) - Y(\theta t) - 2\theta t^{2/3} y\bigr) \le C' + \delta|y| \text{ for all } y \ge -\theta t^{1/3}\Biggr)  \ge 1 - \eta.
\end{equation}
\end{lemma}
\begin{proof}
 Using Proposition~\ref{prop:2M-X}, we represent $Y$ as 
 \[
 Y(x) = 2\sup_{0 \le y \le x}[B(y) + 2\theta y] - [B(x) + 2\theta x],
 \]
 where $B \sim \BM$. Then, we have 
 \begin{align*}
 Y(\theta t) - 2\theta^2 t = B(\theta t) + 2\sup_{0 \le y \le \theta t}[B(y) - B(\theta t) + 2\theta(y-\theta t)]. 
 \end{align*}
 As $t^{-1/2}B(\theta t) \sim \Nor(0,\theta)$, to prove \eqref{eq:YtoNor}, it suffices to show that 
 \[
 2t^{-1/2} \sup_{0 \le y \le \theta t}[B(y) - B(\theta t) + 2\theta(y-\theta t)] \to 0 \quad\text{in probability}.
 \]
 This follows from Brownian rescaling as follows:
 \begin{align*}
 &\quad \, 2t^{-1/2} \sup_{0 \le y \le \theta t}[B(y) - B(\theta t) + 2\theta(y-\theta t)] \\
 &\deq 2t^{-1/2} \sup_{0 \le y \le \theta t}[B(y) - 2\theta y] = 2t^{-1/2} \sup_{0 \le y \le \theta}[B(yt) - 2\theta yt] \\
 &\deq 2\sup_{0 \le y \le \theta} [B(y) - 2\theta yt^{1/2}],
\end{align*}
and this last quantity converges to $0$ a.s. as $t \to \infty$. 

We turn to the proof of \eqref{eq:Ybd}: for $y \ge -\theta t^{1/3}$, using Brownian scaling and reflection invariance in the distributional equality below, we write
\begin{align*}
&\quad \, Y(\theta t +t^{2/3} y) - Y(\theta t) - 2\theta t^{2/3}y \\
 &= 2 \Bigl(\sup_{0 \le z \le \theta t +
 t^{2/3} y}[B(z) - B(\theta t) + 2\theta (z - \theta t)] - \sup_{0 \le z \le \theta t}[B(z) - B(\theta t) + 2\theta (z-\theta t)]\Bigr) \\
 &\qquad + (B(\theta t) -B(\theta t + t^{2/3} y)) - 4\theta t^{2/3} y \\
&\deq t^{1/3}A(t,y),
\end{align*}
where we define
\begin{align*}
A(t,y) :=2 \Bigl(\sup_{-y \le w \le \theta t^{1/3}}[B(w) - 2\theta wt^{1/3}] - \sup_{0 \le w \le \theta t^{1/3}}[B(w) - 2\theta wt^{1/3}]\Bigr) - B(-y) - 4\theta t^{1/3} y,
\end{align*}
and the distributional equality holds as a process in $y \ge -\theta t^{1/3}$. Given $\eta,\ve > 0$, choose $C > 0$ so that $\Pp(|B(x)| \le C + \ve |x|\;\; \forall x \in \R) \ge 1 - \eta$. We claim that, on the event $|B(x)| \le C + \ve |x|$, we have $|A(t,y)| \le 5C + 3\ve |y|$ for all $t \ge (\ve/(2\theta))^3$. We handle this in two cases.

\medskip \noindent \textbf{Case 1: $y \ge 0$.}
In this case, using the fact that $a \vee b - b \le a$ when $a \wedge b \ge 0$,
\begin{align*}
A(t,y)  &\le 2 \sup_{-y \le w \le 0}[B(w) - 2\theta wt^{1/3}] - B(-y) - 4\theta t^{1/3}y \\
&\le 2\sup_{-y \le w \le 0}[C  - \ve w - 2\theta w t^{1/3}] + C + \ve y - 4\theta t^{1/3}y =  3C + 3\ve y.
\end{align*}
To get the lower bound, note that when $t \ge (\ve/(2\theta))^3$,
\begin{align*}
    A(t,y) &\ge 2\sup_{-y \le w \le 0}[-C + \ve w - 2\theta w t^{1/3}] - 2\sup_{0\le w \le \theta t^{1/3}}[C + \ve w - 2\theta wt^{1/3}] -C - \ve y - 4\theta t^{1/3} y \\
    &= -5C - 3\ve y.
\end{align*}

\medskip \noindent \textbf{Case 2: $y < 0$.}
Then, for $t > (\ve/(2\theta))^3$,
\begin{align*}
&-5C + 3\ve y =2\sup_{-y \le w \le \theta t^{1/3}}[-C - \ve w - 2\theta wt^{1/3}] - 2 \sup_{0 \le w \le \theta t^{1/3}}[C + \ve w - 2\theta wt^{1/3}] - C +\ve y - 4\theta t^{1/3} y  \\
&\le A(t,y)   \\
&\le 2\sup_{-y \le w \le \theta t^{1/3}}[C + \ve w - 2\theta wt^{1/3}] - 2 \sup_{0 \le w \le \theta t^{1/3}}[-C - \ve w - 2\theta wt^{1/3}] + C - \ve y - 4\theta t^{1/3} y = 5C - 3\ve y.
\end{align*}
The proof of \eqref{eq:Ybd} is now complete, setting $\delta = 3\ve$ and $C' = 5C$. 
\end{proof}

\bibliographystyle{alpha}
\bibliography{refs} 
\end{document}